\RequirePackage{fix-cm}
\documentclass[smallextended]{svjour3}       
\smartqed  
\usepackage{graphicx}
\usepackage{amsmath,amsfonts,amssymb}
\usepackage{graphicx}
\usepackage{setspace}
\usepackage{tocloft}
\usepackage[noend]{algpseudocode}
\usepackage{algorithmicx,algorithm}
\usepackage{multirow}
\usepackage{mathrsfs}
\usepackage{booktabs}
\usepackage{color}
\usepackage{soul}
\usepackage[colorlinks,linkcolor=black]{hyperref}
\renewcommand{\theequation}{\thesection.\arabic{equation}}
\newtheorem{assumption}[theorem]{Assumption}

\numberwithin{equation}{section}
\usepackage{paralist,graphics,epsfig,graphicx,epstopdf,mathrsfs}
\usepackage{geometry}
\begin{document}

\title{Inhomogeneous spatial patterns in diffusive predator-prey system with spatial memory and predator-taxis}

\author{Yehu Lv}

\institute{Yehu Lv \at
           School of Mathematical Sciences, Beijing Normal University, Beijing 100875, China \\
           \email{mathyehul@163.com}
}

\date{}

\maketitle

\begin{abstract}
In this paper, we consider a diffusive predator-prey system with spatial memory and predator-taxis. Since in this system, the memory delay appears in the diffusion term, and the diffusion term is nonlinear, the classical normal form of Hopf bifurcation in the reaction-diffusion system with delay can't be applied to this system. Thus, in this paper, we first derive an algorithm for calculating the normal form of Hopf bifurcation in this system. Then in order to illustrate the effectiveness of our newly developed algorithm, we consider the diffusive Holling-Tanner model with spatial memory and predator-taxis. The stability and Hopf bifurcation analysis of this model are investigated, and the direction and stability of Hopf bifurcation periodic solution are also researched by using our newly developed algorithm for calculating the normal form of Hopf bifurcation. At last, we carry out some numerical simulations, two stable spatially inhomogeneous periodic solutions corresponding to the mode-1 and mode-2 Hopf bifurcations are found, which verifies our theoretical analysis results.
\end{abstract}

\section{Introduction}
\label{sec:1}

The reaction-diffusion equations based on the Fick's law have been widely used in physics, chemistry and biology \cite{lv1,lv2,lv3}. More precisely, based on the Fick's law, that is the movement flux is in the direction of negative gradient of the density distribution function, the diffusive Brusselator model with gene expression time delay \cite{lv4}, the diffusive predator-prey model in heterogeneous environment \cite{lv5,lv6}, the diffusive predator-prey model with a protection zone \cite{lv7}, the diffusive predator-prey model with prey social behavior \cite{lv8}, the diffusive predator-prey model with protection zone and predator harvesting \cite{lv9} have been studied by many scholars. Furthermore, in order to include the episodic-like spatial memory of animals, Shi et al. \cite{lv10} directed movement toward the negative gradient of the density distribution function at past time, and they proposed the following diffusive model with spatial memory
\begin{equation}\left\{\begin{aligned}
&\frac{\partial u(x,t)}{\partial t}=d_{1}\Delta u(x,t)+d_{2}\left(u(x,t)u_{x}(x,t-\tau)\right)_{x}+f\left(u(x,t)\right), & x \in \Omega,~t>0, \\
&\frac{\partial u}{\partial \mathbf{n}}(x,t)=0, & x \in \partial \Omega,~t>0, \\
&u(x,t)=u_{0}(x,t), & x \in \Omega,~-\tau \leq t\leq 0,
\end{aligned}\right.\end{equation}
where $u(x,t)$ is the population density at spatial location $x$ and time $t$, $d_{1}$ and $d_{2}$ are the Fickian diffusion coefficient and the memory-based diffusion coefficient, respectively, $\Omega \subset \mathbb{R}$ is a smooth and bounded domain, $u_{0}(x,t)$ is the initial function, $\Delta u(x,t)=\partial^{2}u(x,t)/\partial x^{2}$, $u_{x}(x,t)=\partial u(x,t)/\partial x$, $u_{x}(x,t-\tau)=\partial u(x,t-\tau)/\partial x$, $u_{xx}(x,t-\tau)=\partial^{2} u(x,t-\tau)/\partial x^{2}$, and $\mathbf{n}$ is the outward unit normal vector at the smooth boundary $\partial \Omega$. Here, the time delay $\tau>0$ represents the averaged memory period, which is usually called as the memory delay, and $f\left(u(x,t)\right)$ describes the chemical reaction or biological birth/death. Moreover, in order to further investigate the influence of memory delay on the stability of the positive constant steady
state, on the basis of model (1.1), Shi et al. \cite{lv11} studied the spatial memory diffusion model with memory and maturation delays. Furthermore, Song et al. \cite{lv12} proposed the following memory-based diffusion system subjects to homogeneous Neumann boundary condition
\begin{equation}\left\{\begin{aligned}
&\frac{\partial u(x,t)}{\partial t}=d_{11}\Delta u(x,t)+f\left(u(x,t),v(x,t)\right), & x \in \Omega,~t>0, \\
&\frac{\partial v(x,t)}{\partial t}=d_{22}\Delta v(x,t)-d_{21}\left(v(x,t)u_{x}(x,t-\tau)\right)_{x}+g\left(u(x,t),v(x,t)\right), & x \in \Omega,~t>0, \\
&u_{x}(0,t)=u_{x}(\ell\pi,t)=v_{x}(0,t)=v_{x}(\ell\pi,t)=0, & t \geq 0, \\
&u(x,t)=u_{0}(x,t), & x \in \Omega,~-\tau \leq t \leq 0, \\
&v(x,t)=v_{0}(x), & x \in \Omega,
\end{aligned}\right.\end{equation}
where $\Omega=(0,\ell\pi)$ with $\ell\in\mathbb{R}^{+}$, $u(x,t)$ and $v(x,t)$ represent the densities of prey and predator at the location $x$ and time $t$, respectively, $d_{11} \geq 0$ and $d_{22} \geq 0$ are the random diffusion coefficients, $d_{21} \geq 0$ is the memory-based diffusion coefficient, $v_{0}(x)$ is also the initial function, and $f\left(u(x,t),v(x,t)\right)$ and $g\left(u(x,t),v(x,t)\right)$ are the reaction terms. Furthermore, they derived an algorithm for calculating the normal form of Hopf bifurcation in the system (1.2). By numerical simulations, the stable spatially inhomogeneous periodic solutions and the transition from the unstable mode-2 spatially inhomogeneous periodic solution to the stable mode-1 spatially inhomogeneous periodic solution are found.

For the general predator-prey models in ecology, in addition to random diffusion of the predator and prey populations, the spatial movement of predator and prey populations also occurs, which is usually shown by the predator pursuing prey and prey escaping from predator \cite{lv13}. The pursuit and evasion between the predator and prey populations also have a strong impact on the movement pattern of the predator and prey populations \cite{lv14,lv15,lv16}. Furthermore, by noticing that such movement is not random but directed, i.e., predator moves toward the gradient direction of prey distribution, which is called prey-taxis, or prey moves opposite to the gradient of predator distribution, which is called predator-taxis. Recently, the predator-prey model with prey-taxis \cite{lv17,lv18,lv19,lv20,lv21}, the predator-prey model with indirect prey-taxis \cite{lv22,lv23,lv24}, the predator-prey model with predator-taxis \cite{lv25} and the predator-prey model with indirect predator-taxis \cite{lv26} have been researched. Especially, Wang et al. \cite{lv13} considered the following diffusive predator-prey model with both predator-taxis and prey-taxis. The corresponding mathematical model is
\begin{equation}\left\{\begin{aligned}
&\frac{\partial u(x,t)}{\partial t}=d\Delta u(x,t)+\xi \left(u(x,t)v_{x}(x,t)\right)_{x}+f(u(x,t),v(x,t)), & x \in \Omega,~t>0, \\
&\frac{\partial v(x,t)}{\partial t}=\Delta v(x,t)-\eta \left(v(x,t)u_{x}(x,t)\right)_{x}+g(u(x,t),v(x,t)), & x \in \Omega,~t>0, \\
&\frac{\partial u(x,t)}{\partial \mathbf{n}}=\frac{\partial v(x,t)}{\partial \mathbf{n}}=0, & x \in \partial\Omega,~t>0, \\
&u(x,0)=u_{0}(x)\geq 0,~v(x,0)=v_{0}(x)\geq 0, & x \in \Omega,
\end{aligned}\right.\end{equation}
where $u_{0}(x)$ is the initial function, $d$ is the rescaled diffusion coefficient for the prey population, and the diffusion coefficient of the predator population is rescaled as 1. Furthermore, the term $\xi \left(u(x,t)v_{x}(x,t)\right)_{x}$ represents the prey moves away from predator, and $\xi \geq 0$ is the intrinsic predator-taxis rate. The term $-\eta \left(v(x,t)u_{x}(x,t)\right)_{x}$ represents the predator moves towards prey, and $\eta \geq 0$ is the intrinsic prey-taxis rate.

Therefore, based on the models (1.2) and (1.3), and by considering that the spatial memory and predator-taxis, we proposed the following diffusive predator-prey model with spatial memory and predator-taxis subjects to homogeneous Neumann boundary condition
\begin{equation}\left\{\begin{aligned}
&\frac{\partial u(x,t)}{\partial t}=d_{11}\Delta u(x,t)+\xi\left(u(x,t)v_{x}(x,t)\right)_{x}+f(u(x,t),v(x,t)), & x\in(0,\ell\pi),~t>0, \\
&\frac{\partial v(x,t)}{\partial t}=d_{22}\Delta v(x,t)-d_{21}\left(v(x,t)u_{x}(x,t-\tau)\right)_{x}+g(u(x,t),v(x,t)), & x\in(0,\ell\pi),~t>0, \\
&u_{x}(0,t)=u_{x}(\ell\pi,t)=v_{x}(0,t)=v_{x}(\ell\pi,t)=0, & t \geq 0, \\
&u(x,t)=u_{0}(x,t), & x \in (0,\ell\pi),~-\tau \leq t \leq 0, \\
&v(x,t)=v_{0}(x), & x \in (0,\ell\pi).
\end{aligned}\right.\end{equation}

This paper is organized as follows. In Section 2, we derive an algorithm for calculating the normal form of Hopf bifurcation in the model (1.4). In Section 3, we obtain the normal form of Hopf bifurcation truncated to the third-order term by using our newly developed algorithm developed in Sec.2, and the mathematical expressions of the corresponding coefficients are given. In Section 4, we consider the diffusive Holling-Tanner model with spatial memory and predator-taxis. The stability and Hopf bifurcation analysis in this model are studied, and some numerical simulations are also given. In Section 5, we give a brief conclusion and discussion.

\section{Algorithm for calculating the normal form of Hopf bifurcation in the system (1.4)}
\label{sec:2}

\subsection{Characteristic equation at the positive constant steady state}

Define the real-valued Sobolev space
\begin{equation*}
X:=\left\{(u,v)^{T} \in \left(W^{2,2}(0,\ell\pi)\right)^{2}:\frac{\partial u}{\partial x}=\frac{\partial v}{\partial x}=0 \text { at } x=0, \ell\pi\right\}
\end{equation*}
with the inner product defined by
\begin{equation*}
\left[U_{1},U_{2}\right]=\int_{0}^{\ell\pi} U_{1}^{T}U_{2}~dx \text { for } U_{1}=\left(u_{1},v_{1}\right)^{T} \in X \text{ and } U_{2}=\left(u_{2},v_{2}\right)^{T} \in X,
\end{equation*}
where the symbol $T$ represents the transpose of vector, and let $\mathcal{C}:=C([-1,0];X)$ be the Banach space of continuous mappings from $[-1,0]$ to $X$ with the sup norm. It is well known that the eigenvalue problem
\begin{equation*}\left\{\begin{aligned}
&\widetilde{\varphi}^{\prime \prime}(x)=\widetilde{\lambda}\widetilde{\varphi}(x),~x \in(0,\ell\pi), \\
&\widetilde{\varphi}^{\prime}(0)=\widetilde{\varphi}^{\prime}(\ell\pi)=0
\end{aligned}\right.\end{equation*}
has eigenvalues $\widetilde{\lambda}_{n}=-n^{2}/\ell^{2}$ with corresponding normalized eigenfunctions
\begin{equation}
\beta_{n}^{(j)}=\gamma_{n}(x)e_{j},~\gamma_{n}(x)=\frac{\cos(nx/\ell)}{\left\|\cos(nx/\ell)\right\|_{L^{2}}}=\left\{\begin{aligned}
&\frac{1}{\sqrt{\ell\pi}}, & n=0, \\
&\frac{\sqrt{2}}{\sqrt{\ell\pi}}\cos\left(\frac{nx}{\ell}\right), & n \geq 1,
\end{aligned}\right.
\end{equation}
where $e_{j},~j=1,2$ is the unit coordinate vector of $\mathbb{R}^{2}$, and $n \in \mathbb{N}_{0}=\mathbb{N}\cup \left\{0\right\}$ is often called wave number, $\mathbb{N}_{0}$ is the set of all non-negative integers, $\mathbb{N}=\left\{1,2,...\right\}$ represents the set of all positive integers.

Without loss of generality, we assume that $E_{*}\left(u_{*},v_{*}\right)$ is the positive constant steady state of system (1.4). The linearized equation of (1.4) at $E_{*}\left(u_{*},v_{*}\right)$ is
\begin{equation}
\left(\begin{aligned}
&\frac{\partial u(x,t)}{\partial t} \\
&\frac{\partial v(x,t)}{\partial t}
\end{aligned}\right)=D_{1}\left(\begin{aligned}
&\Delta u(x,t) \\
&\Delta v(x,t)
\end{aligned}\right)+D_{2}\left(\begin{aligned}
&\Delta u(x,t-\tau) \\
&\Delta v(x,t-\tau)
\end{aligned}\right)+A\left(\begin{aligned}
&u(x,t) \\
&v(x,t)
\end{aligned}\right),
\end{equation}
where
\begin{equation}
D_{1}=\left(\begin{array}{cc}
d_{11} & \xi u_{*} \\
0 & d_{22}
\end{array}\right),~D_{2}=\left(\begin{array}{cc}
0 & 0 \\
-d_{21}v_{*} & 0
\end{array}\right),~A=\left(\begin{array}{cc}
a_{11} & a_{12} \\
a_{21} & a_{22}
\end{array}\right)
\end{equation}
and
\begin{equation}
a_{11}=\frac{\partial f\left(u_{*},v_{*}\right)}{\partial u},~a_{12}=\frac{\partial f\left(u_{*},v_{*}\right)}{\partial v},~a_{21}=\frac{\partial g\left(u_{*},v_{*}\right)}{\partial u},~a_{22}=\frac{\partial g\left(u_{*},v_{*}\right)}{\partial v}.
\end{equation}

Therefore, the characteristic equation of (2.2) is
\begin{equation*}
\prod_{n \in \mathbb{N}_{0}}\Gamma_{n}(\lambda)=0,
\end{equation*}
where $\Gamma_{n}(\lambda)=\det\left(\mathcal{M}_{n}(\lambda)\right)$ with
\begin{equation}
\mathcal{M}_{n}(\lambda)=\lambda I_{2}+\frac{n^{2}}{\ell^{2}}D_{1}+\frac{n^{2}}{\ell^{2}}e^{-\lambda\tau}D_{2}-A.
\end{equation}
Here, $\det(.)$ represents the determinant of a matrix, $I_{2}$ is the identity matrix of $2 \times 2$, and $D_{1}, D_{2}, A$ are defined by (2.3). Then we can obtain
\begin{equation}
\Gamma_{n}(\lambda)=\operatorname{det}\left(\mathcal{M}_{n}(\lambda)\right)=\lambda^{2}-T_{n}\lambda+\widetilde{J}_{n}(\tau)=0,
\end{equation}
where
\begin{equation}\begin{aligned}
T_{n}&=\operatorname{Tr}(A)-\operatorname{Tr}\left(D_{1}\right)\frac{n^{2}}{\ell^{2}}, \\
\widetilde{J}_{n}(\tau)&=(d_{11}d_{22}+d_{21}\xi u_{*}v_{*}e^{-\lambda\tau})\frac{n^{4}}{\ell^{4}}-\left(d_{11}a_{22}+d_{22}a_{11}-a_{21}\xi u_{*}+d_{21}v_{*}a_{12}e^{-\lambda\tau}\right)\frac{n^{2}}{\ell^{2}}+\operatorname{Det}(A)
\end{aligned}\end{equation}
with $\operatorname{Tr}(A)=a_{11}+a_{22}$, $\operatorname{Tr}\left(D_{1}\right)=d_{11}+d_{22}$ and $\operatorname{Det}(A)=a_{11}a_{22}-a_{12}a_{21}$.

\subsection{Basic assumption and equation transformation}

\begin{assumption}
Assume that at $\tau=\tau_{c}$, (2.6) has a pair of purely imaginary roots $\pm i \omega_{n_{c}}$ with $\omega_{n_{c}}>0$ for $n=n_{c} \in \mathbb{N}$ and all other eigenvalues have negative real parts. Let $\lambda(\tau)=\alpha_{1}(\tau) \pm i \alpha_{2}(\tau)$ be a pair of roots of (2.6) near $\tau=\tau_{c}$ satisfying $\alpha_{1}(\tau_{c})=0$ and $\alpha_{2}(\tau_{c})=\omega_{n_{c}}$. In addition, the corresponding transversality condition holds.
\end{assumption}

Let $\tau=\tau_{c}+\mu,~|\mu| \ll 1$ such that $\mu=0$ corresponds to the Hopf bifurcation value for system (1.4). Moreover, we shift $E_{*}(u_{*},v_{*})$ to the origin by setting
\begin{equation*}
U(x,t)=\left(U_{1}(x,t),U_{2}(x,t)\right)^{T}=\left(u(x,t),v(x,t)\right)^{T}-\left(u_{*},v_{*}\right)^{T},
\end{equation*}
and normalize the delay by rescaling the time variable $t \rightarrow t/\tau$. Furthermore, we rewrite $U(t)$ for $U(x,t)$, and $U_{t} \in \mathcal{C}$ for $U_{t}(\theta)=U(x,t+\theta),~-1 \leq \theta \leq 0$. Then the system (1.4) becomes the compact form
\begin{equation}
\frac{dU(t)}{dt}=d(\mu)\Delta(U_{t})+L(\mu)(U_{t})+F(U_{t},\mu),
\end{equation}
where for $\varphi=\left(\varphi^{(1)},\varphi^{(2)}\right)^{T} \in \mathcal{C}$, $d(\mu)\Delta$ is given by
\begin{equation*}
d(\mu)\Delta(\varphi)=d_{0}\Delta(\varphi)+F^{d}(\varphi,\mu)
\end{equation*}
with
\begin{equation}\begin{aligned}
d_{0}\Delta(\varphi)&=\tau_{c}D_{1}\Delta\varphi(0)+\tau_{c}D_{2}\Delta\varphi(-1), \\
F^{d}(\varphi, \mu)&=\left(\begin{array}{c}
\xi(\tau_{c}+\mu)(\varphi_{x}^{(1)}(0)\varphi_{x}^{(2)}(0)+\varphi^{(1)}(0)\varphi_{xx}^{(2)}(0)) \\
-d_{21}(\tau_{c}+\mu)(\varphi_{x}^{(1)}(-1)\varphi_{x}^{(2)}(0)+\varphi_{xx}^{(1)}(-1)\varphi^{(2)}(0))
\end{array}\right) \\
&+\mu\left(\begin{array}{c}
d_{11}\varphi_{xx}^{(1)}(0)+\xi u_{*}\varphi_{xx}^{(2)}(0) \\
-d_{21}v_{*}\varphi_{xx}^{(1)}(-1)+d_{22}\varphi_{xx}^{(2)}(0)
\end{array}\right).
\end{aligned}\end{equation}

Furthermore, $L(\mu): \mathcal{C} \rightarrow X$ is given by
\begin{equation}
L(\mu)(\varphi)=\left(\tau_{c}+\mu\right)A\varphi(0),
\end{equation}
and $F: \mathcal{C} \times \mathbb{R}^{2} \rightarrow X$ is given by
\begin{equation}
F(\varphi,\mu)=\left(\tau_{c}+\mu\right)\left(\begin{aligned}
f\left(\varphi^{(1)}(0)+u_{*},\varphi^{(2)}(0)+v_{*}\right) \\
g\left(\varphi^{(1)}(0)+u_{*},\varphi^{(2)}(0)+v_{*}\right)
\end{aligned}\right)-L(\mu)(\varphi).
\end{equation}

In what follows, we assume that $F(\varphi,\mu)$ is $C^{k}(k \geq 3)$ function, which is smooth with respect to $\varphi$ and $\mu$. Notice that $\mu$ is the perturbation parameter and is treated as a variable in the calculation of normal form. Moreover, by (2.10), if we denote $L_{0}(\varphi)=\tau_{c}A\varphi(0)$, then (2.8) can be rewritten as
\begin{equation}
\frac{dU(t)}{dt}=d_{0}\Delta(U_{t})+L_{0}(U_{t})+\widetilde{F}\left(U_{t},\mu\right),
\end{equation}
where the linear and nonlinear terms are separated, and
\begin{equation}
\widetilde{F}(\varphi, \mu)=\mu A \varphi(0)+F(\varphi, \mu)+F^{d}(\varphi, \mu).
\end{equation}

Thus, the linearized equation of (2.12) can be written as
\begin{equation}
\frac{dU(t)}{dt}=d_{0}\Delta(U_{t})+L_{0}(U_{t}).
\end{equation}

Moreover, the characteristic equation for the linearized equation (2.14) is
\begin{equation}
\prod_{n \in \mathbb{N}_{0}}\widetilde{\Gamma}_{n}(\lambda)=0,
\end{equation}
where $\widetilde{\Gamma}_{n}(\lambda)=\operatorname{det}\left(\widetilde{\mathcal{M}}_{n}(\lambda)\right)$ with
\begin{equation}
\widetilde{\mathcal{M}}_{n}(\lambda)=\lambda I_{2}+\tau_{c}\frac{n^{2}}{\ell^{2}}D_{1}+\tau_{c}\frac{n^{2}}{\ell^{2}}e^{-\lambda}D_{2}-\tau_{c}A.
\end{equation}

By comparing (2.16) with (2.5), we know that (2.15) has a pair of purely imaginary roots $\pm i \omega_{c}$ for $n=n_{c} \in \mathbb{N}$, and all other eigenvalues have negative real parts, where $\omega_{c}=\tau_{c}\omega_{n_{c}}$. In order to write (2.12) as an abstract ordinary differential equation in a Banach space, follows by \cite{lv27}, we can take the enlarged space
\begin{equation*}
\mathcal{BC}:=\left\{\widetilde{\psi}:[-1,0]\rightarrow X:\widetilde{\psi} \text{ is continuous on } [-1,0), \exists \lim _{\theta \rightarrow 0^{-}} \widetilde{\psi}(\theta)\in X\right\},
\end{equation*}
then the equation (2.12) is equivalent to an abstract ordinary differential equation on $\mathcal{BC}$
\begin{equation*}
\frac{dU_{t}}{dt}=\widetilde{A}U_{t}+X_{0}\widetilde{F}\left(U_{t},\mu\right).
\end{equation*}
Here, $\widetilde{A}$ is a operator from $\mathcal{C}_{0}^{1}=\{\varphi \in \mathcal{C}:\dot{\varphi} \in \mathcal{C},\varphi(0) \in \operatorname{dom}(\Delta)\}$ to $\mathcal{BC}$, which is defined by
\begin{equation*}
\widetilde{A}\varphi=\dot{\varphi}+X_{0}\left(\tau_{c}D_{1}\Delta\varphi(0)+\tau_{c}D_{2}\Delta\varphi(-1)+L_{0}(\varphi)-\dot{\varphi}(0)\right),
\end{equation*}
and $X_{0}=X_{0}(\theta)$ is given by
\begin{equation*}
X_{0}(\theta)=\left\{\begin{aligned}
& 0, & -1 \leq \theta<0, \\
& 1, & \theta=0.
\end{aligned}\right.
\end{equation*}

In the following, the method given in \cite{lv27} is used to complete the decomposition of $\mathcal{BC}$. Let $C:=C\left([-1,0],\mathbb{R}^{2}\right),~C^{*}:=C\left([0,1],\mathbb{R}^{2*}\right)$, where $\mathbb{R}^{2*}$ is the two-dimensional space of row vectors, and define the adjoint bilinear form on $C^{*}\times C$ as follows
\begin{equation*}
\langle\Psi(s),\Phi(\theta)\rangle=\Psi(0)\Phi(0)-\int_{-1}^{0}\int_{0}^{\theta}\Psi(\xi-\theta)dM_{n}(\theta)\Phi(\xi)d\xi
\end{equation*}
for $\Psi \in C^{*},~\Phi \in C$ and $\xi \in [-1,0]$, where $M_{n}(\theta)$ is a bounded variation function from $[-1,0]$ to $\mathbb{R}^{2 \times 2}$, i.e., $M_{n}(\theta)\in BV\left([-1,0];\mathbb{R}^{2 \times 2}\right)$, such that for $\Phi(\theta)\in\mathcal{C}$, one has
\begin{equation*}
-\tau_{c}\frac{n^{2}}{\ell^{2}}D_{1}\Phi(0)-\tau_{c}\frac{n^{2}}{\ell^{2}}D_{2}\Phi(-1)+L_{0}(\Phi(\theta))=\int_{-1}^{0}dM_{n}(\theta)\Phi(\theta).
\end{equation*}

By choosing
\begin{equation*}
\Phi(\theta)=(\phi(\theta),\overline{\phi}(\theta)),~\Psi(s)=\operatorname{col}\left(\psi^{T}(s),\overline{\psi}^{T}(s)\right),
\end{equation*}
where the $\operatorname{col}(.)$ represents the column vector, $\phi(\theta)=\operatorname{col}\left(\phi_{1}(\theta),\phi_{2}(\theta)\right)=\phi e^{i \omega_{c} \theta}\in \mathbb{C}^{2}$ with $\phi=\operatorname{col}\left(\phi_{1},\phi_{2}\right)$ is the eigenvector of (2.14) associated with the eigenvalue $i \omega_{c}$, and $\psi(s)=\operatorname{col}\left(\psi_{1}(s),\psi_{2}(s)\right)=\psi e^{-i\omega_{c}s}\in \mathbb{C}^{2}$ with $\psi=\operatorname{col}\left(\psi_{1},\psi_{2}\right)$ is the corresponding adjoint eigenvector such that
\begin{equation*}
\left\langle\Psi(s), \Phi(\theta)\right\rangle=I_{2},
\end{equation*}
where
\begin{equation*}
\phi=\left(\begin{array}{c}
1 \\
\frac{a_{11}-i\omega_{n_{c}}-d_{11}(n_{c}^{2}/\ell^{2})}{\xi u_{*}(n_{c}^{2}/\ell^{2})-a_{12}}
\end{array}\right),~\psi=\eta\left(\begin{array}{c}
1 \\
\frac{a_{12}-\xi u_{*}(n_{c}^{2}/\ell^{2})}{i\omega_{n_{c}}+d_{22}(n_{c}^{2}/\ell^{2})-a_{22}}
\end{array}\right)
\end{equation*}
and
\begin{equation*}
\eta=\frac{i\omega_{n_{c}}+\left(n_{c}/\ell\right)^{2}d_{22}-a_{22}}{2i\omega_{n_{c}}+\left(n_{c}/\ell\right)^{2}d_{11}-a_{11}+\left(n_{c}/\ell\right)^{2}d_{22}-a_{22}+\tau_{c}a_{12} d_{21}v_{*}\left(n_{c}/\ell\right)^{2}e^{-i\omega_{c}}}.
\end{equation*}

According to \cite{lv27}, the phase space $\mathcal{C}$ can be decomposed as
\begin{equation*}
\mathcal{C}=\mathcal{P} \oplus \mathcal{Q},~\mathcal{P}=\operatorname{Im} \pi,~\mathcal{Q}=\operatorname{Ker} \pi,
\end{equation*}
where for $\widetilde{\phi} \in \mathcal{C}$, the projection $\pi: \mathcal{C} \rightarrow \mathcal{P}$ is defined by
\begin{equation}
\pi(\widetilde{\phi})=\left(\Phi\left\langle\Psi,\left(\begin{aligned}
&\left[\widetilde{\phi}(\cdot),\beta_{n_{c}}^{(1)}\right] \\
&\left[\widetilde{\phi}(\cdot),\beta_{n_{c}}^{(2)}\right]
\end{aligned}\right)\right\rangle\right)^{T} \beta_{n_{c}}.
\end{equation}

Therefore, by following the method in \cite{lv27}, $\mathcal{BC}$ can be divided into a direct sum of center subspace and its complementary space, that is
\begin{equation}
\mathcal{BC}=\mathcal{P} \oplus \operatorname{ker} \pi,
\end{equation}
where $\operatorname{dim}\mathcal{P}=2$. It is easy to see that the projection $\pi$ which is defined by (2.17), is extended to a continuous projection (which is still denoted by $\pi$), that is, $\pi:\mathcal{BC}\mapsto \mathcal{P}$. In particular, for $\alpha \in \mathcal{C}$, we have
\begin{equation}
\pi\left(X_{0}(\theta)\alpha\right)=\left(\Phi(\theta)\Psi(0)\left(\begin{aligned}
&\left[\alpha,\beta_{n_{c}}^{(1)}\right] \\
&\left[\alpha,\beta_{n_{c}}^{(2)}\right]
\end{aligned}\right)\right)^{T} \beta_{n_{c}}.
\end{equation}

By combining with (2.18) and (2.19), $U_{t}(\theta)$ can be decomposed as
\begin{equation}\begin{aligned}
U_{t}(\theta)&=\left(\Phi(\theta)\left(\begin{array}{c}
z_{1} \\
z_{2}
\end{array}\right)\right)^{T}\left(\begin{array}{c}
\beta_{n_{c}}^{(1)} \\
\beta_{n_{c}}^{(2)}
\end{array}\right)+w=\left(z_{1}\phi e^{i\omega_{c}\theta}+z_{2}\overline{\phi} e^{-i\omega_{c}\theta}\right)\gamma_{n_{c}}(x)+w \\
&=\left(\phi(\theta)~~\overline{\phi}(\theta)\right)\left(\begin{array}{c}
z_{1}\gamma_{n_{c}}(x) \\
z_{2}\gamma_{n_{c}}(x)
\end{array}\right)+\left(\begin{array}{c}
w_{1} \\
w_{2}
\end{array}\right),
\end{aligned}\end{equation}
where $w=\operatorname{col}\left(w_{1},w_{2}\right)$ and
\begin{equation*}\begin{aligned}
&\left(\begin{array}{c}
z_{1} \\
z_{2}
\end{array}\right)=\left\langle\Psi(0),\left(\begin{aligned}
&\left[U_{t}(\theta),\beta_{n_{c}}^{(1)}\right] \\
&\left[U_{t}(\theta),\beta_{n_{c}}^{(2)}\right]
\end{aligned}\right)\right\rangle.
\end{aligned}\end{equation*}

If we assume that
\begin{equation*}
\Phi(\theta)=\left(\phi(\theta),\overline{\phi}(\theta)\right),~z_{x}=\left(z_{1}\gamma_{n_{c}}(x),z_{2}\gamma_{n_{c}}(x)\right)^{T},
\end{equation*}
then (2.20) can be rewritten as
\begin{equation}
U_{t}(\theta)=\Phi(\theta)z_{x}+w \text{ with } w \in \mathcal{C}_{0}^{1} \cap \operatorname{Ker} \pi:=\mathcal{Q}^{1}.
\end{equation}

Then by combining with (2.21), the system (2.12) is decomposed as a system of abstract ordinary differential equations (ODEs) on $\mathbb{R}^{2} \times \text{Ker}~\pi$, with finite and infinite dimensional variables are separated in the linear term. That is
\begin{equation}
\left\{\begin{aligned}
&\dot{z}=Bz+\Psi(0)\left(\begin{aligned}
&\left[\widetilde{F}\left(\Phi(\theta)z_{x}+w,\mu\right),\beta_{n_{c}}^{(1)}\right] \\
&\left[\widetilde{F}\left(\Phi(\theta)z_{x}+w,\mu\right),\beta_{n_{c}}^{(2)}\right]
\end{aligned}\right), \\
&\dot{w}=A_{\mathcal{Q}^{1}}w+(I-\pi)X_{0}(\theta)\widetilde{F}\left(\Phi(\theta)z_{x}+w,\mu\right),
\end{aligned}\right.
\end{equation}
where $I$ is the identity matrix, $z=\left(z_{1},z_{2}\right)^{T}$, $B=\text{diag}\left\{i\omega_{c},-i\omega_{c}\right\}$ is the diagonal matrix, and $A_{\mathcal{Q}^{1}}:\mathcal{Q}^{1} \rightarrow \operatorname{Ker}\pi$ is defined by
\begin{equation*}
A_{\mathcal{Q}^{1}}w=\dot{w}+X_{0}(\theta)\left(\tau_{c}D_{1}\Delta w(0)+\tau_{c}D_{2}\Delta w(-1)+L_{0}(w)-\dot{w}(0)\right).
\end{equation*}

Consider the formal Taylor expansion
\begin{equation*}
\widetilde{F}(\varphi,\mu)=\sum_{j \geq 2} \frac{1}{j!}\widetilde{F}_{j}(\varphi,\mu),~F(\varphi,\mu)=\sum_{j \geq 2}\frac{1}{j!} F_{j}(\varphi,\mu),~F^{d}(\varphi,\mu)=\sum_{j \geq 2}\frac{1}{j!}F_{j}^{d}(\varphi,\mu).
\end{equation*}
From (2.13), we have
\begin{equation}
\widetilde{F}_{2}(\varphi,\mu)=2\mu A\varphi(0)+F_{2}(\varphi,\mu)+F_{2}^{d}(\varphi,\mu)
\end{equation}
and
\begin{equation}
\widetilde{F}_{j}(\varphi,\mu)=F_{j}(\varphi,\mu)+F_{j}^{d}(\varphi,\mu),~j=3,4,\cdots.
\end{equation}

By combining with (2.19), the system (2.22) can be rewritten as
\begin{equation*}
\left\{\begin{aligned}
&\dot{z}=Bz+\sum_{j \geq 2}\frac{1}{j!}f_{j}^{1}(z,w,\mu), \\
&\dot{w}=A_{\mathcal{Q}^{1}}w+\sum_{j \geq 2}\frac{1}{j!}f_{j}^{2}(z,w,\mu),
\end{aligned}\right.
\end{equation*}
where
\begin{equation}\begin{aligned}
&f_{j}^{1}(z,w,\mu)=\Psi(0)\left(\begin{aligned}
&\left[\widetilde{F}_{j}\left(\Phi(\theta)z_{x}+w,\mu\right),\beta_{n_{c}}^{(1)}\right] \\
&\left[\widetilde{F}_{j}\left(\Phi(\theta)z_{x}+w,\mu\right),\beta_{n_{c}}^{(2)}\right]
\end{aligned}\right), \\
&f_{j}^{2}(z,w,\mu)=(I-\pi)X_{0}(\theta)\widetilde{F}_{j}\left(\Phi(\theta)z_{x}+w,\mu\right).
\end{aligned}\end{equation}

In terms of the normal form theory of partial functional differential equations \cite{lv27}, after a recursive transformation of variables of the form
\begin{equation}
(z,w)=(\widetilde{z},\widetilde{w})+\frac{1}{j!}\left(U_{j}^{1}(\widetilde{z},\mu),U_{j}^{2}(\widetilde{z},\mu)\right),~j \geq 2,
\end{equation}
where $z, \widetilde{z} \in \mathbb{R}^{2}$, $w, \widetilde{w} \in \mathcal{Q}^{1}$ and $U_{j}^{1}:\mathbb{R}^{3} \rightarrow \mathbb{R}^{2}$, $U_{j}^{2}: \mathbb{R}^{3} \rightarrow \mathcal{Q}^{1}$ are homogeneous polynomials of degree $j$ in $\widetilde{z}$ and $\mu$, a locally center manifold for (2.12) satisfies $w=0$ and the flow on it is given by the two-dimensional ODEs
\begin{equation*}
\dot{z}=B z+\sum_{j \geq 2} \frac{1}{j!} g_{j}^{1}(z,0,\mu),
\end{equation*}
which is the normal form as in the usual sense for ODEs.

By following \cite{lv27} and \cite{lv28}, we have
\begin{equation}
g_{2}^{1}(z,0,\mu)=\operatorname{Proj}_{\operatorname{Ker}\left(M_{2}^{1}\right)}f_{2}^{1}(z,0,\mu)
\end{equation}
and
\begin{equation}
g_{3}^{1}(z,0,\mu)=\operatorname{Proj}_{\operatorname{Ker}\left(M_{3}^{1}\right)}\widetilde{f}_{3}^{1}(z,0,\mu)=\operatorname{Proj}_{S} \widetilde{f}_{3}^{1}(z,0,0)+O\left(\mu^{2}|z|\right),
\end{equation}
where $\operatorname{Proj}_{p}(q)$ represents the projection of $q$ on $p$, and $\widetilde{f}_{3}^{1}(z,0,\mu)$ is vector and its element is the cubic polynomial of $(z, \mu)$ after the variable transformation of (2.26), and it is determined by (2.38),
\begin{equation}\begin{aligned}
&\operatorname{Ker}\left(M_{2}^{1}\right)=\operatorname{Span}\left\{\left(\begin{array}{c}
\mu z_{1} \\
0
\end{array}\right),\left(\begin{array}{c}
0 \\
\mu z_{2}
\end{array}\right)\right\}, \\
&\operatorname{Ker}\left(M_{3}^{1}\right)=\operatorname{Span}\left\{\left(\begin{array}{c}
z_{1}^{2}z_{2} \\
0
\end{array}\right),\left(\begin{array}{c}
\mu^{2}z_{1} \\
0
\end{array}\right),\left(\begin{array}{c}
0 \\
z_{1}z_{2}^{2}
\end{array}\right),\left(\begin{array}{c}
0 \\
\mu^{2}z_{2}
\end{array}\right)\right\},
\end{aligned}\end{equation}
and
\begin{equation}
S=\operatorname{Span}\left\{\left(\begin{array}{c}
z_{1}^{2}z_{2} \\
0
\end{array}\right),\left(\begin{array}{c}
0 \\
z_{1}z_{2}^{2}
\end{array}\right)\right\}.
\end{equation}
In the following, for notational convenience, we let
\begin{equation*}
\mathcal{H}\left(\alpha z_{1}^{q_{1}}z_{2}^{q_{2}}\mu\right)=\left(\begin{array}{c}
\alpha z_{1}^{q_{1}}z_{2}^{q_{2}}\mu \\
\overline{\alpha}z_{1}^{q_{2}}z_{2}^{q_{1}}\mu
\end{array}\right),~\alpha \in \mathbb{C}.
\end{equation*}
We then calculate $g_{j}^{1}(z,0,\mu),~j=2,3$ step by step.

\subsection{Algorithm for calculating the normal form of Hopf bifurcation}

\subsubsection{Calculation of $g_{2}^{1}(z,0,\mu)$}

From the second mathematical expression in (2.9), we have
\begin{equation}
F_{2}^{d}(\varphi,\mu)=F_{20}^{d}(\varphi)+\mu F_{21}^{d}(\varphi)
\end{equation}
and
\begin{equation}
F_{3}^{d}(\varphi,\mu)=\mu F_{31}^{d}(\varphi),~F_{j}^{d}(\varphi,\mu)=(0,0)^{T},~j=4,5,\cdots,
\end{equation}
where
\begin{equation}
\left\{\begin{aligned}
&F_{20}^{d}(\varphi)=2\left(\begin{array}{c}
\xi\tau_{c}\left(\varphi_{x}^{(1)}(0)\varphi_{x}^{(2)}(0)+\varphi^{(1)}(0)\varphi_{xx}^{(2)}(0)\right) \\
-d_{21}\tau_{c}\left(\varphi_{x}^{(1)}(-1)\varphi_{x}^{(2)}(0)+\varphi_{xx}^{(1)}(-1)\varphi^{(2)}(0)\right)
\end{array}\right), \\
&F_{21}^{d}(\varphi)=2D_{1}\Delta\varphi(0)+2D_{2}\Delta\varphi(-1), \\
&F_{31}^{d}(\varphi)=6\left(\begin{array}{c}
\xi\mu\left(\varphi_{x}^{(1)}(0)\varphi_{x}^{(2)}(0)+\varphi^{(1)}(0)\varphi_{xx}^{(2)}(0)\right) \\
-d_{21}\mu\left(\varphi_{x}^{(1)}(-1)\varphi_{x}^{(2)}(0)+\varphi_{xx}^{(1)}(-1)\varphi^{(2)}(0)\right)
\end{array}\right).
\end{aligned}\right.
\end{equation}

Furthermore, it is easy to verify that
\begin{equation}\begin{aligned}
&\left(\begin{aligned}
&\left[2\mu A(\Phi(0)z_{x}),\beta_{n_{c}}^{(1)}\right] \\
&\left[2\mu A(\Phi(0)z_{x}),\beta_{n_{c}}^{(2)}\right]
\end{aligned}\right)=2\mu A\left(\Phi(0)\left(\begin{array}{c}
z_{1} \\
z_{2}
\end{array}\right)\right), \\
&\left(\begin{aligned}
&\left[\mu F_{21}^{d}\left(\Phi(\theta)z_{x}\right),\beta_{n_{c}}^{(1)}\right] \\
&\left[\mu F_{21}^{d}\left(\Phi(\theta)z_{x}\right),\beta_{n_{c}}^{(2)}\right]
\end{aligned}\right)=-\frac{2n_{c}^{2}}{\ell^{2}}\mu\left(D_{1}\left(\Phi(0)\left(\begin{array}{c}
z_{1} \\
z_{2}
\end{array}\right)\right)+D_{2}\left(\Phi(-1)\left(\begin{array}{c}
z_{1} \\
z_{2}
\end{array}\right)\right)\right).
\end{aligned}\end{equation}

From (2.11), we have $F_{2}\left(\Phi(\theta)z_{x},\mu\right)=F_{2}\left(\Phi(\theta)z_{x},0\right)$ for all $\mu \in \mathbb{R}$. It follows from the first mathematical expression in (2.25) that
\begin{equation*}
f_{2}^{1}(z, 0, \mu)=\Psi(0)\left(\begin{aligned}
&\left[\widetilde{F}_{2}\left(\Phi(\theta)z_{x},\mu\right),\beta_{n_{c}}^{(1)}\right] \\
&\left[\widetilde{F}_{2}\left(\Phi(\theta)z_{x},\mu\right),\beta_{n_{c}}^{(2)}\right]
\end{aligned}\right).
\end{equation*}

This, together with (2.23), (2.27), (2.29), (2.31), (2.32), (2.33) and (2.34), yields to
\begin{equation}
g_{2}^{1}(z,0,\mu)=\operatorname{Proj}_{\operatorname{Ker}\left(M_{2}^{1}\right)}f_{2}^{1}(z,0,\mu)=\mathcal{H}\left(B_{1}\mu z_{1}\right),
\end{equation}
where
\begin{equation}
B_{1}=2\psi^{T}(0)\left(A\phi(0)-\frac{n_{c}^{2}}{\ell^{2}}\left(D_{1}\phi(0)+D_{2}\phi(-1)\right)\right).
\end{equation}

\subsubsection{Calculation of $g_{3}^{1}(z,0,\mu)$}

In this subsection, we calculate the third term $g_{3}^{1}(z,0,\mu)$ in terms of (2.28). Denote
\begin{equation}\begin{aligned}
&f_{2}^{(1,1)}(z,w,0)=\Psi(0)\left(\begin{aligned}
&\left[F_{2}\left(\Phi(\theta)z_{x}+w,0\right),\beta_{n_{c}}^{(1)}\right] \\
&\left[F_{2}\left(\Phi(\theta)z_{x}+w,0\right),\beta_{n_{c}}^{(2)}\right]
\end{aligned}\right), \\
&f_{2}^{(1,2)}(z,w,0)=\Psi(0)\left(\begin{aligned}
&\left[F_{2}^{d}\left(\Phi(\theta)z_{x}+w,0\right),\beta_{n_{c}}^{(1)}\right] \\
&\left[F_{2}^{d}\left(\Phi(\theta)z_{x}+w,0\right),\beta_{n_{c}}^{(2)}\right]
\end{aligned}\right).
\end{aligned}\end{equation}

It follows from (2.35) that $g_{2}^{1}(z,0,0)=(0,0)^{T}$. Then $\widetilde{f}_{3}^{1}(z,0,0)$ is determined by
\begin{equation}\begin{aligned}
\widetilde{f}_{3}^{1}(z,0,0)&=f_{3}^{1}(z,0,0)+\frac{3}{2}\left((D_{z}f_{2}^{1}(z,0,0))U_{2}^{1}(z,0)+(D_{w}f_{2}^{(1,1)}(z,0,0))U_{2}^{2}(z,0)\right. \\
&~~~~~~~~~~~~~~~~~\left.+(D_{w,w_{x},w_{xx}}f_{2}^{(1,2)}(z,0,0))U_{2}^{(2,d)}(z,0)(\theta)\right),
\end{aligned}\end{equation}
where $f_{2}^{1}(z,0,0)=f_{2}^{(1,1)}(z,0,0)+f_{2}^{(1,2)}(z,0,0)$,
\begin{equation}\begin{aligned}
&D_{w,w_{x},w_{xx}}f_{2}^{(1,2)}(z,0,0)=\left(D_{w}f_{2}^{(1,2)}(z,0,0),D_{w_{x}}f_{2}^{(1,2)}(z,0,0),D_{w_{xx}}f_{2}^{(1,2)}(z,0,0)\right), \\
&U_{2}^{1}(z,0)=\left(M_{2}^{1}\right)^{-1}\operatorname{Proj}_{\operatorname{Im}\left(M_{2}^{1}\right)}f_{2}^{1}(z,0,0),~ U_{2}^{2}(z,0)(\theta)=\left(M_{2}^{2}\right)^{-1}f_{2}^{2}(z,0,0),
\end{aligned}\end{equation}
and
\begin{equation}
U_{2}^{(2,d)}(z,0)(\theta)=\operatorname{col}\left(U_{2}^{2}(z,0)(\theta),U_{2 x}^{2}(z,0)(\theta),U_{2xx}^{2}(z,0)(\theta)\right).
\end{equation}

We calculate $\operatorname{Proj}_{S}\widetilde{f}_{3}^{1}(z,0,0)$ in the following four steps.
\par~~~

\noindent {\bf{Step 1: Calculation of $\operatorname{Proj}_{S} f_{3}^{1}(z,0,0)$}}

Writing $F_{3}\left(\Phi(\theta)z_{x},0\right)$ as follows
\begin{equation}\begin{aligned}
&F_{3}\left(\Phi(\theta)z_{x},0\right)=\sum_{q_{1}+q_{2}=3}A_{q_{1}q_{2}}z_{1}^{q_{1}}z_{2}^{q_{2}}\gamma_{n_{c}}^{3}(x),
\end{aligned}\end{equation}
where $A_{q_{1}q_{2}}=\overline{A_{q_{2}q_{1}}}$ with $q_{1}, q_{2} \in \mathbb{N}_{0}$. From (2.24) and (2.32), we have $\widetilde{F}_{3}\left(\Phi(\theta) z_{x}, 0\right)=F_{3}\left(\Phi(\theta) z_{x}, 0\right)$, and thus
\begin{equation*}
\operatorname{Proj}_{S}f_{3}^{1}(z,0,0)=\mathcal{H}\left(B_{21}z_{1}^{2}z_{2}\right),
\end{equation*}
where
\begin{equation}
B_{21}=\frac{3}{2\ell\pi}\psi^{T}A_{21}.
\end{equation}
\par~~~

\noindent {\bf{Step 2: Calculation of $\operatorname{Proj}_{S}\left((D_{z}f_{2}^{1}(z,0,0))U_{2}^{1}(z,0)\right)$}}

Form (2.23) and (2.31), we have
\begin{equation}
\widetilde{F}_{2}\left(\Phi(\theta)z_{x},0\right)=F_{2}\left(\Phi(\theta)z_{x},0\right)+F_{20}^{d}\left(\Phi(\theta) z_{x}\right).
\end{equation}
By (2.11), we write
\begin{equation}\begin{aligned}
&F_{2}\left(\Phi(\theta)z_{x}+w,\mu\right)=F_{2}\left(\Phi(\theta)z_{x}+w,0\right) \\
&=\sum_{q_{1}+q_{2}=2}A_{q_{1}q_{2}}z_{1}^{q_{1}}z_{2}^{q_{2}}\gamma_{n_{c}}^{2}(x)+\mathcal{S}_{2}\left(\Phi(\theta)z_{x},w\right)+O\left(|w|^{2}\right),
\end{aligned}\end{equation}
where $\mathcal{S}_{2}\left(\Phi(\theta)z_{x},w\right)$ is the product term of $\Phi(\theta)z_{x}$ and $w$.

By (2.31) and (2.33), we write
\begin{equation}
F_{2}^{d}\left(\Phi(\theta)z_{x},0\right)=F_{20}^{d}\left(\Phi(\theta)z_{x}\right)=\frac{n_{c}^{2}}{\ell^{2}}\sum_{q_{1}+q_{2}=2}A_{q_{1}q_{2}}^{d} z_{1}^{q_{1}}z_{2}^{q_{2}}\left(\xi_{n_{c}}^{2}(x)-\gamma_{n_{c}}^{2}(x)\right),
\end{equation}
where $\xi_{n_{c}}(x)=(\sqrt{2}/\sqrt{\ell \pi})\sin \left((n_{c}/\ell)x\right)$, and
\begin{equation}\begin{aligned}
&\left\{\begin{aligned}
&A_{20}^{d}=\left(\begin{array}{c}
2\xi \tau_{c}\phi_{1}(0)\phi_{2}(0) \\
-2d_{21}\tau_{c}\phi_{1}(-1)\phi_{2}(0)
\end{array}\right)=\overline{A_{02}^{d}}, \\
&A_{11}^{d}=\left(\begin{array}{c}
4\xi \tau_{c}\operatorname{Re}\left\{\phi_{1}(0)\overline{\phi_{2}}(0)\right\} \\
-4d_{21}\tau_{c}\operatorname{Re}\left\{\phi_{1}(-1)\overline{\phi_{2}}(0)\right\}
\end{array}\right).
\end{aligned}\right.
\end{aligned}\end{equation}

From (2.1), it is easy to verify that
\begin{equation*}
\int_{0}^{\ell\pi}\gamma_{n_{c}}^{3}(x)dx=\int_{0}^{\ell\pi}\xi_{n_{2}}^{2}(x)\gamma_{n_{c}}(x)dx=0.
\end{equation*}
Then, from (2.43), (2.44) and (2.45), we have
\begin{equation}
f_{2}^{1}(z,0,0)=\Psi(0)\left(\begin{aligned}
&\left[\widetilde{F}_{2}\left(\Phi(\theta)z_{x},0\right),\beta_{n_{c}}^{(1)}\right] \\
&\left[\widetilde{F}_{2}\left(\Phi(\theta)z_{x},0\right),\beta_{n_{c}}^{(2)}\right]
\end{aligned}\right)=\left(\begin{array}{c}
0 \\
0
\end{array}\right).
\end{equation}
Thus, by combining with (2.30) and (2.47), we have
\begin{equation*}
\operatorname{Proj}_{S}\left((D_{z}f_{2}^{1}(z,0,0))U_{2}^{1}(z,0)\right)=\left(\begin{array}{c}
0 \\
0
\end{array}\right).
\end{equation*}
\par~~~

\noindent {\bf{Step 3: Calculation of $\operatorname{Proj}_{S}\left((D_{w}f_{2}^{(1,1)}(z,0,0))U_{2}^{2}(z,0)(\theta)\right)$}}

Let
\begin{equation}
U_{2}^{2}(z,0)(\theta)\triangleq h(\theta,z)=\sum_{n \in \mathbb{N}_{0}}h_{n}(\theta,z)\gamma_{n}(x),
\end{equation}
where $h_{n}(\theta,z)=\sum_{q_{1}+q_{2}=2}h_{n,q_{1}q_{2}}(\theta)z_{1}^{q_{1}}z_{2}^{q_{2}}$. Then, we have
\begin{equation*}\begin{aligned}
&\left(\begin{aligned}
\left[\mathcal{S}_{2}\left(\Phi(\theta)z_{x},\sum_{n \in \mathbb{N}_{0}} h_{n}(\theta,z)\gamma_{n}(x)\right),\beta_{n_{c}}^{(1)}\right] \\
\left[\mathcal{S}_{2}\left(\Phi(\theta)z_{x},\sum_{n \in \mathbb{N}_{0}} h_{n}(\theta,z)\gamma_{n}(x)\right),\beta_{n_{c}}^{(2)}\right]
\end{aligned}\right) \\
&=\sum_{n \in \mathbb{N}_{0}} b_{n}\left(\mathcal{S}_{2}\left(\phi(\theta)z_{1},h_{n}(\theta,z)\right)+\mathcal{S}_{2}\left(\overline{\phi}(\theta)z_{2},h_{n}(\theta,z)\right)\right),
\end{aligned}\end{equation*}
where
\begin{equation}
b_{n}=\int_{0}^{\ell\pi}\gamma_{n_{c}}^{2}(x)\gamma_{n}(x)dx=\left\{\begin{aligned}
& \frac{1}{\sqrt{\ell\pi}}, & n=0, \\
& \frac{1}{\sqrt{2\ell\pi}}, & n=2n_{c}, \\
& 0, & \text { otherwise }.
\end{aligned}\right.
\end{equation}
Hence, we have
\begin{equation*}\begin{aligned}
&(D_{w}f_{2}^{(1,1)}(z,0,0))U_{2}^{2}(z,0)(\theta) \\
&=\Psi(0)\left(\sum_{n=0,2n_{c}}b_{n}\left(\mathcal{S}_{2}\left(\phi(\theta)z_{1},h_{n}(\theta,z)\right)+\mathcal{S}_{2}\left(\overline{\phi}(\theta)z_{2},h_{n}(\theta,z)\right)\right)\right),
\end{aligned}\end{equation*}
and
\begin{equation*}
\operatorname{Proj}_{S}\left((D_{w}f_{2}^{(1,1)}(z,0,0))U_{2}^{2}(z,0)(\theta)\right)=\mathcal{H}\left(B_{22}z_{1}^{2}z_{2}\right),
\end{equation*}
where
\begin{equation}\begin{aligned}
B_{22}&=\frac{1}{\sqrt{\ell\pi}}\psi^{T}\left(\mathcal{S}_{2}\left(\phi(\theta),h_{0,11}(\theta)\right)+\mathcal{S}_{2}\left(\overline{\phi}(\theta), h_{0,20}(\theta)\right)\right) \\
&+\frac{1}{\sqrt{2\ell\pi}}\psi^{T}\left(\mathcal{S}_{2}\left(\phi(\theta),h_{2n_{c},11}(\theta)\right)+\mathcal{S}_{2}\left(\overline{\phi}(\theta),h_{2n_{c},20}(\theta)\right)\right).
\end{aligned}\end{equation}
\par~~~

\noindent {\bf{Step 4: Calculation of $\operatorname{Proj}_{S}\left((D_{w,w_{x},w_{xx}}f_{2}^{(1,2)}(z,0,0))U_{2}^{(2,d)}(z,0)(\theta)\right)$}}

Denote $\varphi(\theta)=\left(\varphi^{(1)}(\theta),\varphi^{(2)}(\theta)\right)^{T}=\Phi(\theta)z_{x}$,
\begin{equation*}\begin{aligned}
&F_{2}^{d}\left(\varphi(\theta),w,w_{x},w_{xx}\right)=F_{2}^{d}(\varphi(\theta)+w,0)=F_{20}^{d}(\varphi(\theta)+w) \\
&=2\left(\begin{array}{c}
\xi\tau_{c}\left((\varphi_{x}^{(1)}(0)+w_{x}^{(1)}(0))(\varphi_{x}^{(2)}(0)+w_{x}^{(2)}(0))+(\varphi^{(1)}(0)+w^{(1)}(0))(\varphi_{xx}^{(2)}(0)+w_{xx}^{(2)}(0))\right) \\
-d_{21}\tau_{c}\left((\varphi_{x}^{(1)}(-1)+w_{x}^{(1)}(-1))(\varphi_{x}^{(2)}(0)+w_{x}^{(2)}(0))+(\varphi_{xx}^{(1)}(-1)+w_{xx}^{(1)}(-1))(\varphi^{(2)}(0)+w^{(2)}(0))\right)
\end{array}\right)
\end{aligned}\end{equation*}
and
\begin{equation*}\begin{aligned}
&\widetilde{\mathcal{S}}_{2}^{(d,1)}(\varphi(\theta),w)=2\left(\begin{array}{c}
\xi\tau_{c}\varphi_{xx}^{(2)}(0)w^{(1)}(0) \\
-d_{21}\tau_{c}\varphi_{xx}^{(1)}(-1)w^{(2)}(0)
\end{array}\right), \\
&\widetilde{\mathcal{S}}_{2}^{(d,2)}(\varphi(\theta),w_{x})=2\left(\begin{array}{c}
\xi\tau_{c}(\varphi_{x}^{(2)}(0)w_{x}^{(1)}(0)+\varphi_{x}^{(1)}(0)w_{x}^{2}(0)) \\
-d_{21}\tau_{c}(\varphi_{x}^{(2)}(0)w_{x}^{(1)}(-1)+\varphi_{x}^{(1)}(-1)w_{x}^{(2)}(0))
\end{array}\right), \\
&\widetilde{\mathcal{S}}_{2}^{(d,3)}(\varphi(\theta),w_{xx})=2\left(\begin{array}{c}
\xi \tau_{c}\varphi^{(1)}(0)w_{xx}^{(2)}(0) \\
-d_{21}\tau_{c}\varphi^{(2)}(0)w_{xx}^{(1)}(-1)
\end{array}\right).
\end{aligned}\end{equation*}

By combining with (2.1) and (2.48), we have
\begin{equation*}
\left\{\begin{aligned}
&U_{2x}^{2}(z,0)(\theta)=h_{x}(\theta,z)=-\frac{n}{\ell}\sum_{n \in \mathbb{N}_{0}}h_{n}(\theta,z)\xi_{n}(x), \\
&U_{2xx}^{2}(z,0)(\theta)=h_{xx}(\theta,z)=-\frac{n^{2}}{\ell^{2}}\sum_{n \in \mathbb{N}_{0}}h_{n}(\theta,z)\gamma_{n}(x).
\end{aligned}\right.
\end{equation*}
Then we have
\begin{equation*}\begin{aligned}
&\left(D_{w,w_{x},w_{xx}}F_{2}^{d}\left(\varphi(\theta),w,w_{x},w_{xx}\right)\right)U_{2}^{(2,d)}(z,0)(\theta) \\
&=\widetilde{\mathcal{S}}_{2}^{(d,1)}(\varphi(\theta),h(\theta,z))+\widetilde{\mathcal{S}}_{2}^{(d,2)}(\varphi(\theta),h_{x}(\theta,z))+\widetilde{\mathcal{S}}_{2}^{(d,3)}(\varphi(\theta),h_{xx}(\theta,z))
\end{aligned}\end{equation*}
and
\begin{equation*}\begin{aligned}
&\left(\begin{aligned}
&\left[\widetilde{S}_{2}^{(d,1)}(\varphi(\theta),h(\theta,z)),\beta_{n_{c}}^{(1)}\right] \\
&\left[\widetilde{S}_{2}^{(d,1)}(\varphi(\theta),h(\theta,z)),\beta_{n_{c}}^{(2)}\right] \\
\end{aligned}\right) \\
&=-\left(n_{c}/\ell\right)^{2}\sum_{n \in \mathbb{N}_{0}}b_{n}\left(\mathcal{S}_{2}^{(d,1)}\left(\phi(\theta)z_{1},h_{n}(\theta,z)\right)+\mathcal{S}_{2}^{(d,1)}\left(\overline{\phi}(\theta)z_{2},h_{n}(\theta,z)\right)\right), \\
&\left(\begin{aligned}
&\left[\widetilde{S}_{2}^{(d,2)}\left(\varphi(\theta),h_{x}(\theta,z)\right),\beta_{n_{c}}^{(1)}\right] \\
&\left[\widetilde{S}_{2}^{(d,2)}\left(\varphi(\theta),h_{x}(\theta,z)\right),\beta_{n_{c}}^{(2)}\right]
\end{aligned}\right) \\
&=\left(n_{c}/\ell\right)\sum_{n \in \mathbb{N}_{0}}(n/\ell)c_{n}\left(\mathcal{S}_{2}^{(d,2)}\left(\phi(\theta)z_{1},h_{n}(\theta,z)\right)+\mathcal{S}_{2}^{(d,2)}\left(\overline{\phi}(\theta)z_{2},h_{n}(\theta,z)\right)\right), \\
&\left(\begin{aligned}
&\left[\widetilde{S}_{2}^{(d,3)}\left(\varphi(\theta),h_{xx}(\theta,z)\right),\beta_{n_{c}}^{(1)}\right] \\
&\left[\widetilde{S}_{2}^{(d,3)}\left(\varphi(\theta),h_{xx}(\theta,z)\right),\beta_{n_{c}}^{(2)}\right]
\end{aligned}\right) \\
&=-\sum_{n \in \mathbb{N}_{0}}(n/\ell)^{2}b_{n}\left(\mathcal{S}_{2}^{(d,3)}\left(\phi(\theta)z_{1},h_{n}(\theta,z)\right)+\mathcal{S}_{2}^{(d,3)}\left(\overline{\phi}(\theta)z_{2},h_{n}(\theta,z)\right)\right),
\end{aligned}\end{equation*}
where $b_{n}$ is defined by (2.49) and
\begin{equation*}
c_{n}=\int_{0}^{\ell\pi}\xi_{n_{c}}(x)\xi_{n}(x)\gamma_{n_{c}}(x)dx=\left\{\begin{aligned}
& \frac{1}{\sqrt{2\ell\pi}}, & n=2n_{c}, \\
& 0, & \text { otherwise },
\end{aligned}\right.
\end{equation*}
and for $\phi(\theta)=\left(\phi_{1}(\theta),\phi_{2}(\theta)\right)^{T},~y(\theta)=\left(y_{1}(\theta),y_{2}(\theta)\right)^{T} \in C\left([-1,0], \mathbb{R}^{2}\right)$, we have
\begin{equation*}
\left\{\begin{aligned}
&\mathcal{S}_{2}^{(d,1)}(\phi(\theta),y(\theta))=2\left(\begin{array}{c}
\xi\tau_{c}\phi_{2}(0)y_{1}(0) \\
-d_{21}\tau_{c}\phi_{1}(-1)y_{2}(0)
\end{array}\right), \\
&\mathcal{S}_{2}^{(d,2)}(\phi(\theta),y(\theta))=2\left(\begin{array}{c}
\xi\tau_{c}(\phi_{2}(0)y_{1}(0)+\phi_{1}(0)y_{2}(0)) \\
-d_{21}\tau_{c}(\phi_{2}(0)y_{1}(-1)+\phi_{1}(-1)y_{2}(0))
\end{array}\right), \\
&\mathcal{S}_{2}^{(d,3)}(\phi(\theta),y(\theta))=2\left(\begin{array}{c}
\xi \tau_{c}\phi_{1}(0)y_{2}(0) \\
-d_{21}\tau_{c}\phi_{2}(0)y_{1}(-1)
\end{array}\right).
\end{aligned}\right.
\end{equation*}

Furthermore, from (2.37), (2.39) and (2.40), we have
\begin{equation*}\begin{aligned}
&(D_{w,w_{x},w_{xx}}f_{2}^{(1,2)}(z,0,0))U_{2}^{(2,d)}(z,0)(\theta) \\
&=\Psi(0)\left(\begin{aligned}
&\left[D_{w,w_{x},w_{xx}}F_{2}^{d}\left(\varphi(\theta),w,w_{x},w_{xx}\right)U_{2}^{(2,d)}(z,0)(\theta),\beta_{n_{c}}^{(1)}\right] \\
&\left[D_{w,w_{x},w_{xx}}F_{2}^{d}\left(\varphi(\theta),w,w_{x},w_{xx}\right)U_{2}^{(2,d)}(z,0)(\theta),\beta_{n_{c}}^{(2)}\right]
\end{aligned}\right),
\end{aligned}\end{equation*}
and then we obtain
\begin{equation*}
\operatorname{Proj}_{S}\left((D_{w,w_{x},w_{xx}}f_{2}^{(1,2)}(z,0,0))U_{2}^{(2,d)}(z,0)(\theta)\right)=\mathcal{H}\left(B_{23}z_{1}^{2}z_{2}\right),
\end{equation*}
where
\begin{equation}\begin{aligned}
B_{23}=&-\frac{1}{\sqrt{\ell \pi}}(n_{c}/\ell)^{2}\psi^{T}\left(\mathcal{S}_{2}^{(d,1)}\left(\phi(\theta),h_{0,11}(\theta)\right)+\mathcal{S}_{2}^{(d,1)}\left(\overline{\phi}(\theta), h_{0,20}(\theta)\right)\right) \\
&+\frac{1}{\sqrt{2\ell\pi}}\psi^{T}\sum_{j=1,2,3}b_{2n_{c}}^{(j)}\left(\mathcal{S}_{2}^{(d,j)}\left(\phi(\theta),h_{2n_{c},11}(\theta)\right)+\mathcal{S}_{2}^{(d,j)}\left(\overline{\phi}(\theta), h_{2n_{c},20}(\theta)\right)\right)
\end{aligned}\end{equation}
with
\begin{equation*}
b_{2 n_{c}}^{(1)}=-\frac{n_{c}^{2}}{\ell^{2}},~b_{2n_{c}}^{(2)}=\frac{2n_{c}^{2}}{\ell^{2}},~b_{2n_{c}}^{(3)}=-\frac{(2n_{c})^{2}}{\ell^{2}}.
\end{equation*}

\section{Normal form of the Hopf bifurcation and the corresponding coefficients}
\label{sec:3}

According to the algorithm developed in Section 2, we can obtain the normal form of the Hopf bifurcation truncated to the third-order term
\begin{equation}
\dot{z}=Bz+\frac{1}{2}\left(\begin{array}{c}
B_{1}z_{1} \mu \\
\overline{B}_{1}z_{2}\mu
\end{array}\right)+\frac{1}{3!}\left(\begin{array}{c}
B_{2}z_{1}^{2}z_{2} \\
\overline{B}_{2}z_{1}z_{2}^{2}
\end{array}\right)+O\left(|z| \mu^{2}+\left|z\right|^{4}\right),
\end{equation}
where
\begin{equation*}\begin{aligned}
B_{1}&=2\psi^{T}(0)\left(A\phi(0)-\frac{n_{c}^{2}}{\ell^{2}}\left(D_{1}\phi(0)+D_{2}\phi(-1)\right)\right), \\
B_{2}&=B_{21}+\frac{3}{2}\left(B_{22}+B_{23}\right).
\end{aligned}\end{equation*}

Here, $B_{1}$ is determined by (2.36), $B_{21}$, $B_{22}$ and $B_{23}$ are determined by (2.42), (2.50), (2.51), respectively, and they can be calculated by using the MATLAB software. The normal form (3.1) can be written in real coordinates through the change of variables $z_{1}=v_{1}-iv_{2},~z_{2}=v_{1}+iv_{2}$, and then changing to polar coordinates by $v_{1}=\rho \cos \Theta,~v_{2}=\rho \sin \Theta$, where $\Theta$ is the azimuthal angle. Therefore, by the above transformation and removing the azimuthal term $\Theta$, (3.1) can be rewritten as
\begin{equation*}
\dot{\rho}=K_{1}\mu\rho+K_{2}\rho^{3}+O\left(\mu^{2}\rho+\lvert(\rho,\mu)\rvert^{4}\right),
\end{equation*}
where
\begin{equation*}
K_{1}=\frac{1}{2}\operatorname{Re}\left(B_{1}\right),~K_{2}=\frac{1}{3!}\operatorname{Re}\left(B_{2}\right).
\end{equation*}

According to \cite{lv29}, the sign of $K_{1}K_{2}$ determines the direction of the Hopf bifurcation, and the sign of $K_{2}$ determines the stability of the Hopf bifurcation periodic solution. More precisely, we have the following results

(i) when $K_{1}K_{2}<0$, the Hopf bifurcation is supercritical, and the Hopf bifurcation periodic solution is stable for $K_{2}<0$ and unstable for $K_{2}>0$;

(ii) when $K_{1}K_{2}>0$, the Hopf bifurcation is subcritical, and the Hopf bifurcation periodic solution is stable for $K_{2}<0$ and unstable for $K_{2}>0$.

From (2.42), (2.50) and (2.51), it is obvious that in order to obtain the value of $K_{2}$, we still need to calculate $h_{0,20}(\theta), h_{0,11}(\theta), h_{2 n_{c},20}(\theta), h_{2 n_{c},11}(\theta)$ and $A_{ij}$.

\subsection{Calculations of $h_{0,20}(\theta), h_{0,11}(\theta), h_{2n_{c},20}(\theta)$ and $h_{2n_{c},11}(\theta)$}

From \cite{lv27}, we have
\begin{equation*}
M_{2}^{2}\left(h_{n}(\theta,z) \gamma_{n}(x)\right)=D_{z}\left(h_{n}(\theta,z) \gamma_{n}(x)\right)Bz-A_{\mathcal{Q}^{1}}\left(h_{n}(\theta,z)\gamma_{n}(x)\right),
\end{equation*}
which leads to
\begin{equation}\begin{aligned}
&\left(\begin{aligned}
\left[M_{2}^{2}\left(h_{n}(\theta,z)\gamma_{n}(x)\right),\beta_{n}^{(1)}\right] \\
\left[M_{2}^{2}\left(h_{n}(\theta,z)\gamma_{n}(x)\right),\beta_{n}^{(2)}\right]
\end{aligned}\right) \\
&=2i\omega_{c}\left(h_{n,20}(\theta)z_{1}^{2}-h_{n,02}(\theta)z_{2}^{2}\right)-\left(\dot{h}_{n}(\theta,z)+X_{0}(\theta)\left(\mathscr{L}_{0}\left(h_{n}(\theta,z)\right)-\dot{h}_{n}(0,z)\right)\right),
\end{aligned}\end{equation}
where
\begin{equation*}
\mathscr{L}_{0}\left(h_{n}(\theta,z)\right)=-\tau_{c}(n/\ell)^{2}\left(D_{1}h_{n}(0,z)+D_{2}h_{n}(-1,z)\right)+\tau_{c}Ah_{n}(0,z).
\end{equation*}

By (2.19) and the second mathematical expression in (2.25), we have
\begin{equation}\begin{aligned}
f_{2}^{2}(z,0,0)&=X_{0}(\theta)\widetilde{F}_{2}\left(\Phi(\theta)z_{x},0\right)-\pi\left(X_{0}(\theta)\widetilde{F}_{2}\left(\Phi(\theta)z_{x},0\right)\right) \\
&=X_{0}(\theta)\widetilde{F}_{2}\left(\Phi(\theta)z_{x},0\right)-\Phi(\theta)\Psi(0)\left(\begin{aligned}
\left[\widetilde{F}_{2}\left(\Phi(\theta)z_{x},0\right),\beta_{n_{c}}^{(1)}\right] \\
\left[\widetilde{F}_{2}\left(\Phi(\theta)z_{x},0\right),\beta_{n_{c}}^{(2)}\right]
\end{aligned}\right)\gamma_{n_{c}}(x).
\end{aligned}\end{equation}

Furthermore, by (2.43), (2.44) and (2.45), we have
\begin{equation}
\left(\begin{aligned}
\left[f_{2}^{2}(z,0,0),\beta_{n}^{(1)}\right] \\
\left[f_{2}^{2}(z,0,0),\beta_{n}^{(2)}\right]
\end{aligned}\right)=\left\{\begin{aligned}
& \frac{1}{\sqrt{\ell\pi}}X_{0}(\theta)\left(A_{20}z_{1}^{2}+A_{02}z_{2}^{2}+A_{11}z_{1}z_{2}\right), & n=0, \\
& \frac{1}{\sqrt{2\ell\pi}}X_{0}(\theta)\left(\widetilde{A}_{20}z_{1}^{2}+\widetilde{A}_{02}z_{2}^{2}+\widetilde{A}_{11}z_{1}z_{2}\right), & n=2n_{c},
\end{aligned}\right.
\end{equation}
where $\widetilde{A}_{j_{1}j_{2}}$ is defined as follows
\begin{equation}
\left\{\begin{aligned}
&\widetilde{A}_{j_{1}j_{2}}=A_{j_{1}j_{2}}-2\left(n_{c}/\ell\right)^{2}A_{j_{1}j_{2}}^{d}, \\
&j_{1},j_{2}=0,1,2,~j_{1}+j_{2}=2,
\end{aligned}\right.
\end{equation}
where $A_{j_{1}j_{2}}^{d}$ is determined by (2.46), and $A_{j_{1}j_{2}}$ will be calculated in the following section. Therefore, from (2.39), (3.2), (3.3), (3.4), and by matching the coefficients of $z_{1}^{2}$ and $z_{1}z_{2}$, we have
\begin{equation}
n=0,~\left\{\begin{aligned}
&z_{1}^{2}: \left\{\begin{aligned}
&\dot{h}_{0,20}(\theta)-2i\omega_{c}h_{0,20}(\theta)=(0,0)^{T}, \\
&\dot{h}_{0,20}(0)-L_{0}\left(h_{0,20}(\theta)\right)=\frac{1}{\sqrt{\ell\pi}}A_{20},
\end{aligned}\right. \\
&z_{1} z_{2}: \left\{\begin{aligned}
&\dot{h}_{0,11}(\theta)=(0,0)^{T}, \\
&\dot{h}_{0,11}(0)-L_{0}\left(h_{0,11}(\theta)\right)=\frac{1}{\sqrt{\ell\pi}}A_{11}
\end{aligned}\right.
\end{aligned}\right.
\end{equation}
and
\begin{equation}
n=2n_{c},~\left\{\begin{aligned}
&z_{1}^{2}: \left\{\begin{aligned}
&\dot{h}_{2n_{c},20}(\theta)-2i\omega_{c}h_{2n_{c},20}(\theta)=(0,0)^{T}, \\
&\dot{h}_{2n_{c},20}(0)-\mathscr{L}_{0}\left(h_{2n_{c},20}(\theta)\right)=\frac{1}{\sqrt{2\ell\pi}}\widetilde{A}_{20},
\end{aligned}\right. \\
&z_{1} z_{2}: \left\{\begin{aligned}
&\dot{h}_{2n_{c},11}(\theta)=(0,0)^{T}, \\
&\dot{h}_{2n_{c},11}(0)-\mathscr{L}_{0}\left(h_{2n_{c},11}(\theta)\right)=\frac{1}{\sqrt{2\ell\pi}}\widetilde{A}_{11}.
\end{aligned}\right.
\end{aligned}\right.
\end{equation}

Next, by combining with (3.6) and (3.7), we will give the mathematical expressions of $h_{0,20}(\theta)$, $h_{0,11}(\theta)$, $h_{2n_{c},20}(\theta)$ and $h_{2n_{c},11}(\theta)$.
\par~~~
\par\noindent (1) Calculations of $h_{0,20}(\theta)$ and $ h_{0,11}(\theta)$

(i) Notice that
\begin{equation}
\left\{\begin{aligned}
\dot{h}_{0,20}(\theta)-2i\omega_{c}h_{0,20}(\theta)&=(0,0)^{T}, \\
\dot{h}_{0,20}(0)-L_{0}\left(h_{0,20}(\theta)\right)&=\frac{1}{\sqrt{\ell\pi}}A_{20},
\end{aligned}\right.
\end{equation}
then from (3.8), we have $h_{0,20}(\theta)=e^{2iw_{c}\theta}h_{0,20}(0)$ and $\dot{h}_{0,20}(0)-2i\omega_{c}h_{0,20}(0)=(0,0)^{T}$. Notice that $L_{0}\left(h_{0,20}(\theta)\right)=\tau_{c}Ah_{0,20}(0)$, then we have
\begin{equation*}
(2iw_{c}I-\tau_{c}A)h_{0,20}(0)=\frac{1}{\sqrt{\ell\pi}}A_{20},
\end{equation*}
and hence $h_{0,20}(\theta)=e^{2iw_{c}\theta}C_{1}$ with
\begin{equation*}
C_{1}=(2iw_{c}I-\tau_{c} A)^{-1}\frac{1}{\sqrt{\ell\pi}}A_{20}.
\end{equation*}

(ii) Notice that
\begin{equation}
\left\{\begin{aligned}
&\dot{h}_{0,11}(\theta)=(0,0)^{T}, \\
&\dot{h}_{0,11}(0)-L_{0}\left(h_{0,11}(\theta)\right)=\frac{1}{\sqrt{\ell\pi}}A_{11},
\end{aligned}\right.
\end{equation}
then from (3.9), we have $h_{0,11}(\theta)=h_{0,11}(0)$ and $\dot{h}_{0,11}(0)=(0,0)^{T}$. Notice that $L_{0}\left(h_{0,11}(\theta)\right)=\tau_{c}Ah_{0,11}(0)$, then we have
\begin{equation*}
-\tau_{c}Ah_{0,11}(0)=\frac{1}{\sqrt{\ell\pi}}A_{11},
\end{equation*}
and hence $h_{0,11}(\theta)=C_{2}$ with
\begin{equation*}
C_{2}=(-\tau_{c}A)^{-1}\frac{1}{\sqrt{\ell\pi}}A_{11}.
\end{equation*}

\par~~~
\par\noindent (2) Calculations of $h_{2n_{c},20}(\theta)$ and $ h_{2n_{c},11}(\theta)$

(i) Notice that
\begin{equation}
\left\{\begin{aligned}
&\dot{h}_{2n_{c},20}(\theta)-2i\omega_{c}h_{2n_{c},20}(\theta)=(0,0)^{T}, \\
&\dot{h}_{2n_{c},20}(0)-\mathscr{L}_{0}\left(h_{2n_{c},20}(\theta)\right)=\frac{1}{\sqrt{2\ell\pi}}\widetilde{A}_{20},
\end{aligned}\right.
\end{equation}
then from (3.10), we have $h_{2n_{c},20}(\theta)=e^{2iw_{c}\theta}h_{2n_{c},20}(0)$, and hence $h_{2 n_{c},20}(-1)=e^{-2iw_{c}}h_{2 n_{c},20}(0)$. Furthermore, from (3.10) and
\begin{equation*}
\mathscr{L}_{0}\left(h_{2n_{c},20}(\theta)\right)=-\tau_{c}\frac{4n_{c}^{2}}{\ell^{2}}\left(D_{1}h_{2n_{c},20}(0)+D_{2}h_{2n_{c},20}(-1)\right)+\tau_{c}Ah_{2n_{c},20}(0),
\end{equation*}
we have
\begin{equation}
2i\omega_{c}h_{2 n_{c},20}(0)=\frac{1}{\sqrt{2\ell\pi}}\widetilde{A}_{20}-\tau_{c}\frac{4n_{c}^{2}}{\ell^{2}}\left(D_{1}h_{2n_{c},20}(0)+D_{2}h_{2n_{c},20}(-1)\right)+\tau_{c}Ah_{2n_{c},20}(0).
\end{equation}
Therefore, by combining with $h_{2 n_{c},20}(-1)=e^{-2iw_{c}}h_{2n_{c},20}(0)$ and (3.11), we can obtain
\begin{equation*}
(2iw_{c}I+\tau_{c}\frac{4n_{c}^{2}}{\ell^{2}}D_{1}+\tau_{c}\frac{4n_{c}^{2}}{\ell^{2}}D_{2}e^{-2iw_{c}}-\tau_{c}A)h_{2n_{c},20}(0)=\frac{1}{\sqrt{2\ell\pi}}\widetilde{A}_{20},
\end{equation*}
and hence
\begin{equation*}
h_{2n_{c},20}(\theta)=e^{2iw_{c}\theta}C_{3}
\end{equation*}
with
\begin{equation*}
C_{3}=(2iw_{c}I+\tau_{c}\frac{4n_{c}^{2}}{\ell^{2}}D_{1}+\tau_{c}\frac{4n_{c}^{2}}{\ell^{2}}D_{2}e^{-2iw_{c}}-\tau_{c}A)^{-1}\frac{1}{\sqrt{2\ell\pi}}\widetilde{A}_{20}.
\end{equation*}

Here, $\widetilde{A}_{20}$ and $A_{20}^{d}$ are defined by (3.5) and (2.46), respectively.

(ii) Notice that
\begin{equation}
\left\{\begin{aligned}
&\dot{h}_{2n_{c},11}(\theta)=(0,0)^{T}, \\
&\dot{h}_{2n_{c},11}(0)-\mathscr{L}_{0}\left(h_{2 n_{c},11}(\theta)\right)=\frac{1}{\sqrt{2\ell\pi}}\widetilde{A}_{11},
\end{aligned}\right.
\end{equation}
then from (3.12), we have $h_{2n_{c},11}(\theta)=h_{2n_{c},11}(0)$, and hence $h_{2n_{c},11}(-1)=h_{2n_{c},11}(0)$. Furthermore, from (3.12) and
\begin{equation*}
\mathscr{L}_{0}\left(h_{2n_{c},11}(\theta)\right)=-\tau_{c}\frac{4n_{c}^{2}}{\ell^{2}}\left(D_{1}h_{2n_{c},11}(0)+D_{2}h_{2n_{c},11}(-1)\right)+\tau_{c}Ah_{2n_{c},11}(0),
\end{equation*}
we have
\begin{equation}
(0,0)^{T}=-\tau_{c}\frac{4n_{c}^{2}}{\ell^{2}}\left(D_{1}h_{2n_{c},11}(0)+D_{2}h_{2n_{c},11}(-1)\right)+\tau_{c}Ah_{2n_{c},11}(0)+\frac{1}{\sqrt{2\ell\pi}}\widetilde{A}_{11}.
\end{equation}
Therefore, by combining with $h_{2n_{c},11}(-1)=h_{2n_{c},11}(0)$ and (3.13), we can obtain
\begin{equation*}
\left(\tau_{c}\frac{4n_{c}^{2}}{\ell^{2}}D_{1}+\tau_{c}\frac{4n_{c}^{2}}{\ell^{2}}D_{2}-\tau_{c}A\right)h_{2n_{c},11}(0)=\frac{1}{\sqrt{2\ell\pi}}\widetilde{A}_{11},
\end{equation*}
and hence
\begin{equation*}
h_{2n_{c},11}(\theta)=C_{4}
\end{equation*}
with
\begin{equation*}
C_{4}=\left(\tau_{c}\frac{4n_{c}^{2}}{\ell^{2}}D_{1}+\tau_{c}\frac{4n_{c}^{2}}{\ell^{2}}D_{2}-\tau_{c}A\right)^{-1}\frac{1}{\sqrt{2\ell\pi}}\widetilde{A}_{11}.
\end{equation*}

Here, $\widetilde{A}_{11}$ and $A_{11}^{d}$ are defined by (3.5) and (2.46), respectively.

\subsection{Calculations of $A_{i,j}$ and $\mathcal{S}_{2}(\Phi(\theta)z_{x},w)$}

In this subsection, let $F(\varphi,\mu)=\left(F^{(1)}(\varphi,\mu),F^{(2)}(\varphi,\mu)\right)^{T}$ and $\varphi=(\varphi_{1},\varphi_{2})^{T}\in\mathcal{C}$, and we write
\begin{equation}
\frac{1}{j!}F_{j}(\varphi,\mu)=\sum_{j_{1}+j_{2}=j}\frac{1}{j_{1}!j_{2}!}f_{j_{1}j_{2}}\varphi_{1}^{j_{1}}(0)\varphi_{2}^{j_{2}}(0),
\end{equation}
where $f_{j_{1}j_{2}}=\operatorname{col}\left(f_{j_{1}j_{2}}^{(1)},f_{j_{1}j_{2}}^{(2)}\right)$ with
\begin{equation*}
f_{j_{1}j_{2}}^{(k)}=\frac{\partial^{j_{1}+j_{2}}F^{(k)}(0,0)}{\partial\varphi_{1}^{j_{1}}(0) \partial\varphi_{2}^{j_{2}}(0)},~k=1,2.
\end{equation*}

Then from (3.14), we have
\begin{equation}\begin{aligned}
F_{2}(\varphi,\mu)&=F_{2}(\varphi,0)=2\sum_{j_{1}+j_{2}=2}\frac{1}{j_{1}!j_{2}!}f_{j_{1}j_{2}}\varphi_{1}^{j_{1}}(0)\varphi_{2}^{j_{2}}(0) \\
&=f_{20}\varphi_{1}^{2}(0)+f_{02}\varphi_{2}^{2}(0)+2f_{11}\varphi_{1}(0)\varphi_{2}(0)
\end{aligned}\end{equation}
and
\begin{equation}\begin{aligned}
F_{3}(\varphi,0)&=6\sum_{j_{1}+j_{2}=3}\frac{1}{j_{1}!j_{2}!}f_{j_{1}j_{2}}\varphi_{1}^{j_{1}}(0)\varphi_{2}^{j_{2}}(0) \\
&=f_{30}\varphi_{1}^{3}(0)+f_{03}\varphi_{2}^{3}(0)+3f_{21}\varphi_{1}^{2}(0)\varphi_{2}(0)+3f_{12}\varphi_{1}(0)\varphi_{2}^{2}(0).
\end{aligned}\end{equation}

Notice that
\begin{equation}\begin{aligned}
&\varphi(\theta)=\Phi(\theta)z_{x}=\phi(\theta)z_{1}(t)\gamma_{n_{c}}(x)+\overline{\phi}(\theta)z_{2}(t)\gamma_{n_{c}}(x) \\
&=\left(\begin{array}{c}
\phi_{1}(\theta)z_{1}(t)\gamma_{n_{c}}(x)+\overline{\phi}_{1}(\theta)z_{2}(t)\gamma_{n_{c}}(x) \\
\phi_{2}(\theta)z_{1}(t)\gamma_{n_{c}}(x)+\overline{\phi}_{2}(\theta)z_{2}(t)\gamma_{n_{c}}(x)
\end{array}\right) \\
&=\left(\begin{array}{c}
\varphi_{1}(\theta) \\
\varphi_{2}(\theta)
\end{array}\right),
\end{aligned}\end{equation}
and similar to (2.41), we have
\begin{equation}
F_{2}\left(\Phi(\theta)z_{x},0\right)=\sum_{q_{1}+q_{2}=2}A_{q_{1}q_{2}}\gamma_{n_{c}}^{q_{1}+q_{2}}(x)z_{1}^{q_{1}}z_{2}^{q_{2}},
\end{equation}
then by combining with (3.15), (3.17) and (3.18), we have
\begin{equation*}\begin{aligned}
A_{20}&=f_{20}\phi_{1}^{2}(0)+f_{02}\phi_{2}^{2}(0)+2f_{11}\phi_{1}(0)\phi_{2}(0), \\
A_{02}&=f_{20}\overline{\phi_{1}}^{2}(0)+f_{02}\overline{\phi_{2}}^{2}(0)+2f_{11}\overline{\phi_{1}}(0)\overline{\phi_{2}}(0), \\
A_{11}&=2f_{20}\phi_{1}(0)\overline{\phi_{1}}(0)+2f_{02}\phi_{2}(0)\overline{\phi_{2}}(0)+2f_{11}(\phi_{1}(0)\overline{\phi_{2}}(0)+\overline{\phi_{1}}(0)\phi_{2}(0)).
\end{aligned}\end{equation*}

Furthermore, from (2.41), (3.16) and (3.17), we have
\begin{equation*}\begin{aligned}
A_{30}&=f_{30}\phi_{1}^{3}(0)+f_{03}\phi_{2}^{3}(0)+3f_{21}\phi_{1}^{2}(0)\phi_{2}(0)+3f_{12}\phi_{1}(0)\phi_{2}^{2}(0), \\
A_{03}&=f_{30}\overline{\phi_{1}}^{3}(0)+f_{03}\overline{\phi_{2}}^{3}(0)+3f_{21}\overline{\phi_{1}}^{2}(0)\overline{\phi_{2}}(0)+3f_{12}\overline{\phi_{1}}(0)\overline{\phi_{2}}^{2}(0), \\
A_{21}&=3f_{30}\phi_{1}^{2}(0)\overline{\phi_{1}}(0)+3f_{03}\phi_{2}^{2}(0)\overline{\phi_{2}}(0)+3f_{21}(\phi_{1}^{2}(0)\overline{\phi_{2}}(0)+2\phi_{1}(0)\overline{\phi_{1}}(0)\phi_{2}(0)) \\
&+3f_{12}(2\phi_{1}(0)\phi_{2}(0)\overline{\phi_{2}}(0)+\overline{\phi_{1}}(0)\phi_{2}^{2}(0)), \\
A_{12}&=3f_{30}\phi_{1}(0)\overline{\phi_{1}}^{2}(0)+3f_{03}\phi_{2}(0)\overline{\phi_{2}}^{2}(0)+3f_{21}(2\phi_{1}(0)\overline{\phi_{1}}(0)\overline{\phi_{2}}(0)+\overline{\phi_{1}}^{2}(0)\phi_{2}(0)) \\
&+3f_{12}(\phi_{1}(0)\overline{\phi_{2}}^{2}(0)+2\overline{\phi_{1}}(0)\phi_{2}(0)\overline{\phi_{2}}(0)).
\end{aligned}\end{equation*}

Moreover, from (3.14), we have
\begin{equation}\begin{aligned}
F_{2}(\varphi(\theta)+w,\mu)&=F_{2}(\varphi(\theta)+w,0)=2\sum_{j_{1}+j_{2}=2}\frac{1}{j_{1}!j_{2}!}f_{j_{1}j_{2}} (\varphi_{1}(0)+w_{1}(0))^{j_{1}}(\varphi_{2}(0)+w_{2}(0))^{j_{2}} \\
&=f_{20}(\varphi_{1}(0)+w_{1}(0))^{2}+f_{02}(\varphi_{2}(0)+w_{2}(0))^{2}+2f_{11}(\varphi_{1}(0)+w_{1}(0))(\varphi_{2}(0)+w_{2}(0)).
\end{aligned}\end{equation}
Notice that
\begin{equation}\begin{aligned}
\varphi(\theta)+w&=\Phi(\theta)z_{x}+w=\phi(\theta)z_{1}(t)\gamma_{n_{c}}(x)+\overline{\phi}(\theta)z_{2}(t)\gamma_{n_{c}}(x)+w \\
&=\left(\begin{array}{c}
\phi_{1}(\theta)z_{1}(t)\gamma_{n_{c}}(x)+\overline{\phi}_{1}(\theta)z_{2}(t)\gamma_{n_{c}}(x)+w_{1} \\
\phi_{2}(\theta)z_{1}(t)\gamma_{n_{c}}(x)+\overline{\phi}_{2}(\theta)z_{2}(t)\gamma_{n_{c}}(x)+w_{2}
\end{array}\right) \\
&=\left(\begin{array}{c}
\varphi_{1}(\theta)+w_{1} \\
\varphi_{2}(\theta)+w_{2}
\end{array}\right)
\end{aligned}\end{equation}
and
\begin{equation}\begin{aligned}
F_{2}\left(\Phi(\theta)z_{x}+w,\mu\right)&=F_{2}\left(\Phi(\theta)z_{x}+w,0\right) \\
&=\sum_{q_{1}+q_{2}=2} A_{q_{1}q_{2}}\gamma_{n_{c}}^{q_{1}+q_{2}}(x)z_{1}^{q_{1}}z_{2}^{q_{2}}+\mathcal{S}_{2}\left(\Phi(\theta)z_{x}, w\right)+O\left(|w|^{2}\right),
\end{aligned}\end{equation}
then by combining with (3.19), (3.20) and (3.21), we have
\begin{equation*}\begin{aligned}
&\mathcal{S}_{2}\left(\Phi(\theta)z_{x},w\right) \\
&=2f_{20}(\phi_{1}(0)z_{1}(t)\gamma_{n_{c}}(x)+\overline{\phi_{1}}(0)z_{2}(t)\gamma_{n_{c}}(x))w_{1}(0) \\
&+2f_{02}(\phi_{2}(0)z_{1}(t)\gamma_{n_{c}}(x)+\overline{\phi_{2}}(0)z_{2}(t)\gamma_{n_{c}}(x))w_{2}(0) \\
&+2f_{11}\left((\phi_{1}(0)z_{1}(t)\gamma_{n_{c}}(x)+\overline{\phi_{1}}(0)z_{2}(t)\gamma_{n_{c}}(x))w_{2}(0)+(\phi_{2}(0)z_{1}(t)\gamma_{n_{c}}(x)+\overline{\phi_{2}}(0)z_{2}(t)\gamma_{n_{c}}(x))w_{1}(0) \right).
\end{aligned}\end{equation*}

\section{Application to Holling-Tanner model with spatial memory and predator-taxis}
\label{sec:4}

In this section, we apply our newly developed algorithm in Section 2 to the Holling-Tanner model with spatial memory and predator-taxis, i.e., in the model (1.4), we let
\begin{equation}\begin{aligned}
&f\left(u(x,t),v(x,t)\right)=u(x,t)\left(1-\beta u(x,t)\right)-\frac{mu(x,t)v(x,t)}{1+u(x,t)}, \\
&g\left(u(x,t),v(x,t)\right)=sv(x,t)\left(1-\frac{v(x,t)}{u(x,t)}\right),
\end{aligned}\end{equation}
where $\beta>0$, $m>0$ and $s>0$.

Thus, (1.4) becomes the following model
\begin{equation}\left\{\begin{aligned}
&\frac{\partial u(x,t)}{\partial t}=d_{11}\Delta u(x,t)+\xi\left(u(x,t)v_{x}(x,t)\right)_{x}+u(x,t)\left(1-\beta u(x,t)\right)-\frac{mu(x,t)v(x,t)}{1+u(x,t)}, & x \in (0,\ell\pi),~t>0, \\
&\frac{\partial v(x,t)}{\partial t}=d_{22}\Delta v(x,t)-d_{21}\left(v(x,t)u_{x}(x,t-\tau)\right)_{x}+sv(x,t)\left(1-\frac{v(x,t)}{u(x,t)}\right), & x \in (0,\ell\pi),~t>0, \\
&u_{x}(0,t)=u_{x}(\ell\pi,t)=v_{x}(0,t)=v_{x}(\ell\pi,t)=0, & t \geq 0.
\end{aligned}\right.\end{equation}

The Holling-tanner model is one of the typical predator-prey models. For the ordinary differential equation (4.2) with $d_{11}=\xi=d_{21}=d_{22}=0$, it has been completely analyzed in \cite{lv30}. For the diffusive model (4.2) with $\xi=d_{21}=0$, the global stability of the positive constant steady state was proved in \cite{lv31,lv32}, and the Hopf bifurcation and Turing instability have been studied in \cite{lv33}.

\subsection{Stability and Hopf bifurcation analysis}

The system (4.2) has the positive constant steady state $E_{*}\left(u_{*},v_{*}\right)$, where
\begin{equation}
u_{*}=v_{*}=\frac{1}{2 \beta}\left(\sqrt{R^{2}+4 \beta}-R\right)
\end{equation}
with $R=\beta+m-1$. By combining with $E_{*}\left(u_{*},v_{*}\right)$, (2.4) and (4.1), we have
\begin{equation}\begin{aligned}
&a_{11}=1-2\beta u_{*}-\frac{mu_{*}}{(1+u_{*})^{2}},~a_{12}=-\frac{mu_{*}}{1+u_{*}}<0, \\
&a_{21}=s>0,~a_{22}=-s<0.
\end{aligned}\end{equation}

Moreover, by combining with (2.3), (2.5), (2.6) and (4.4), the characteristic equation of system (4.2) can be written as
\begin{equation}
\Gamma_{n}(\lambda)=\det\left(\mathcal{M}_{n}(\lambda)\right)=\lambda^{2}-T_{n}\lambda+\widetilde{J}_{n}(\tau)=0,
\end{equation}
where
\begin{equation}\begin{aligned}
T_{n}&=\operatorname{Tr}(A)-\operatorname{Tr}(D_{1})\frac{n^{2}}{\ell^{2}}, \\
\widetilde{J}_{n}(\tau)&=(d_{11}d_{22}+d_{21}\xi u_{*}v_{*}e^{-\lambda\tau})\frac{n^{4}}{\ell^{4}}-\left(d_{11}a_{22}+d_{22}a_{11}-a_{21}\xi u_{*}+d_{21}v_{*}a_{12}e^{-\lambda\tau}\right)\frac{n^{2}}{\ell^{2}}+\operatorname{Det}(A).
\end{aligned}\end{equation}
Notice that the mathematical expression in (4.6) the same as in (2.7).

Furthermore, when $\tau=0$, the characteristic equation (4.5) becomes
\begin{equation}
\lambda^{2}-T_{n}\lambda+\widetilde{J}_{n}(0)=0,
\end{equation}
where
\begin{equation}
\widetilde{J}_{n}(0)=(d_{11}d_{22}+d_{21}\xi u_{*}v_{*})\frac{n^{4}}{\ell^{4}}-\left(d_{11}a_{22}+d_{22}a_{11}-a_{21}\xi u_{*}+d_{21}v_{*}a_{12}\right)\frac{n^{2}}{\ell^{2}}+\operatorname{Det}(A).
\end{equation}

A set of sufficient and necessary condition that all roots of (4.7) have negative real parts is $T_{n}<0,~\widetilde{J}_{n}(0)>0$, which is always holds provided that $a_{11}<0$, i.e.,
\begin{equation*}
(C_{0}):~1-2\beta u_{*}-\frac{mu_{*}}{(1+u_{*})^{2}}<0.
\end{equation*}

This implies that when $\tau=0$ and the condition $(C_{0})$ holds, the positive steady state $E_{*}(u_{*},v_{*})$ is asymptotically stable for $d_{11} \geq 0$, $\xi \geq 0$, $d_{21} \geq 0$ and $d_{22} \geq 0$. Meanwhile, if we let $d_{21}=0$, then by (4.8), we have
\begin{equation*}
\widetilde{J}_{n}:=d_{11}d_{22}\frac{n^{4}}{\ell^{4}}-\left(d_{11}a_{22}+d_{22}a_{11}-a_{21}\xi u_{*}\right)\frac{n^{2}}{\ell^{2}}+\operatorname{Det}(A).
\end{equation*}

It is easy to verify that $T_{n}<0$ and $\widetilde{J}_{n}>0$ provided that the condition $(C_{0})$ holds. This implies that when $\tau=0$, $d_{21}=0$ and the condition $(C_{0})$ holds, the positive steady state $E_{*}(u_{*},v_{*})$ is asymptotically stable for $d_{11} \geq 0$, $\xi \geq 0$ and $d_{22} \geq 0$. Furthermore, since $\Gamma_{n}(0)=\widetilde{J}_{n}(0)>0$ under the condition $(C_{0})$, this implies that $\lambda=0$ is not a root of (4.5).

In the following, we let
\begin{equation*}\begin{aligned}
J_{n}=d_{11}d_{22}\frac{n^{4}}{\ell^{4}}-\left(d_{11}a_{22}+d_{22}a_{11}-a_{21}\xi u_{*}\right)\frac{n^{2}}{\ell^{2}}+\operatorname{Det}(A).
\end{aligned}\end{equation*}
Furthermore, let $\lambda=i \omega_{n}(\omega_{n}>0)$ be a root of (4.5). By substituting it along with expressions in (4.6) into (4.5), and separating the real part from the imaginary part, we have
\begin{equation}
\left\{\begin{aligned}
&\omega_{n}^{2}-J_{n}=\left(d_{21}\xi u_{*}v_{*}\frac{n^{4}}{\ell^{4}}-d_{21}v_{*}a_{12}\frac{n^{2}}{\ell^{2}}\right)\cos(\omega_{n}\tau), \\
&-T_{n}\omega_{n}=\left(d_{21}\xi u_{*}v_{*}\frac{n^{4}}{\ell^{4}}-d_{21}v_{*}a_{12}\frac{n^{2}}{\ell^{2}}\right)\sin(\omega_{n}\tau),
\end{aligned}\right.
\end{equation}
which yields
\begin{equation}
\omega_{n}^{4}+P_{n}\omega_{n}^{2}+Q_{n}=0,
\end{equation}
where
\begin{equation*}
P_{n}=T_{n}^{2}-2J_{n}=\left(d_{11}^{2}+d_{22}^{2}\right)\frac{n^{4}}{\ell^{4}}-2\left(d_{11}a_{11}+d_{22}a_{22}+a_{21}\xi u_{*}\right)\frac{n^{2}}{\ell^{2}}+a_{11}^{2}+a_{22}^{2}+2a_{12}a_{21},
\end{equation*}
and
\begin{equation}
Q_{n}=\left(J_{n}+\left(d_{21}\xi u_{*}v_{*}\frac{n^{4}}{\ell^{4}}-d_{21}v_{*}a_{12}\frac{n^{2}}{\ell^{2}}\right)\right)\left(J_{n}-\left(d_{21}\xi u_{*} v_{*}\frac{n^{4}}{\ell^{4}}-d_{21}v_{*}a_{12}\frac{n^{2}}{\ell^{2}}\right)\right).
\end{equation}

Notice that for (4.10), it is easy to see that if
\begin{equation*}
\text{ either } P_{n}>0 \text{ and } Q_{n}>0 \text{ or } P_{n}^{2}-4Q_{n}<0,
\end{equation*}
then (4.10) has no positive root. Suppose that
\begin{equation*}
Q_{n}>0,~P_{n}<0 \text { and } P_{n}^{2}-4Q_{n}>0,
\end{equation*}
then (4.10) has two positive roots. In addition, if
\begin{equation*}
\text{ either } Q_{n}<0 \text{ or } P_{n}<0 \text{ and } P_{n}^{2}-4Q_{n}=0,
\end{equation*}
then (4.10) has only one positive root.

\begin{case}
It is easy to see that if the conditions $(C_{0})$ and
\begin{equation*}
(C_{1}):~P_{n}>0 \text{ and } Q_{n}>0 \text{ or } P_{n}^{2}-4Q_{n}<0,
\end{equation*}
hold, then (4.10) has no positive roots. Hence, by combining with the assumption 2.1, we know that all roots of (4.5) have negative real parts when $\tau \in[0,+\infty)$ under the conditions $(C_{0})$ and $(C_{1})$.
\end{case}

In the following, we mainly consider the case of $Q_{n}<0$, that is (4.10) has only one positive root $\omega_{n}$. In the following, we will discuss the case which is used to guarantee $Q_{n}<0$ under the condition $(C_{0})$.

When $\tau>0$, according to (4.5) and (4.11), we can define $Q_{n}=\Gamma_{n}(0)\widetilde{Q}_{n}$ with
\begin{equation*}
\Gamma_{n}(0)=\widetilde{J}_{n}(0)=(d_{11}d_{22}+d_{21}\xi u_{*}v_{*})\frac{n^{4}}{\ell^{4}}-\left(d_{11}a_{22}+d_{22}a_{11}-a_{21}\xi u_{*}+d_{21}v_{*}a_{12}\right)\frac{n^{2}}{\ell^{2}}+\operatorname{Det}(A)
\end{equation*}
and
\begin{equation}
\widetilde{Q}_{n}=(d_{11}d_{22}-d_{21}\xi u_{*}v_{*})\frac{n^{4}}{\ell^{4}}-(d_{11}a_{22}+d_{22}a_{11}-a_{21}\xi u_{*}-d_{21}v_{*}a_{12})\frac{n^{2}}{\ell^{2}}+\operatorname{Det}(A),
\end{equation}
and then by a simple analysis, we have $\Gamma_{n}(0)=\widetilde{J}_{n}(0)>0$ for any $n \in \mathbb{N}_{0}$. Therefore, the sign of $Q_{n}$ coincides with that of $\widetilde{Q}_{n}$, and in order to guaranteeing $Q_{n}<0$, we only need to study the case of $\widetilde{Q}_{n}<0$.

\begin{case}
If $(d_{11}d_{22}-d_{21}\xi u_{*}v_{*})>0$ and the conditions $(C_{0})$,
\begin{equation*}\begin{aligned}
&(C_{2}):~\operatorname{Det}(A)>0,~d_{11}a_{22}+d_{22}a_{11}-a_{21}\xi u_{*}-d_{21}v_{*}a_{12}>0 \\
&(d_{11}a_{22}+d_{22}a_{11}-a_{21}\xi u_{*}-d_{21}v_{*}a_{12})^{2}-4(d_{11}d_{22}-d_{21}\xi u_{*}v_{*})\operatorname{Det}(A)>0
\end{aligned}\end{equation*}
hold, then (4.12) has two positive roots. Without loss of generality, we assume that the two positive roots of (4.12) are $\widetilde{x}_{1}$ and $\widetilde{x}_{2}$, i.e.,
\begin{equation}
\widetilde{x}_{1,2}=\frac{\widetilde{A}_{1} \mp \sqrt{\widetilde{A}_{3}}}{2\widetilde{A}_{2}},
\end{equation}
where
\begin{equation}\begin{aligned}
\widetilde{A}_{1}&=d_{11}a_{22}+d_{22}a_{11}-a_{21}\xi u_{*}-d_{21}v_{*}a_{12},~\widetilde{A}_{2}=d_{11}d_{22}-d_{21}\xi u_{*}v_{*}, \\
\widetilde{A}_{3}&=\widetilde{A}_{1}^{2}-4\widetilde{A}_{2}\operatorname{Det}(A).
\end{aligned}\end{equation}

Since $\widetilde{x}_{1}=n_{1}^{2}/\ell^{2}$ and $\widetilde{x}_{2}=n_{2}^{2}/\ell^{2}$, then $n_{1}=\ell\sqrt{\widetilde{x}_{1}}$ and $n_{2}=\ell\sqrt{\widetilde{x}_{2}}$. By using a geometric argument, we can conclude that
\begin{equation*}
Q_{n}=\Gamma_{n}(0)\widetilde{Q}_{n}\left\{\begin{aligned}
& <0, & n_{1}<n<n_{2}, \\
& \geq 0, & n\leq n_{1} \text{ or } n \geq n_{2},
\end{aligned}
\right.\end{equation*}
where $n \in \mathbb{N}$. Therefore, (4.10) has one positive root $\omega_{n}$ for $n_{1}<n<n_{2}$ with $n\in \mathbb{N}$, where
\begin{equation}
\omega_{n}=\sqrt{\frac{-P_{n}+\sqrt{P_{n}^{2}-4Q_{n}}}{2}}.
\end{equation}
Furthermore, by combining with the second mathematical expression in (4.9), and notice that $a_{12}<0$, $T_{n}<0$ under the condition $(C_{0})$, then we have $\sin(\omega_{n}\tau)>0$. Thus, from the first mathematical expression in (4.9), we can set
\begin{equation}
\tau_{n,j}=\frac{1}{\omega_{n}}\left\{\arccos\left\{\frac{\omega_{n}^{2}-J_{n}}{d_{21}\xi u_{*}v_{*}(n^{4}/\ell^{4})-d_{21}v_{*}a_{12} (n^{2}/\ell^{2})}\right\}+2j\pi\right\},~n \in \mathbb{N},~j \in \mathbb{N}_{0}.
\end{equation}
\end{case}

Next, we continue to verify the transversality conditions for the Case 4.2.
\begin{lemma}
Suppose that $(d_{11}d_{22}-d_{21}\xi u_{*}v_{*})>0$, the conditions $(C_{0})$, $(C_{2})$ hold, and $n_{1}<n<n_{2}$ with $n\in \mathbb{N}$, then we have
\begin{equation*}
\left.\frac{d \operatorname{Re}(\lambda(\tau))}{d \tau}\right|_{\tau=\tau_{n,j}}>0,
\end{equation*}
where $\operatorname{Re}(\lambda(\tau))$ represents the real part of $\lambda(\tau)$.
\end{lemma}

\begin{proof}
By differentiating the two sides of
\begin{equation*}
\Gamma_{n}(\lambda)=\operatorname{det}\left(\mathcal{M}_{n}(\lambda)\right)=\lambda^{2}-T_{n}\lambda+\widetilde{J}_{n}(\tau)=0
\end{equation*}
with respect to $\tau$, where $T_{n}$ and $\widetilde{J}_{n}(\tau)$ are defined by (4.6), we have
\begin{equation}
\left(\frac{d\lambda(\tau)}{d\tau}\right)^{-1}=\frac{\left(2\lambda-T_{n}\right)e^{\lambda\tau}}{-\lambda d_{21}v_{*}a_{12}(n^{2}/\ell^{2})+\lambda d_{21}\xi u_{*}v_{*}(n^{4}/\ell^{4})}-\frac{\tau}{\lambda}.
\end{equation}
Therefore, by (4.17), we have
\begin{equation}\begin{aligned}
\operatorname{Re}\left(\left.\frac{d\lambda(\tau)}{d\tau}\right|_{\tau=\tau_{n,j}}\right)^{-1}&=\operatorname{Re}\left(\frac{\left(2i\omega_{n}-T_{n}\right)e^{i\omega_{n}\tau_{n,j}}}{-i\omega_{n}d_{21}v_{*}a_{12}(n^{2}/\ell^{2})+i\omega_{n}d_{21}\xi u_{*}v_{*}(n^{4}/\ell^{4})}\right) \\
&=\operatorname{Re}\left(\frac{\left(2i\omega_{n}-T_{n}\right)\left(\cos(\omega_{n}\tau_{n,j})+i\sin(\omega_{n}\tau_{n,j})\right)}{-i\omega_{n}d_{21}v_{*}a_{12}(n^{2}/\ell^{2})+i\omega_{n}d_{21}\xi u_{*}v_{*}(n^{4}/\ell^{4})}\right) \\
&=\operatorname{Re}\left(\frac{\left(2i\omega_{n}-T_{n}\right)\left(\cos(\omega_{n}\tau_{n,j})+i\sin(\omega_{n}\tau_{n,j})\right)}{i\omega_{n}(d_{21}\xi u_{*}v_{*}(n^{4}/\ell^{4})-d_{21}v_{*}a_{12}(n^{2}/\ell^{2}))}\right) \\
&=\frac{2\cos(\omega_{n}\tau_{n,j})}{d_{21}\xi u_{*}v_{*}(n^{4}/\ell^{4})-d_{21}v_{*}a_{12}(n^{2}/\ell^{2})}-\frac{T_{n}\sin(\omega_{n}\tau_{n,j})}{\omega_{n}(d_{21}\xi u_{*}v_{*}(n^{4}/\ell^{4})-d_{21}v_{*}a_{12}(n^{2}/\ell^{2}))}.
\end{aligned}\end{equation}
Furthermore, according to (4.9), we have
\begin{equation}\begin{aligned}
&\sin(\omega_{n}\tau_{n,j})=\frac{-T_{n}\omega_{n}}{d_{21}\xi u_{*}v_{*}(n^{4}/\ell^{4})-d_{21}v_{*}a_{12}(n^{2}/\ell^{2})}, \\
&\cos(\omega_{n}\tau_{n,j})=\frac{\omega_{n}^{2}-J_{n}}{d_{21}\xi u_{*}v_{*}(n^{4}/\ell^{4})-d_{21}v_{*}a_{12}(n^{2}/\ell^{2})}.
\end{aligned}\end{equation}
Moreover, by combining with (4.18), (4.19) and
\begin{equation*}
\omega_{n}=\sqrt{\frac{-P_{n}+\sqrt{P_{n}^{2}-4Q_{n}}}{2}}>0,~Q_{n}<0,~a_{12}<0,
\end{equation*}
we have
\begin{equation*}\begin{aligned}
\operatorname{Re}\left(\left.\frac{d\lambda(\tau)}{d\tau}\right|_{\tau=\tau_{n,j}}\right)^{-1}&=\frac{2\cos(\omega_{n}\tau_{n,j})}{d_{21}\xi u_{*}v_{*}(n^{4}/\ell^{4})-d_{21}v_{*}a_{12}(n^{2}/\ell^{2})}-\frac{T_{n}\sin(\omega_{n}\tau_{n,j})}{\omega_{n}(d_{21}\xi u_{*}v_{*}(n^{4}/\ell^{4})-d_{21}v_{*}a_{12}(n^{2}/\ell^{2}))} \\
&=\frac{2\omega_{n}^{3}+\omega_{n}(T_{n}^{2}-2J_{n})}{\omega_{n}\left(d_{21}\xi u_{*}v_{*}(n^{4}/\ell^{4})-d_{21}v_{*}a_{12}(n^{2}/\ell^{2})\right)^{2}} \\
&=\frac{\sqrt{P_{n}^{2}-4Q_{n}}}{\left(d_{21}\xi u_{*}v_{*}(n^{4}/\ell^{4})-d_{21}v_{*}a_{12}(n^{2}/\ell^{2})\right)^{2}}>0.
\end{aligned}\end{equation*}

This, together with the fact that
\begin{equation*}
\operatorname{sign}\left\{\left.\frac{d \operatorname{Re}(\lambda(\tau))}{d\tau}\right|_{\tau=\tau_{n,j}}\right\}=\operatorname{sign}\left\{\operatorname{Re}\left(\left.\frac{d \lambda(\tau)}{d\tau}\right|_{\tau=\tau_{n,j}}\right)^{-1}\right\}
\end{equation*}
completes the proof, where $\operatorname{sign}(.)$ represents the sign function.
\end{proof}

Moreover, according to the above analysis, we have the following results.
\begin{lemma}
If the condition $(C_{0})$ is satisfied, then we have the following conclusions

(i) if the condition $(C_{1})$ holds, then the positive constant steady state $E_{*}\left(u_{*},v_{*}\right)$ of system (4.2) is asymptotically stable for all $\tau \geq 0$;

(ii) if $(d_{11}d_{22}-d_{21}\xi u_{*}v_{*})>0$ and the condition $(C_{2})$ holds, by denoting $\tau_{*}=\min\left\{\tau_{n,0}: n_{1}<n<n_{2},~n \in \mathbb{N}\right\}$, then the positive constant steady state $E_{*}\left(u_{*},v_{*}\right)$ of system (4.2) is asymptotically stable for $0 \leq \tau<\tau_{*}$ and unstable for $\tau>\tau_{*}$. Furthermore, system (4.2) undergoes mode-$n$ Hopf bifurcations at $\tau=\tau_{n,0}$ for $n \in \mathbb{N}$.
\end{lemma}

\subsection{Numerical simulations}

In this section, we verify the analytical results given in the previous sections by some numerical simulations and investigate the direction and stability of Hopf bifurcation. We use the following initial conditions for the system (4.2)
\begin{equation*}
u(x,t)=u_{0}(x),~v(x,t)=v_{0}(x),~t \in\left[-\tau,0\right].
\end{equation*}

\subsubsection{Mode-1 Hopf bifurcation}

\begin{figure}[!htbp]
\centering
\includegraphics[width=2.3in]{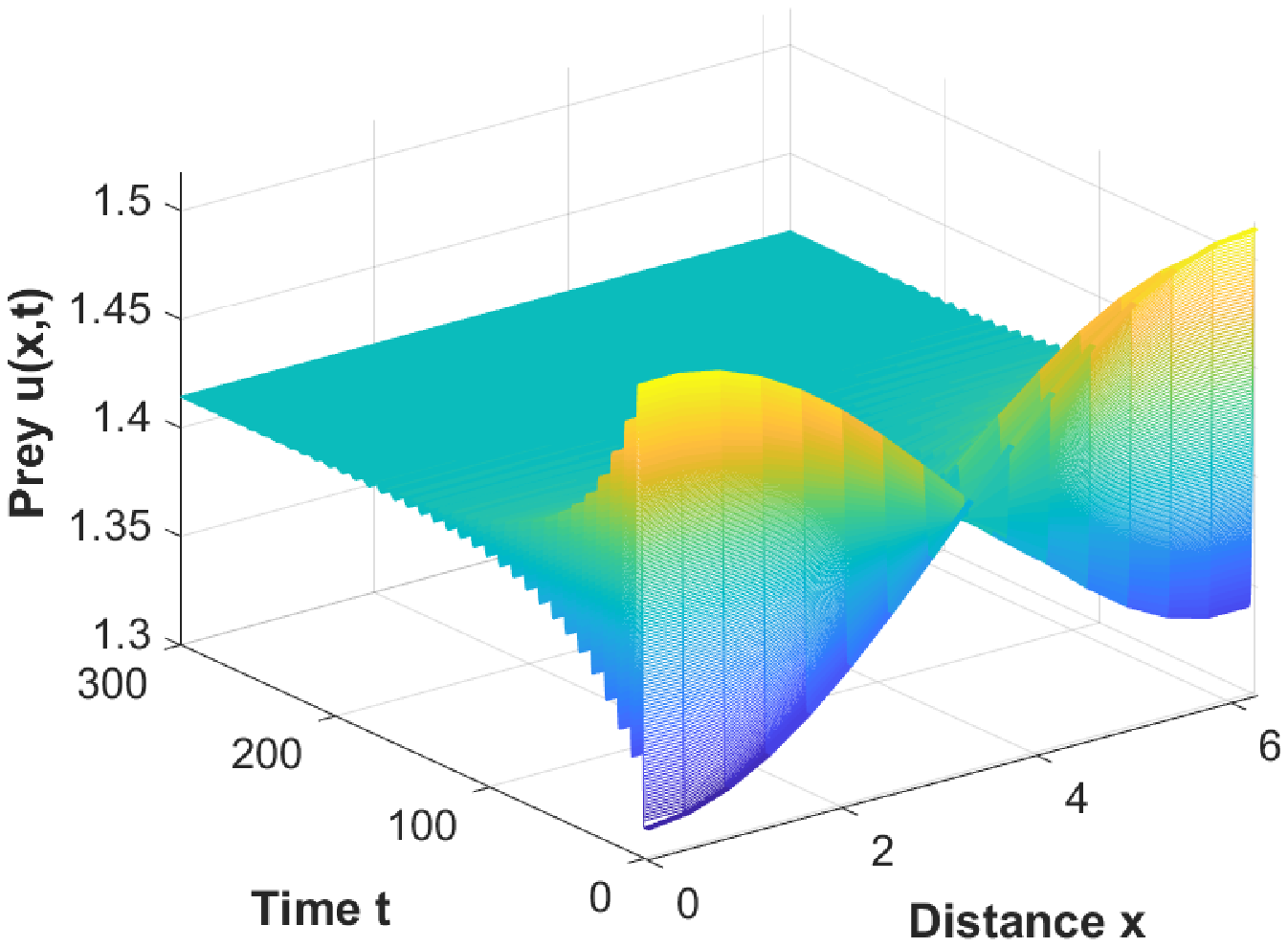}
\includegraphics[width=2.3in]{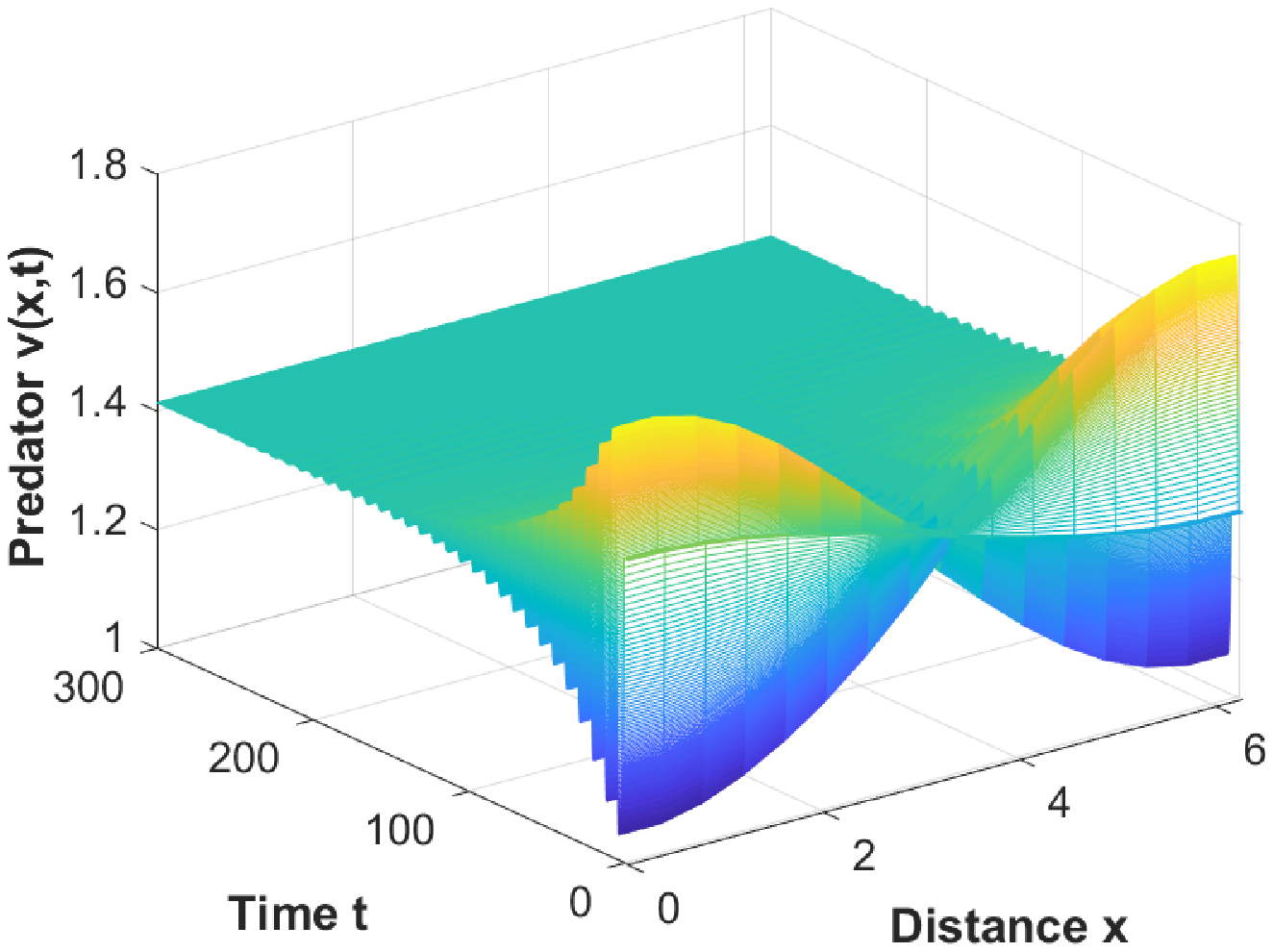} \\
\textbf{(a)} \hspace{5cm} \textbf{(b)} \\
\includegraphics[width=2.3in]{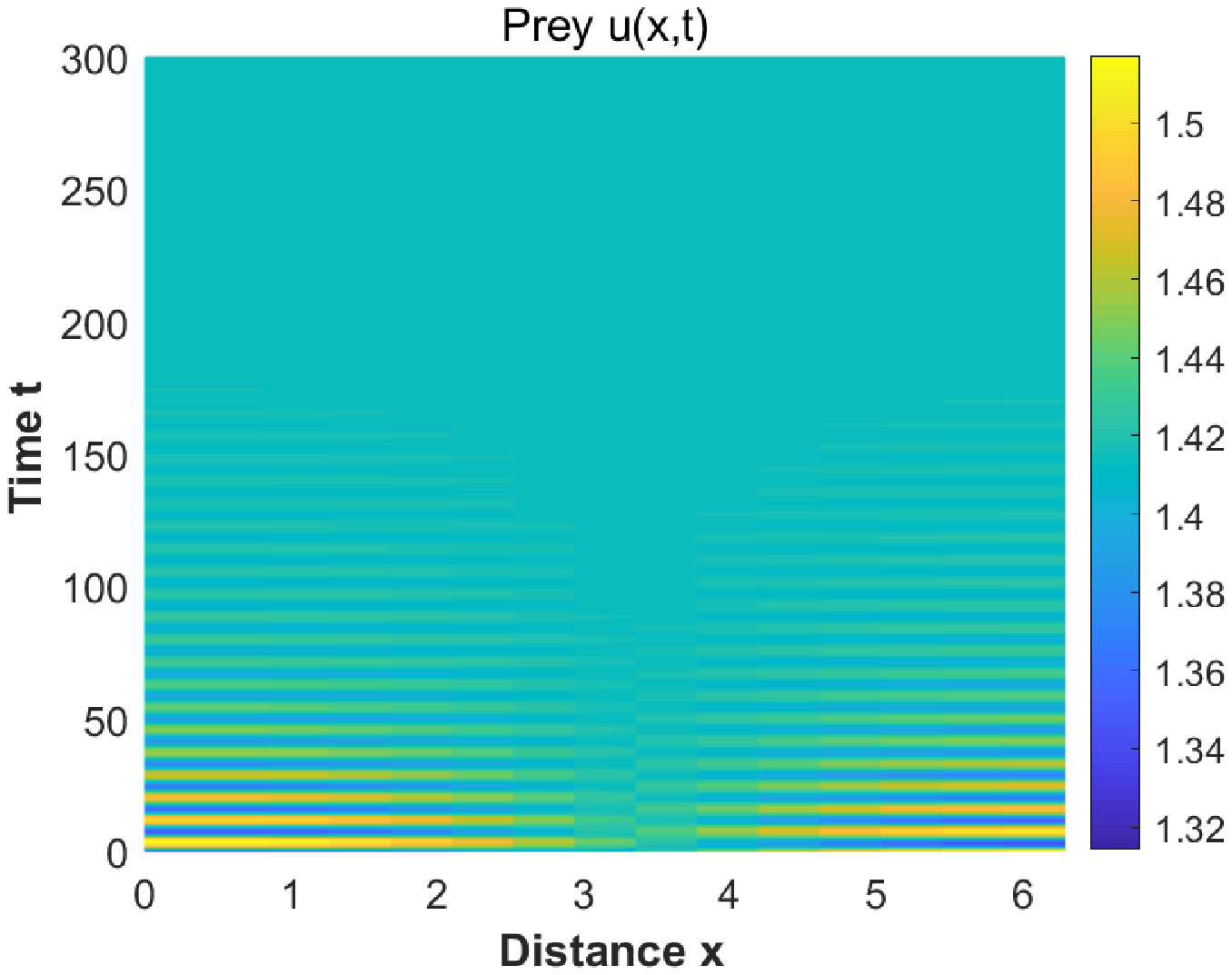}
\includegraphics[width=2.3in]{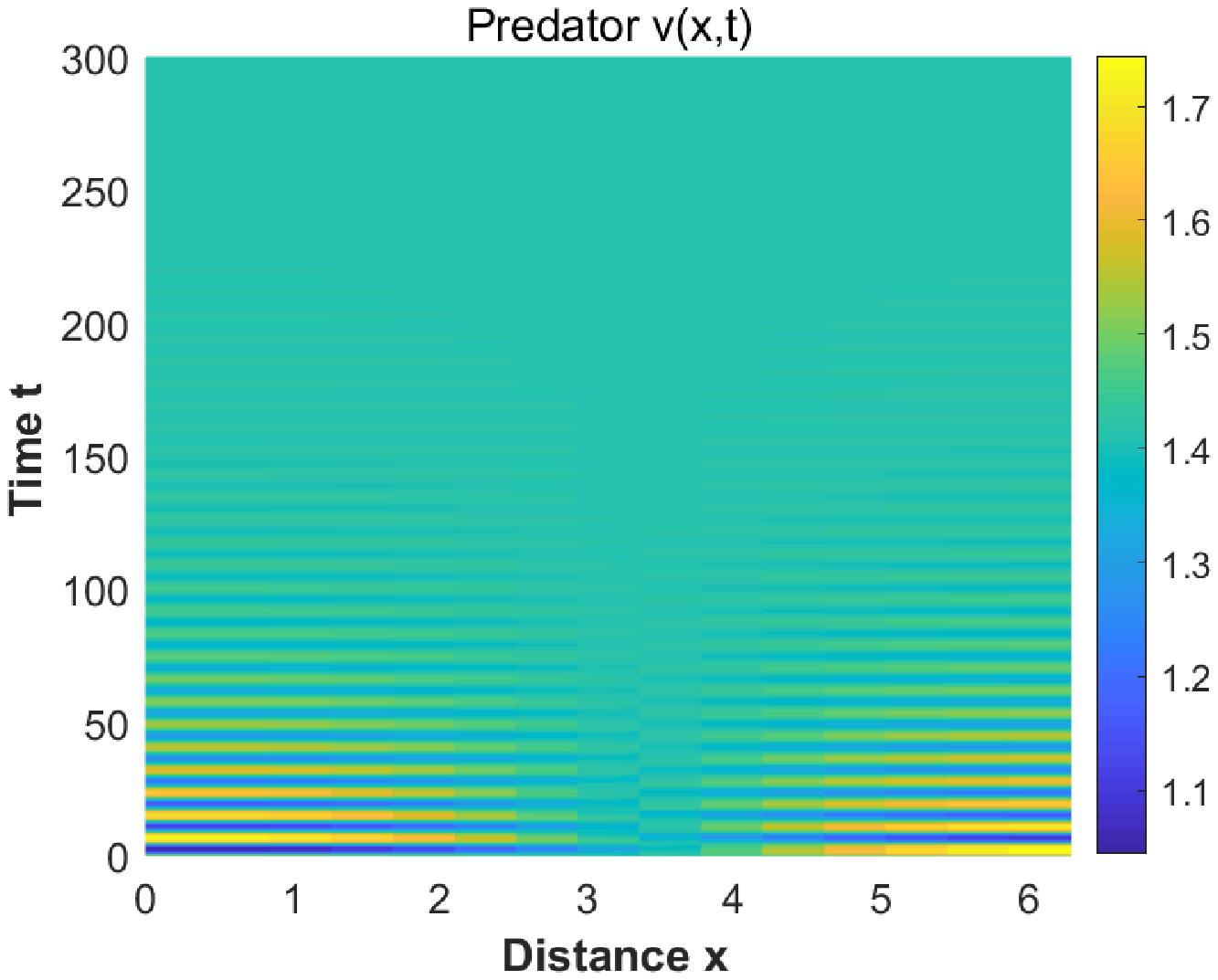} \\
\textbf{(c)} \hspace{5cm} \textbf{(d)} \\
\caption{For the parameters $\ell=2,~d_{11}=2,~d_{22}=3,~d_{21}=18,~\xi=0.06,~\beta=0.5,~m=0.5,~s=0.8$, when $\tau=3<\tau_{1,0}=6.1498$, the positive constant steady state $E_{*}\left(u_{*},v_{*}\right)=(1.4142,1.4142)$ is locally asymptotically stable. The initial values are $u_{0}(x)=1.4142-0.1\cos(x/2)$ and $v_{0}(x)=1.4142+0.1\cos(x/2)$.}
\label{fig:1}
\end{figure}

\begin{figure}[!htbp]
\centering
\includegraphics[width=2.3in]{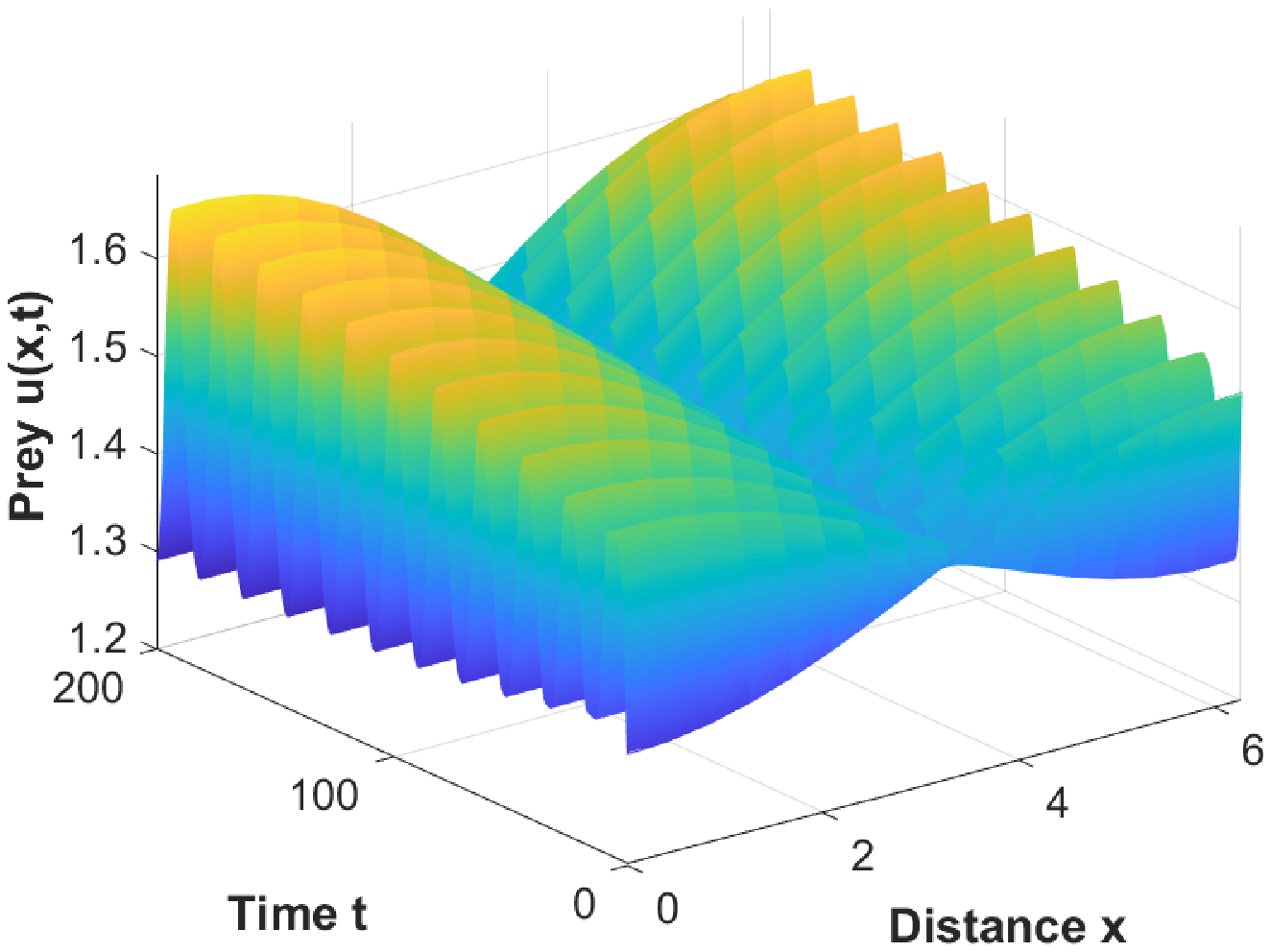}
\includegraphics[width=2.3in]{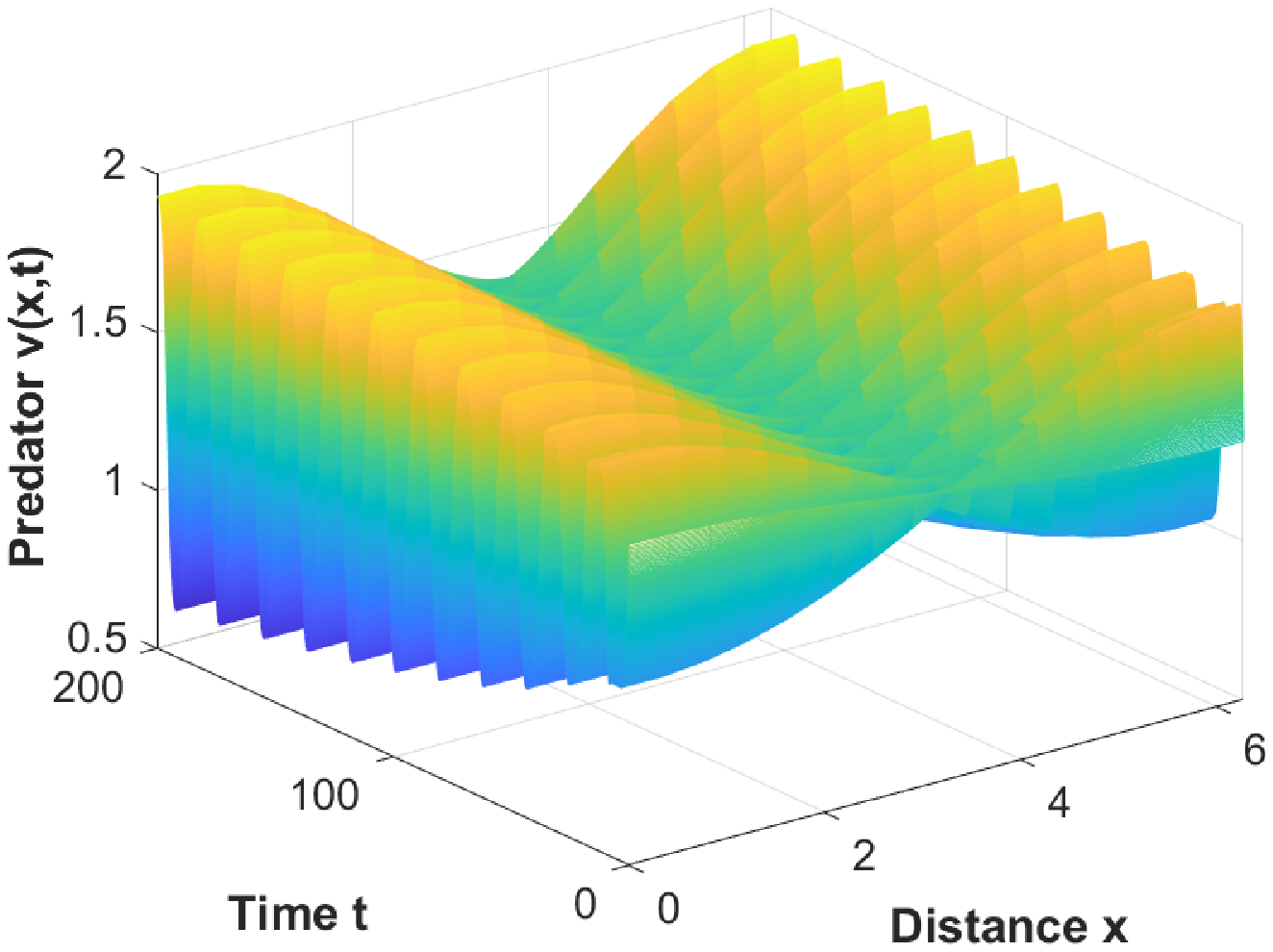} \\
\textbf{(a)} \hspace{5cm} \textbf{(b)} \\
\includegraphics[width=2.3in]{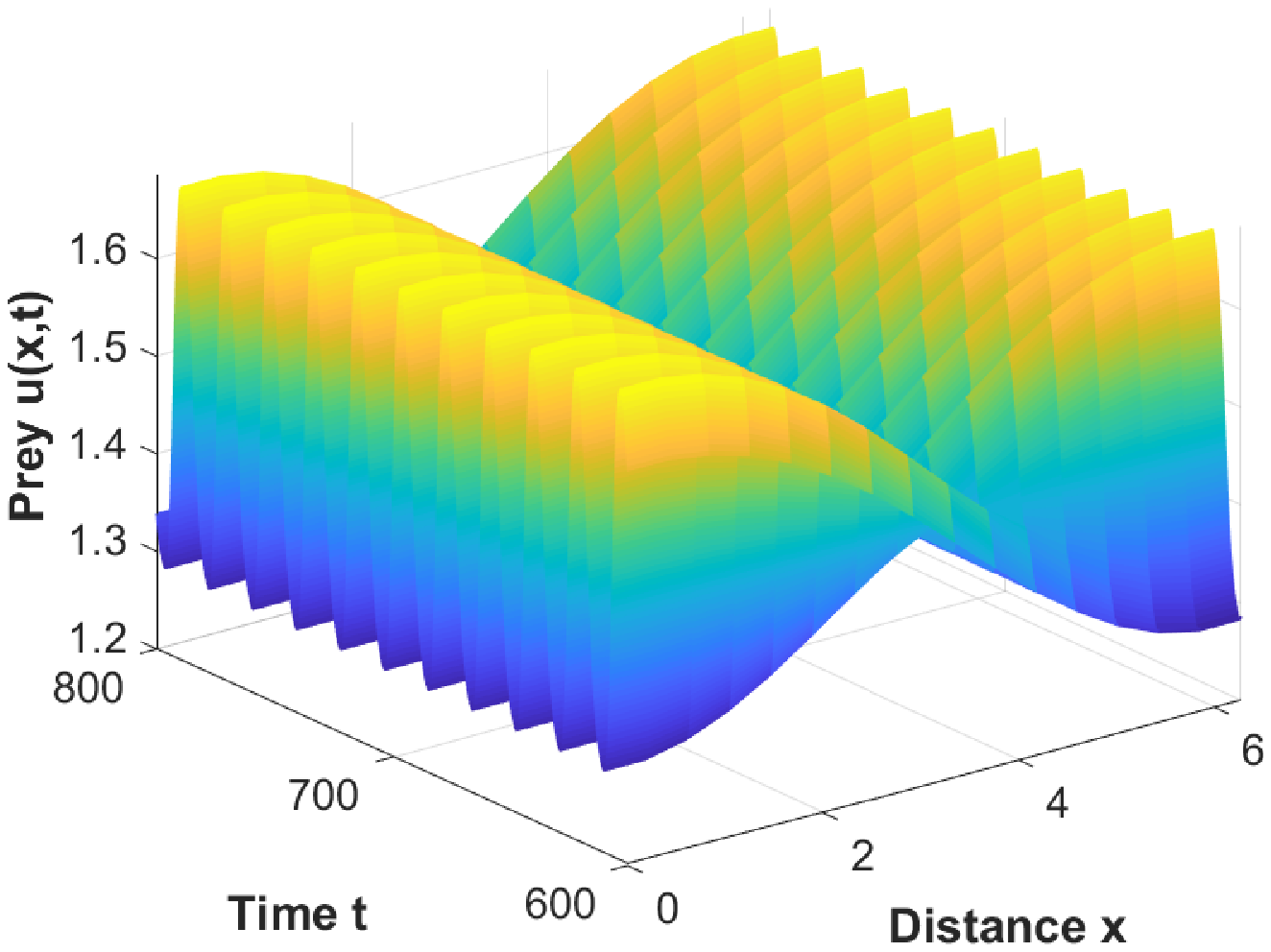}
\includegraphics[width=2.3in]{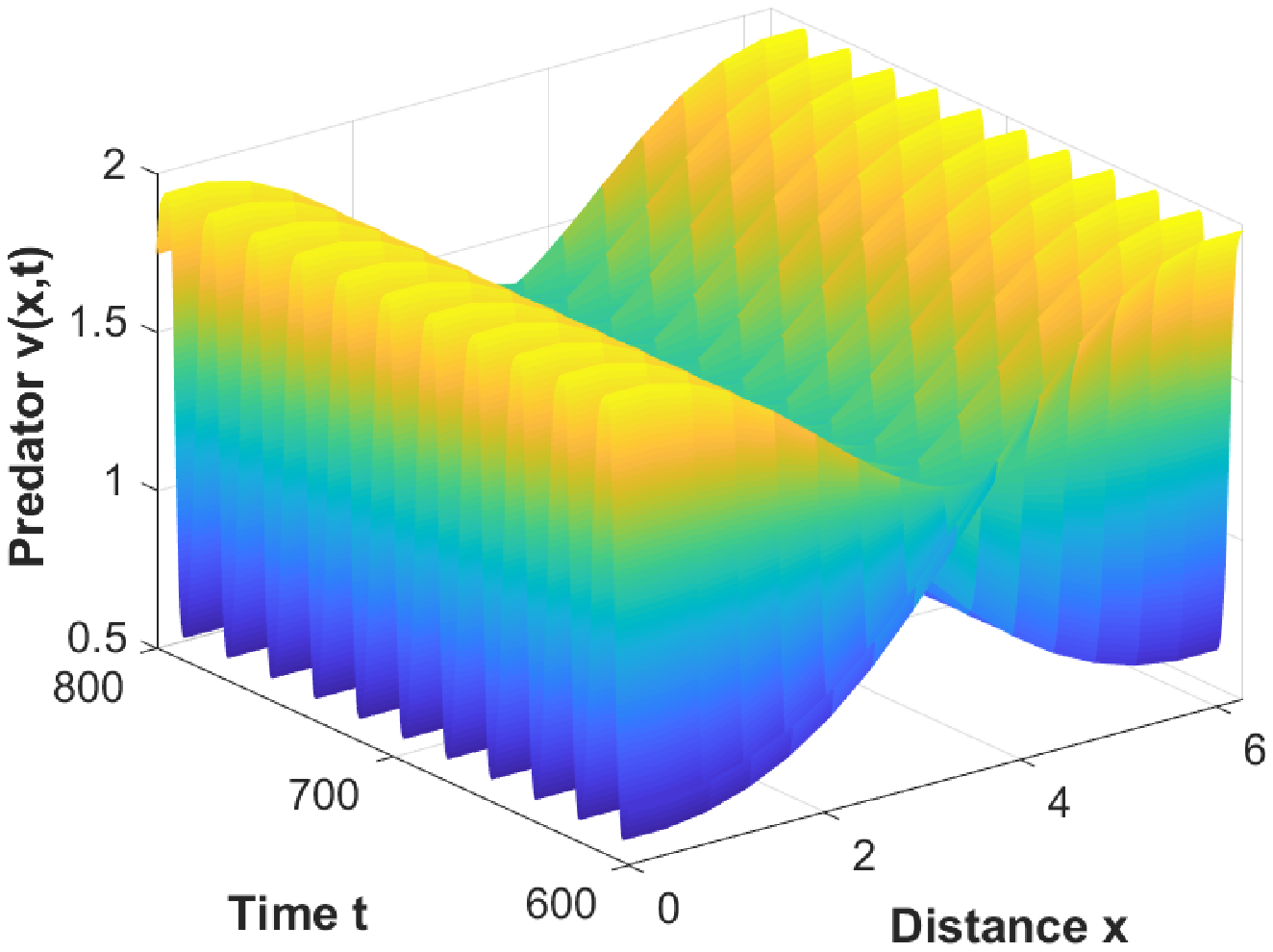} \\
\textbf{(c)} \hspace{5cm} \textbf{(d)} \\
\caption{For the parameters $\ell=2,~d_{11}=2,~d_{22}=3,~d_{21}=18,~\xi=0.06,~\beta=0.5,~m=0.5,~s=0.8$, when $\tau=8>\tau_{1,0}=6.1498$, there exists a stable spatially inhomogeneous periodic solution. (a) and (b) are the transient behaviours for $u(x,t)$ and $v(x,t)$, respectively, (c) and (d) are long-term behaviours for $u(x,t)$ and $v(x,t)$, respectively. The initial values are $u_{0}(x)=1.4142-0.1\cos(x/2)$ and $v_{0}(x)=1.4142+0.1\cos(x/2)$.}
\label{fig:2}
\end{figure}

\begin{figure}[!htbp]
\centering
\includegraphics[width=2.3in]{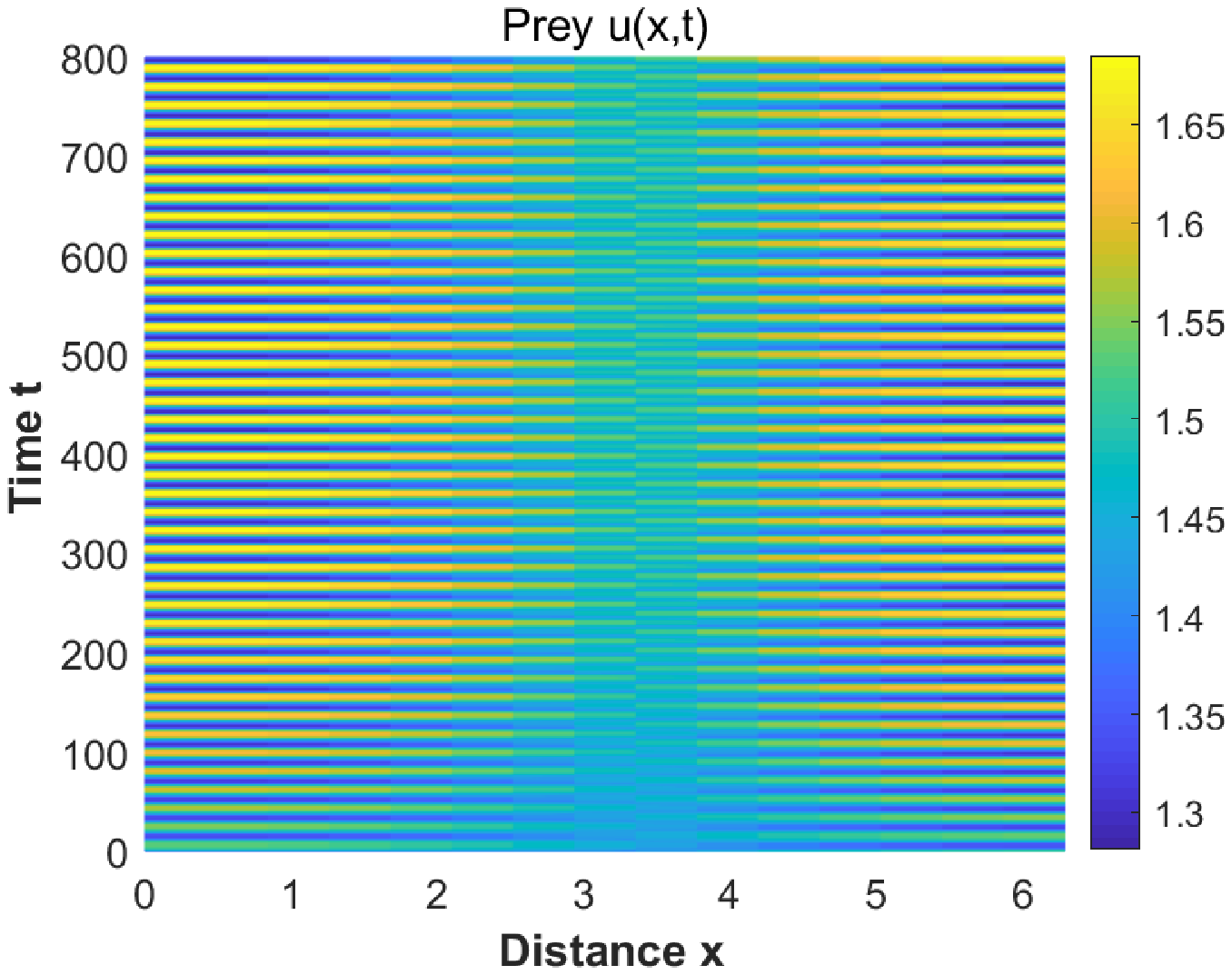}
\includegraphics[width=2.3in]{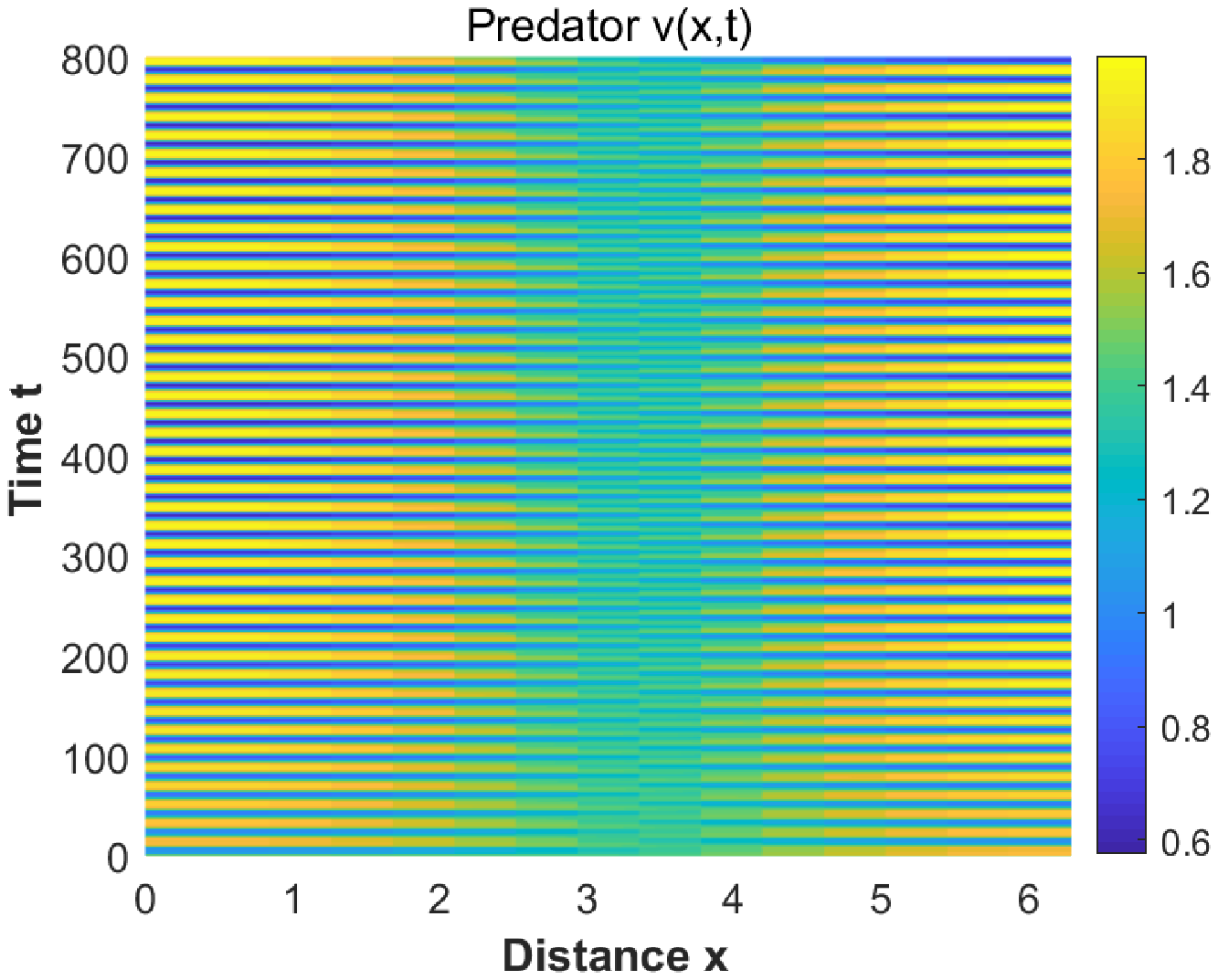} \\
\textbf{(a)} \hspace{5cm} \textbf{(b)} \\
\caption{For the parameters $\ell=2,~d_{11}=2,~d_{22}=3,~d_{21}=18,~\xi=0.06,~\beta=0.5,~m=0.5,~s=0.8$, when $\tau=8>\tau_{1,0}=6.1498$, there exists a stable spatially inhomogeneous periodic solution. The initial values are $u_{0}(x)=1.4142-0.1\cos(x/2)$ and $v_{0}(x)=1.4142+0.1\cos(x/2)$.}
\label{fig:3}
\end{figure}

If we set the parameters as follows
\begin{equation*}
\ell=2,~d_{11}=2,~d_{22}=3,~d_{21}=18,~\xi=0.06,~\beta=0.5,~m=0.5,~s=0.8,
\end{equation*}
then we can easily obtain that
\begin{equation*}\begin{aligned}
&a_{11}=1-2\beta u_{*}-\frac{mu_{*}}{(1+u_{*})^{2}}=-0.5355<0,~d_{11}d_{22}-d_{21}\xi u_{*}v_{*}=3.84>0, \\
&\operatorname{Det}(A)=a_{11}a_{22}-a_{12}a_{21}=0.6627>0,~d_{11}a_{22}+d_{22}a_{11}-a_{21}\xi u_{*}-d_{21}v_{*}a_{12}=4.1814>0, \\
&(d_{11}a_{22}+d_{22}a_{11}-a_{21}\xi u_{*}-d_{21}v_{*}a_{12})^{2}-4(d_{11}d_{22}-d_{21}\xi u_{*}v_{*})\operatorname{Det}(A)=7.3041>0.
\end{aligned}\end{equation*}
Therefore, the conditions $(C_{0})$ and $(C_{2})$ are satisfied under the above parameters settings. In the following, we mainly verify the conclusion in Lemma 4.4 (ii). According to (4.3) and (4.4), we have $E_{*}\left(u_{*},v_{*}\right)=(1.4142,1.4142)$,
\begin{equation*}
a_{11}=-0.5355,~a_{12}=-0.2929,~a_{21}=0.8,~a_{22}=-0.8.
\end{equation*}

By combining with (4.13), (4.14), (4.15) and (4.16), we have $n_{1}=0.8776$, $n_{2}=1.8935$, and consider that $n\in\mathbb{N}$, we have $\omega_{c}=\omega_{1}=0.418$ and $\tau_{c}=\tau_{1,0}=6.1498$. Moreover, by Lemma 4.4 (ii), we have the following proposition.

\begin{proposition}
For system (4.2) with the parameters $\ell=2,~d_{11}=2,~d_{22}=3,~d_{21}=18,~\xi=0.06,~\beta=0.5,~m=0.5,~s=0.8$, the positive constant steady state $E_{*}\left(u_{*},v_{*}\right)$ of system (4.2) is asymptotically stable for $0 \leq \tau<\tau_{1,0}=6.1498$ and unstable for $\tau>\tau_{1,0}=6.1498$. Furthermore, system (4.2) undergoes mode-1 Hopf bifurcation at $\tau=\tau_{1,0}=6.1498$.
\end{proposition}

For the parameters $\ell=2,~d_{11}=2,~d_{22}=3,~d_{21}=18,~\xi=0.06,~\beta=0.5,~m=0.5,~s=0.8$, according to Proposition 4.5, we know that system (4.2) undergoes a Hopf bifurcation at $\tau_{1,0}=6.1498$. Furthermore, the direction and stability of Hopf bifurcation can be determined by calculating $K_{1}K_{2}$ and $K_{2}$ using the procedures developed in Section 2. After a direct calculation using MATLAB software, we obtain
\begin{equation*}
K_{1}=0.016>0,~K_{2}=-0.9283<0,~K_{1}K_{2}=-0.0148<0,
\end{equation*}
which implies that the Hopf bifurcation at $\tau_{1,0}=6.1498$ is supercritical and stable.

When $\tau=3<\tau_{1,0}=6.1498$, Fig.1 (a)-(d) illustrate the evolution of the solution of system (4.2) starting from the initial values $u_{0}(x)=1.4142-0.1\cos(x/2)$ and $v_{0}(x)=1.4142+0.1\cos(x/2)$, finally converging to the positive constant steady state $E_{*}\left(u_{*},v_{*}\right)$. Furthermore, when $\tau=8>\tau_{1,0}=6.1498$, Fig.2 (a)-(d) illustrate the existence of the spatially homogeneous periodic solution with the initial values $u_{0}(x)=1.4142-0.1\cos(x/2)$ and $v_{0}(x)=1.4142+0.1\cos(x/2)$. When $\tau=8>\tau_{1,0}=6.1498$, the space-time diagrams for the prey population $u(x,t)$ and the predator population $v(x,t)$ are given in Fig.3 (a) and (b), respectively.

\subsubsection{Mode-2 Hopf bifurcation}

\begin{figure}[!htbp]
\centering
\includegraphics[width=2.3in]{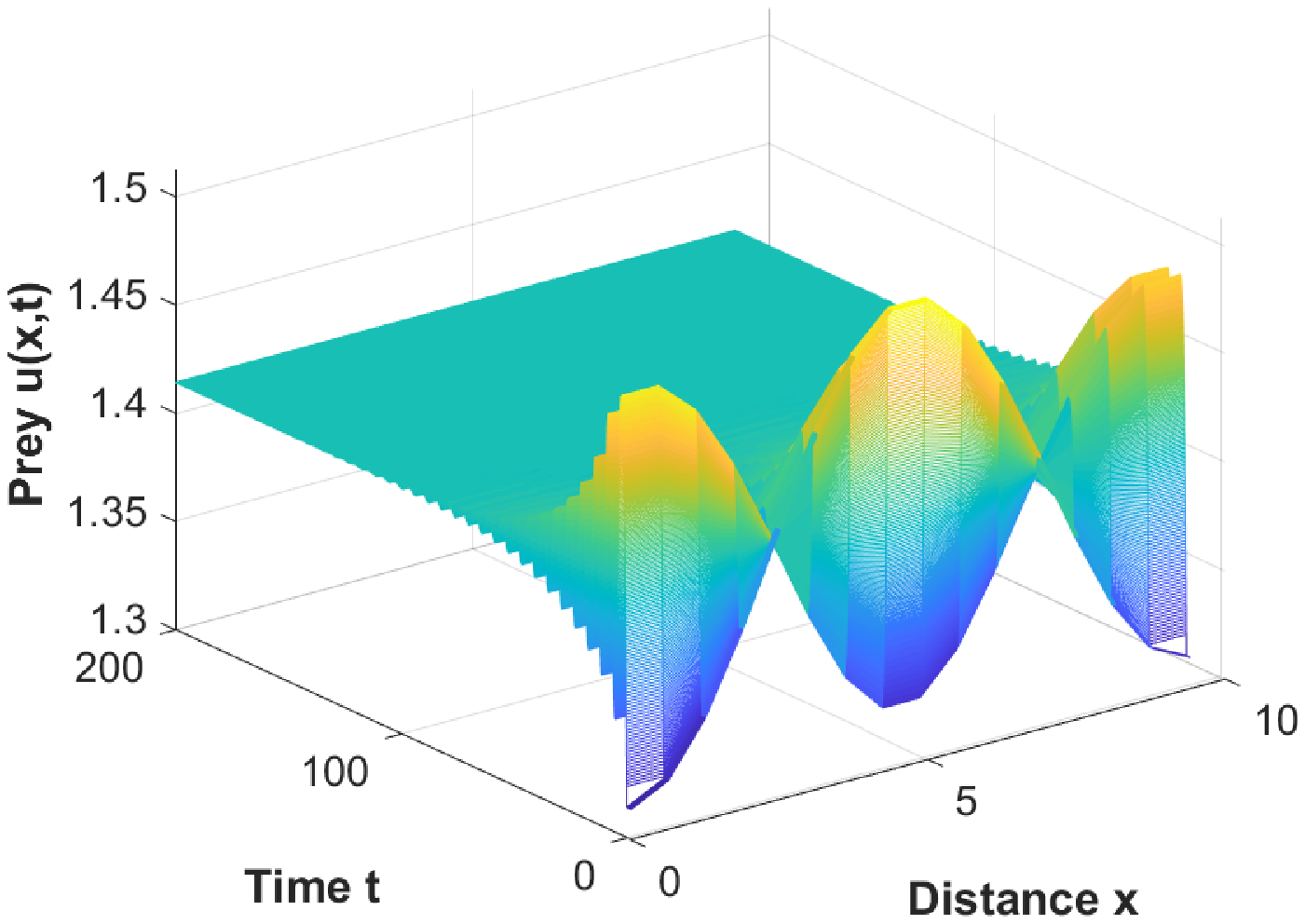}
\includegraphics[width=2.3in]{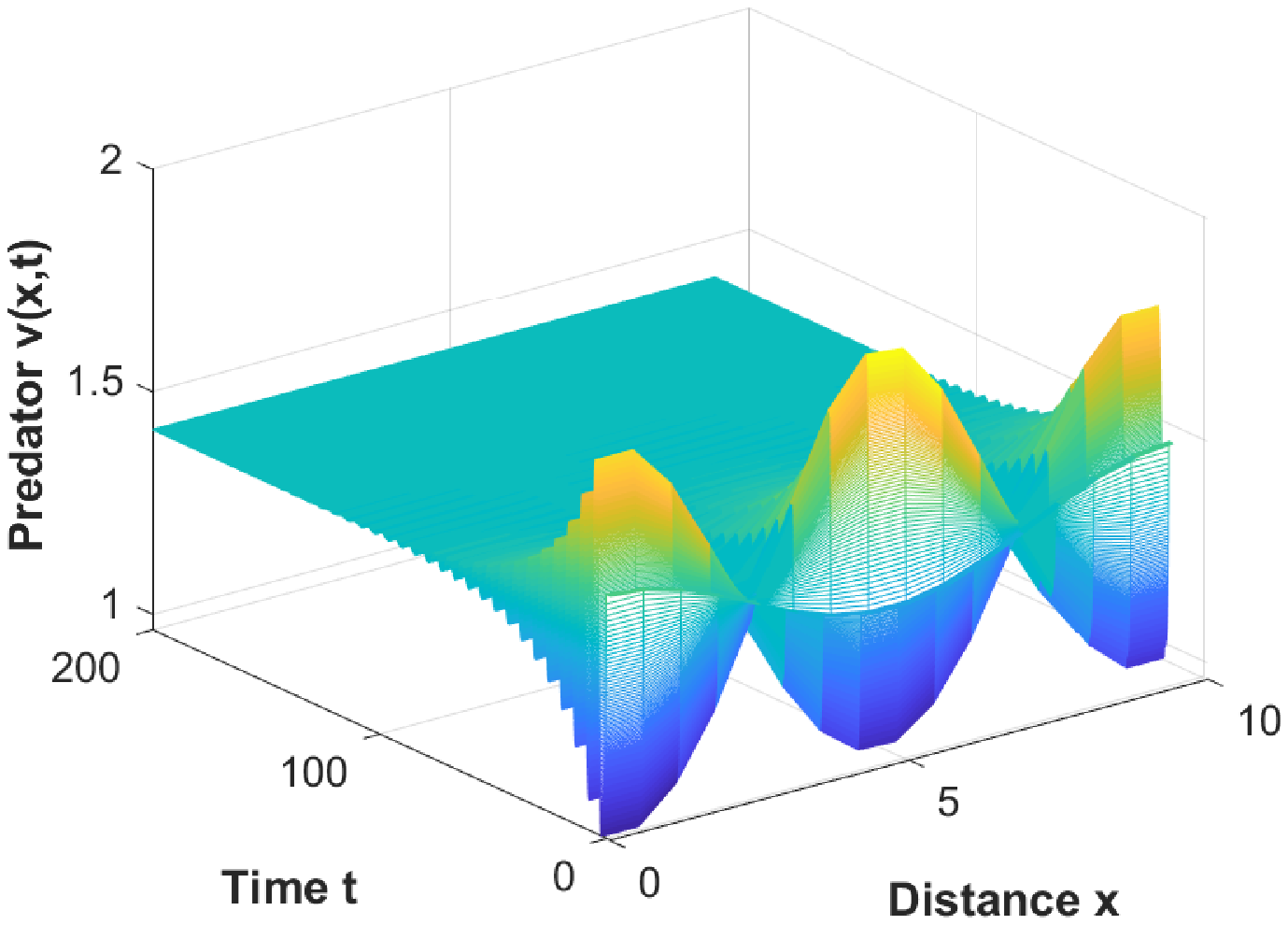} \\
\textbf{(a)} \hspace{5cm} \textbf{(b)} \\
\includegraphics[width=2.3in]{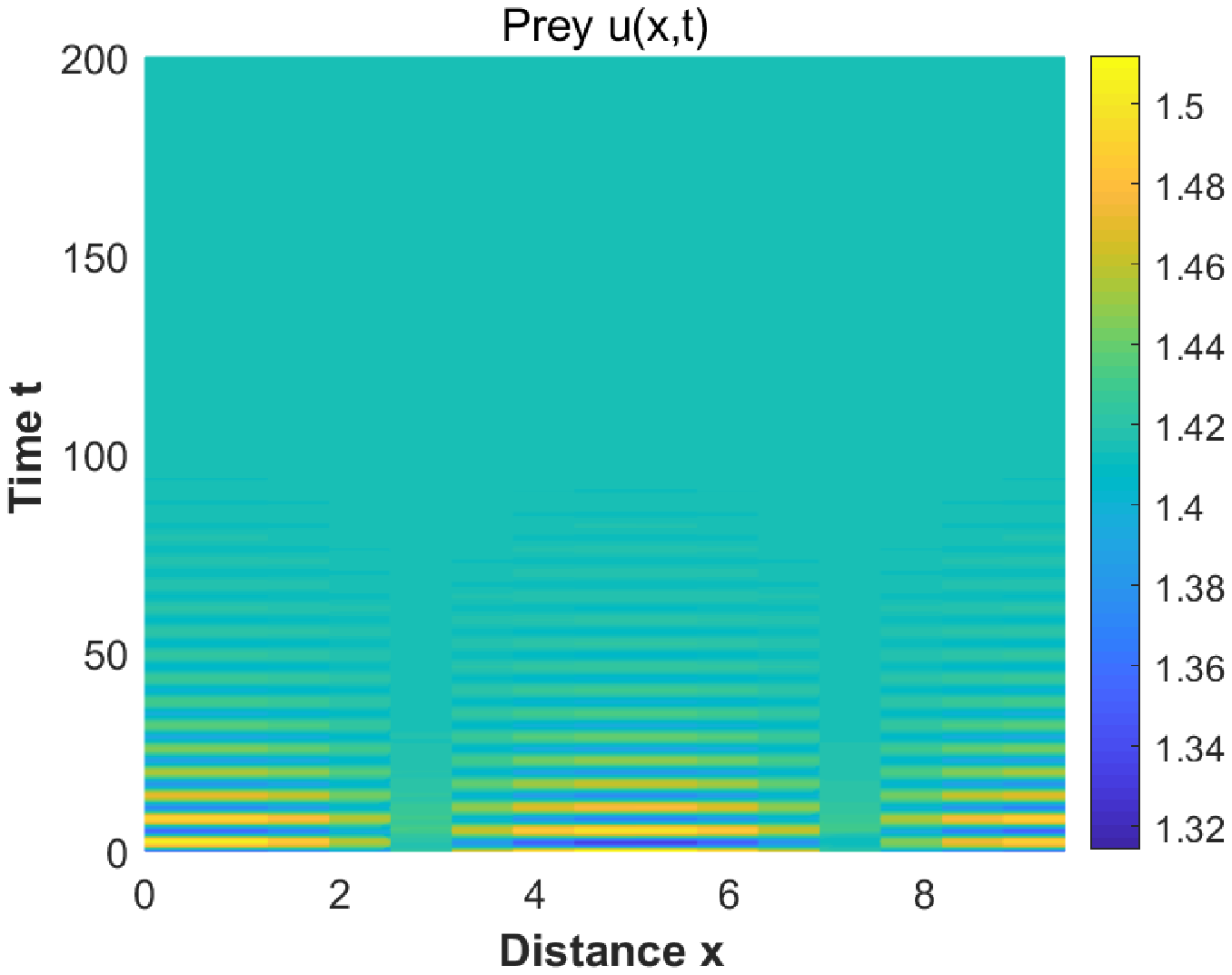}
\includegraphics[width=2.3in]{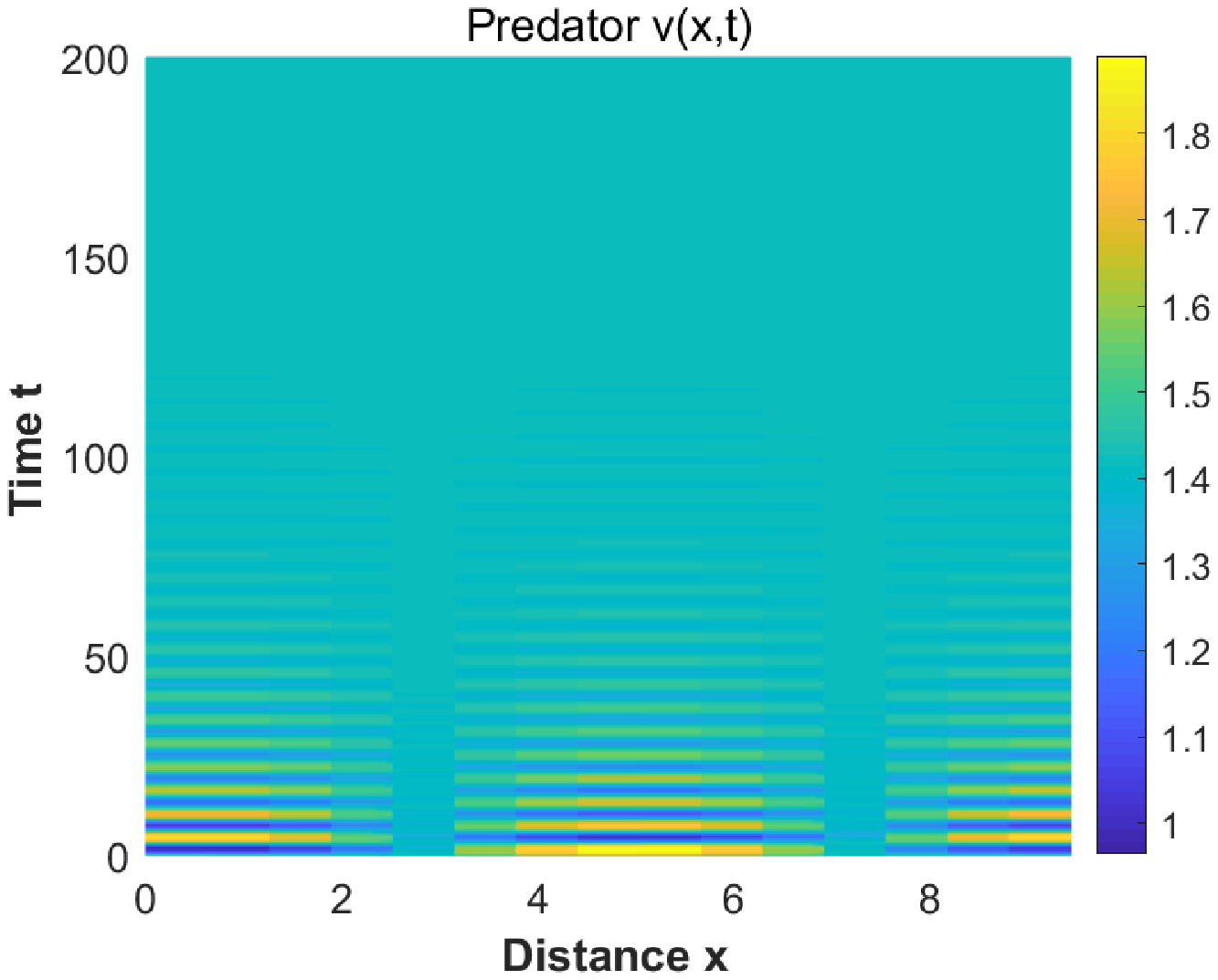} \\
\textbf{(c)} \hspace{5cm} \textbf{(d)} \\
\caption{For the parameters $\ell=3,~d_{11}=2,~d_{22}=3,~d_{21}=18,~\xi=0.06,~\beta=0.5,~m=0.5,~s=0.8$, when $\tau=2<\tau_{2,0}=3.5361$, the positive constant steady state $E_{*}\left(u_{*},v_{*}\right)=(1.4142,1.4142)$ is locally asymptotically stable. The initial values are $u_{0}(x)=1.4142-0.1\cos(2x/3)$ and $v_{0}(x)=1.4142+0.1\cos(2x/3)$.}
\label{fig:4}
\end{figure}

\begin{figure}[!htbp]
\centering
\includegraphics[width=2.3in]{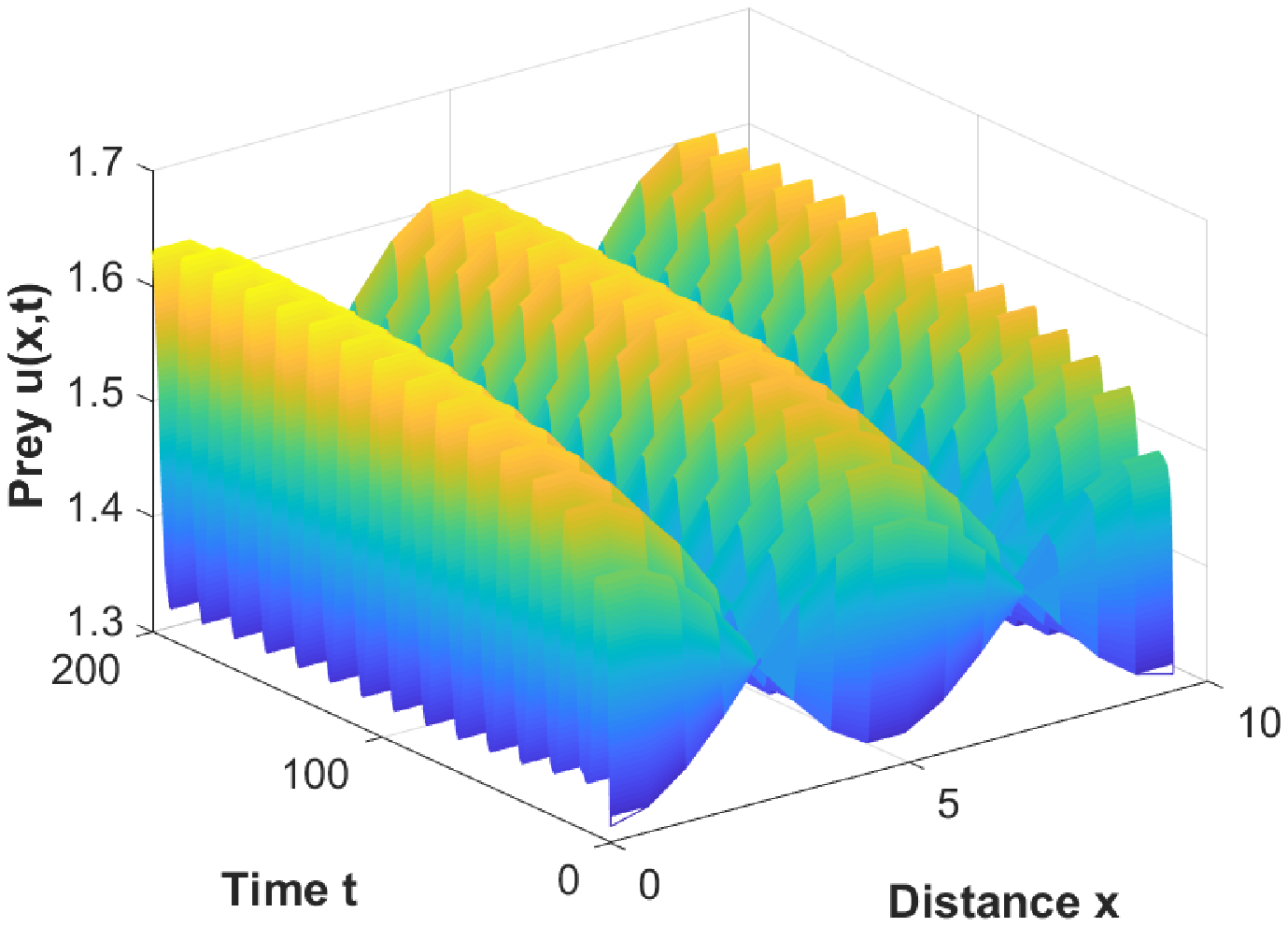}
\includegraphics[width=2.3in]{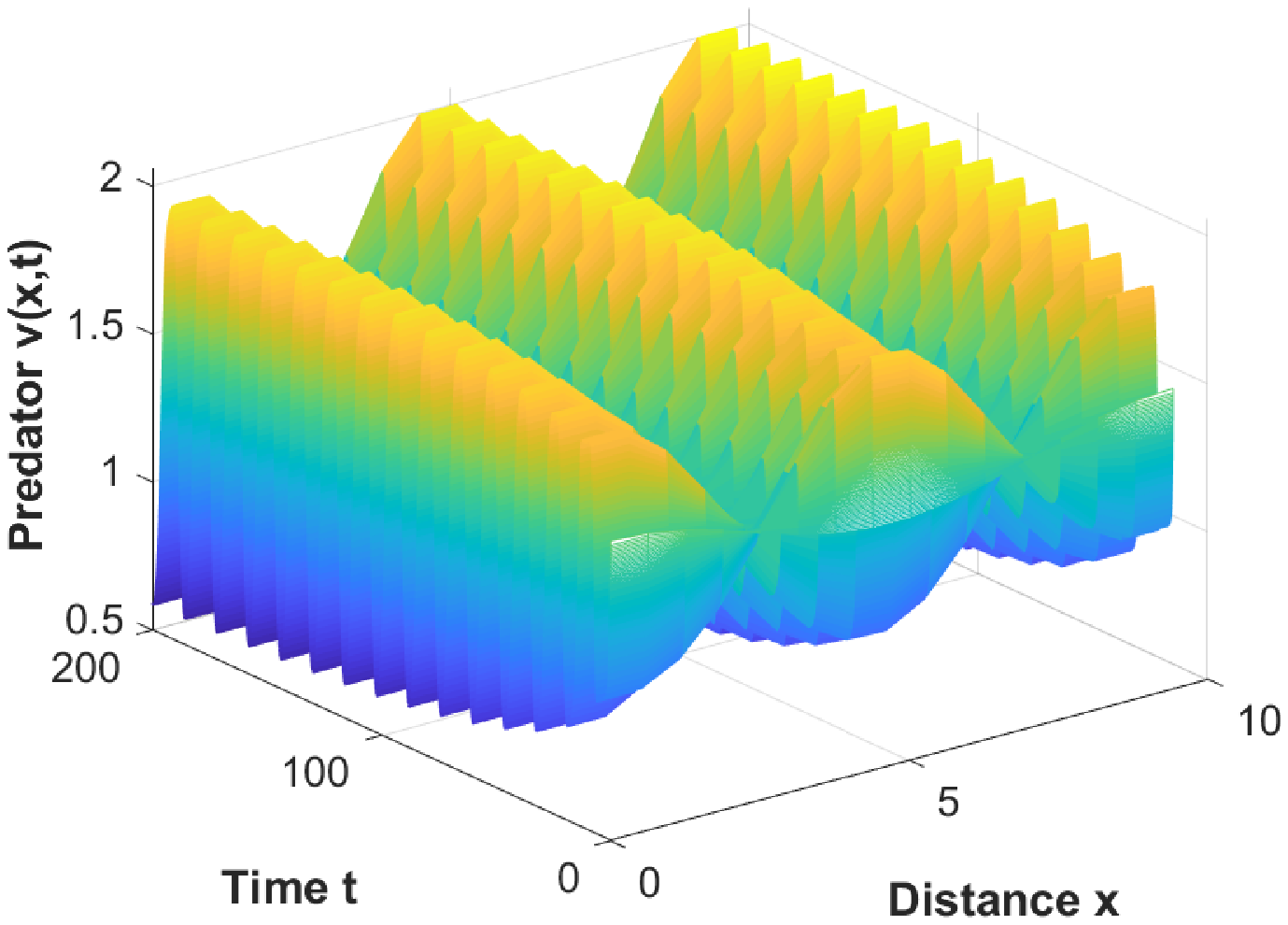} \\
\textbf{(a)} \hspace{5cm} \textbf{(b)} \\
\includegraphics[width=2.3in]{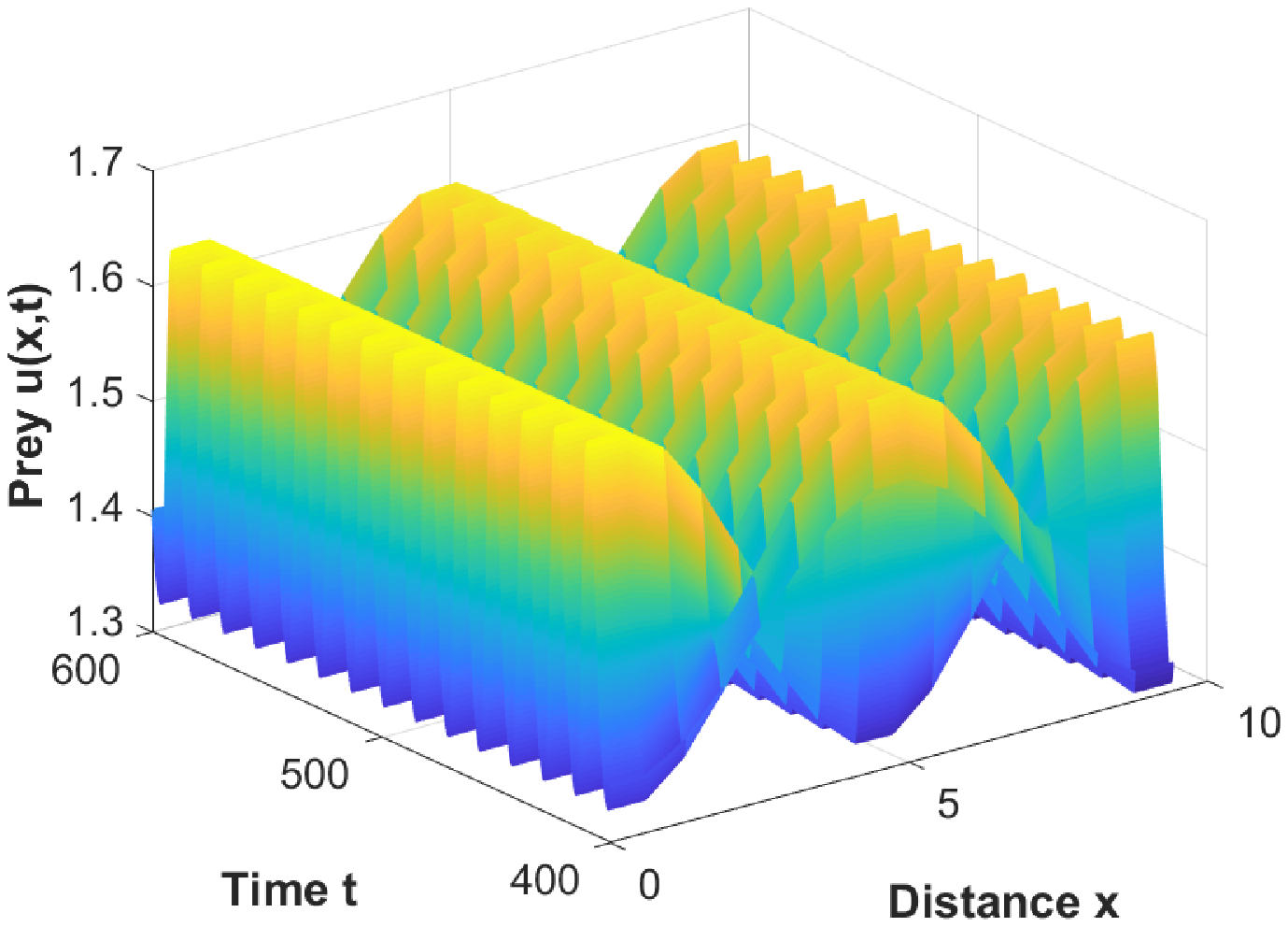}
\includegraphics[width=2.3in]{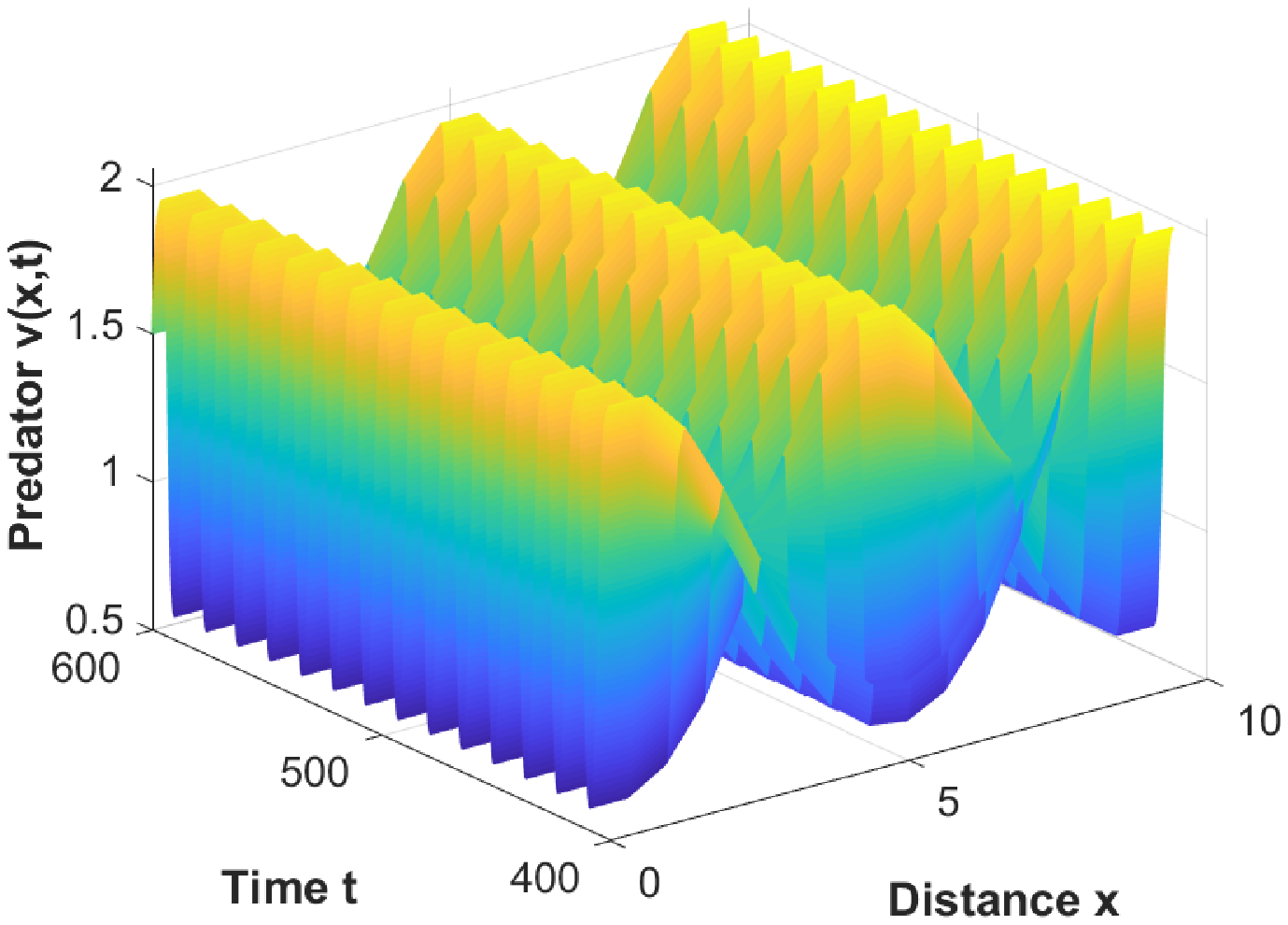} \\
\textbf{(c)} \hspace{5cm} \textbf{(d)} \\
\caption{For the parameters $\ell=3,~d_{11}=2,~d_{22}=3,~d_{21}=18,~\xi=0.06,~\beta=0.5,~m=0.5,~s=0.8$, when $\tau=6>\tau_{2,0}=3.5361$, there exists a stable spatially inhomogeneous periodic solution. (a) and (b) are the transient behaviours for $u(x,t)$ and $v(x,t)$, respectively, (c) and (d) are long-term behaviours for $u(x,t)$ and $v(x,t)$, respectively. The initial values are $u_{0}(x)=1.4142-0.1\cos(2x/3)$ and $v_{0}(x)=1.4142+0.1\cos(2x/3)$.}
\label{fig:5}
\end{figure}

\begin{figure}[!htbp]
\centering
\includegraphics[width=2.3in]{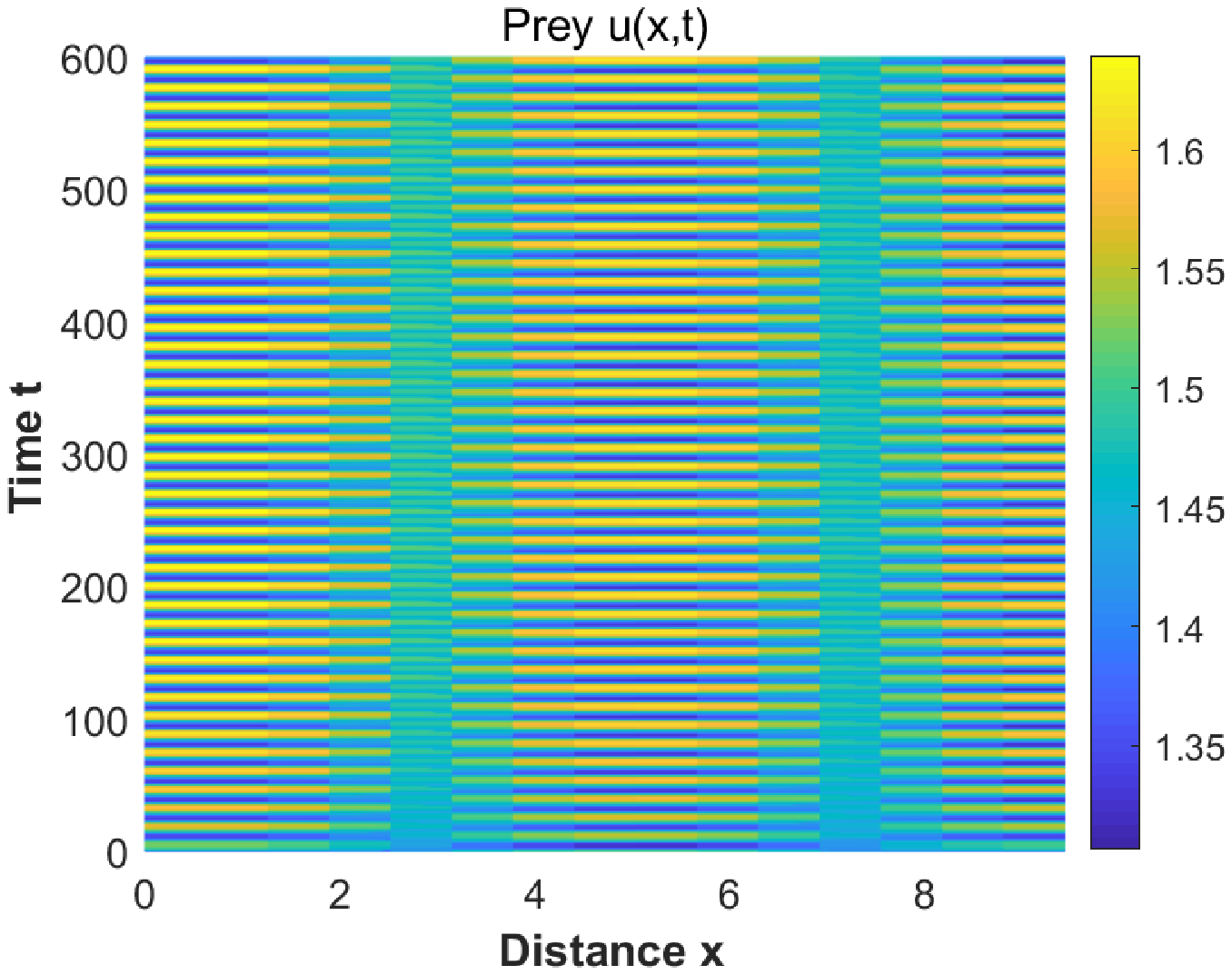}
\includegraphics[width=2.3in]{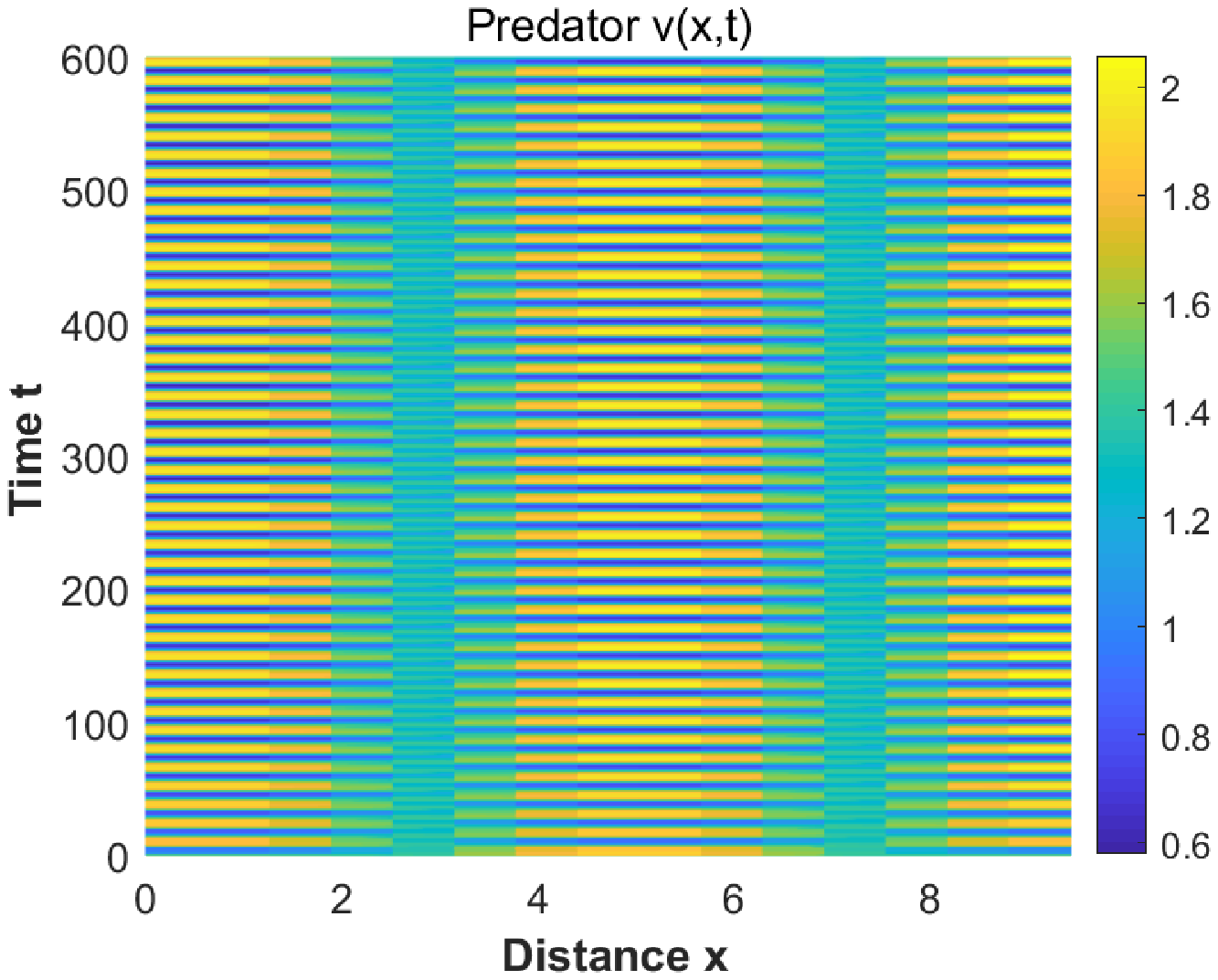} \\
\textbf{(a)} \hspace{5cm} \textbf{(b)} \\
\caption{For the parameters $\ell=3,~d_{11}=2,~d_{22}=3,~d_{21}=18,~\xi=0.06,~\beta=0.5,~m=0.5,~s=0.8$, when $\tau=6>\tau_{2,0}=3.5361$, there exists a stable spatially inhomogeneous periodic solution. The initial values are $u_{0}(x)=1.4142-0.1\cos(2x/3)$ and $v_{0}(x)=1.4142+0.1\cos(2x/3)$.}
\label{fig:6}
\end{figure}

If we set the parameters as follows
\begin{equation*}
\ell=3,~d_{11}=2,~d_{22}=3,~d_{21}=18,~\xi=0.06,~\beta=0.5,~m=0.5,~s=0.8,
\end{equation*}
then we can also easily obtain that
\begin{equation*}\begin{aligned}
&a_{11}=1-2\beta u_{*}-\frac{mu_{*}}{(1+u_{*})^{2}}=-0.5355<0,~d_{11}d_{22}-d_{21}\xi u_{*}v_{*}=3.84>0, \\
&\operatorname{Det}(A)=a_{11}a_{22}-a_{12}a_{21}=0.6627>0,~d_{11}a_{22}+d_{22}a_{11}-a_{21}\xi u_{*}-d_{21}v_{*}a_{12}=4.1814>0, \\
&(d_{11}a_{22}+d_{22}a_{11}-a_{21}\xi u_{*}-d_{21}v_{*}a_{12})^{2}-4(d_{11}d_{22}-d_{21}\xi u_{*}v_{*})\operatorname{Det}(A)=7.3041>0.
\end{aligned}\end{equation*}
Therefore, the conditions $(C_{0})$ and $(C_{2})$ are satisfied under the above parameters settings. In the following, we mainly verify the conclusion in Lemma 4.4 (ii). According to (4.3) and (4.4), we have $E_{*}\left(u_{*},v_{*}\right)=(1.4142,1.4142)$,
\begin{equation*}
a_{11}=-0.5355,~a_{12}=-0.2929,~a_{21}=0.8,~a_{22}=-0.8.
\end{equation*}

By combining with (4.13), (4.14), (4.15) and (4.16), we have $n_{1}=1.3164$, $n_{2}=2.8403$, and consider that $n\in\mathbb{N}$, we have $\omega_{c}=\omega_{2}=0.6870$ and $\tau_{c}=\tau_{2,0}=3.5361$. Moreover, by Lemma 4.4 (ii), we have the following proposition.

\begin{proposition}
For system (4.2) with the parameters $\ell=3,~d_{11}=2,~d_{22}=3,~d_{21}=18,~\xi=0.06,~\beta=0.5,~m=0.5,~s=0.8$, the positive constant steady state $E_{*}\left(u_{*},v_{*}\right)$ of system (4.2) is asymptotically stable for $0 \leq \tau<\tau_{2,0}=3.5361$ and unstable for $\tau>\tau_{2,0}=3.5361$. Furthermore, system (4.2) undergoes mode-2 Hopf bifurcation at $\tau=\tau_{2,0}=3.5361$.
\end{proposition}

For the parameters $\ell=3,~d_{11}=2,~d_{22}=3,~d_{21}=18,~\xi=0.06,~\beta=0.5,~m=0.5,~s=0.8$, according to Proposition 4.6, we know that system (4.2) undergoes a Hopf bifurcation at $\tau_{2,0}=3.5361$. Furthermore, the direction and stability of Hopf bifurcation can be determined by calculating $K_{1}K_{2}$ and $K_{2}$ using the procedures developed in Section 2. After a direct calculation using MATLAB software, we obtain
\begin{equation*}
K_{1}=0.0410>0,~K_{2}=-1.3669<0,~K_{1}K_{2}=-0.0561<0,
\end{equation*}
which implies that the Hopf bifurcation at $\tau_{2,0}=3.5361$ is supercritical and stable.

When $\tau=2<\tau_{2,0}=3.5361$, Fig.4 (a)-(d) illustrate the evolution of the solution of system (4.2) starting from the initial values $u_{0}(x)=1.4142-0.1\cos(2x/3)$ and $v_{0}(x)=1.4142+0.1\cos(2x/3)$, finally converging to the positive constant steady state $E_{*}\left(u_{*},v_{*}\right)$. Furthermore, when $\tau=6>\tau_{2,0}=3.5361$, Fig.5 (a)-(d) illustrate the existence of the spatially homogeneous periodic solution with the initial values $u_{0}(x)=1.4142-0.1\cos(2x/3)$ and $v_{0}(x)=1.4142+0.1\cos(2x/3)$. When $\tau=6>\tau_{2,0}=3.5361$, the space-time diagrams for the prey population $u(x,t)$ and the predator population $v(x,t)$ are given in Fig.6 (a) and (b), respectively.

\section{Conclusion and discussion}
\label{sec:5}

In this paper, the diffusive predator-prey system with spatial memory and predator-taxis is proposed, and we derive an algorithm for calculating the normal form of Hopf bifurcation in this system. As a real application, we consider the Holling-Tanner model with spatial memory and predator-taxis. By carrying out the stability and Hopf bifurcation analysis for the Holling-Tanner model with spatial memory and predator-taxis, the inhomogeneous spatial patterns, i.e., two stable spatially inhomogeneous periodic solutions are found. Furthermore, the supercritical and stable mode-1 and mode-2 Hopf bifurcation periodic solutions are found by using our newly developed algorithm. Numerical simulations verify the results of our theoretical analysis and give us a more intuitive display.

It is worth mentioning that in this paper, for our proposed diffusive predator-prey system with spatial memory and predator-taxis, the delay only occurs in the diffusion term, not in the reaction form. However, for the general predator-prey models, the gestation, hunting, migration and maturation delays, etc., often occur in the reaction term. By noticing this point, on the basis of system (1.4), the system
\begin{equation*}\left\{\begin{aligned}
&\frac{\partial u(x,t)}{\partial t}=d_{11}\Delta u(x,t)+\xi\left(u(x,t)v_{x}(x,t)\right)_{x}+f(u(x,t),v(x,t),u(x,t-\sigma),v(x,t-\sigma)), & x\in(0,\ell\pi),~t>0, \\
&\frac{\partial v(x,t)}{\partial t}=d_{22}\Delta v(x,t)-d_{21}\left(v(x,t)u_{x}(x,t-\tau)\right)_{x}+g(u(x,t),v(x,t),u(x,t-\sigma),v(x,t-\sigma)), & x\in(0,\ell\pi),~t>0,
\end{aligned}\right.\end{equation*}
which needs further research, where $\sigma>0$ is the delay occurs in the reaction term.

\section*{Acknowledgments}

The author is grateful to the anonymous referees for their useful suggestions which improve the contents of this article.

\section*{Disclosures and declarations}
\par\noindent This research did not involve human participants and animals.
\par\noindent Funding: This research did not receive any specific grant from funding agencies in the public, commercial, or not-for-profit sectors.
\par\noindent Conflicts of interest: The author declares that there is not conflict of interest, whether financial or non-financial.
\par\noindent Availability of data and material: This research didn't involve the private data, and the involving data and material are all available.
\par\noindent Code availability: The numerical simulations in this paper are carried by using the MATLAB software.
\par\noindent Authors' contributions: This manuscript is investigated and written by Yehu Lv.

\section*{Appendix A}
\setcounter{equation}{0}
\renewcommand\theequation{A.\arabic{equation}}

\begin{remark}
Assume that at $\tau=\tau_{c}$, system (4.2) has a pair of purely imaginary roots $\pm i\omega_{n_{c}}$ with $\omega_{n_{c}}>0$ for $n=n_{c} \in \mathbb{N}$ and all other eigenvalues have negative real parts. Let $\lambda(\tau)=\alpha_{1}(\tau) \pm i\alpha_{2}(\tau)$ be a pair of roots of system (4.2) near $\tau=\tau_{c}$ satisfying $\alpha_{1}(\tau_{c})=0$ and $\alpha_{2}(\tau_{c})=\omega_{n_{c}}$. In addition, the corresponding transversality condition holds.
\end{remark}

The normal form of Hopf bifurcation for the system (4.2) can be calculated by using our newly developed algorithm in Section 2. Here, we give the detail calculation procedures of $B_{1}, B_{21}, B_{22}, B_{23}$ steps by steps.
\begin{enumerate}[{\bf{Step 1:}}]
\item
\begin{equation*}
B_{1}=2\psi^{T}(0)\left(A\phi(0)-\frac{n_{c}^{2}}{\ell^{2}}\left(D_{1}\phi(0)+D_{2}\phi(-1)\right)\right)
\end{equation*}
with
\begin{equation*}
D_{1}=\left(\begin{array}{cc}
d_{11} & \xi u_{*} \\
0 & d_{22}
\end{array}\right),~D_{2}=\left(\begin{array}{cc}
0 & 0 \\
-d_{21}v_{*} & 0
\end{array}\right),~A=\left(\begin{array}{cc}
a_{11} & a_{12} \\
a_{21} & a_{22}
\end{array}\right).
\end{equation*}
Here,
\begin{equation*}
\phi=\left(\begin{array}{c}
1 \\
\frac{a_{11}-i\omega_{n_{c}}-d_{11}(n_{c}^{2}/\ell^{2})}{\xi u_{*}(n_{c}^{2}/\ell^{2})-a_{12}}
\end{array}\right),~\psi=\eta\left(\begin{array}{c}
1 \\
\frac{a_{12}-\xi u_{*}(n_{c}^{2}/\ell^{2})}{i\omega_{n_{c}}+d_{22}(n_{c}^{2}/\ell^{2})-a_{22}}
\end{array}\right)
\end{equation*}
with
\begin{equation*}
\eta=\frac{i\omega_{n_{c}}+\left(n_{c}/\ell\right)^{2}d_{22}-a_{22}}{2i\omega_{n_{c}}+\left(n_{c}/\ell\right)^{2}d_{11}-a_{11}+\left(n_{c}/\ell\right)^{2}d_{22}-a_{22}+\tau_{c} a_{12}d_{21}v_{*}\left(n_{c}/\ell\right)^{2}e^{-i\omega_{c}}}.
\end{equation*}
\end{enumerate}

\begin{enumerate}[{\bf{Step 2:}}]
\item
\begin{equation*}
B_{21}=\frac{3}{2\ell\pi}\psi^{T}A_{21}
\end{equation*}
with
\begin{equation*}\begin{aligned}
A_{21}&=3f_{30}\phi_{1}^{2}(0)\overline{\phi_{1}}(0)+3f_{03}\phi_{2}^{2}(0)\overline{\phi_{2}}(0)+3f_{21}(\phi_{1}^{2}(0)\overline{\phi_{2}}(0)+2\phi_{1}(0)\overline{\phi_{1}}(0)\phi_{2}(0)) \\
&+3f_{12}(2\phi_{1}(0)\phi_{2}(0)\overline{\phi_{2}}(0)+\overline{\phi_{1}}(0)\phi_{2}^{2}(0)).
\end{aligned}\end{equation*}
Here,
\begin{equation*}\begin{aligned}
f^{(1)}_{30}&=-6\tau_{c}m(1+u_{*})^{-4}v_{*},~f^{(2)}_{30}=6\tau_{c}su_{*}^{-4}v_{*}^{2}, \\
f^{(1)}_{21}&=2\tau_{c}m(1+u_{*})^{-3},~f^{(2)}_{21}=-4\tau_{c}su_{*}^{-3}v_{*}, \\
f^{(1)}_{12}&=0,~f^{(2)}_{12}=2\tau_{c}su_{*}^{-2}, \\
f^{(1)}_{03}&=0,~f^{(2)}_{03}=0.
\end{aligned}\end{equation*}
\end{enumerate}

\begin{enumerate}[{\bf{Step 3:}}]
\item
\begin{equation*}\begin{aligned}
B_{22}&=\frac{1}{\sqrt{\ell\pi}}\psi^{T}\left(\mathcal{S}_{2}\left(\phi(\theta),h_{0,11}(\theta)\right)+\mathcal{S}_{2}\left(\overline{\phi}(\theta),h_{0,20}(\theta)\right)\right) \\
&+\frac{1}{\sqrt{2\ell\pi}}\psi^{T}\left(\mathcal{S}_{2}\left(\phi(\theta),h_{2n_{c},11}(\theta)\right)+\mathcal{S}_{2}\left(\overline{\phi}(\theta),h_{2n_{c},20}(\theta)\right)\right)
\end{aligned}\end{equation*}
with
\begin{equation*}\begin{aligned}
\mathcal{S}_{2}\left(\phi(\theta),h_{0,11}(\theta)\right)&=2f_{20}\phi_{1}(0)h^{(1)}_{0,11}(0)+2f_{02}\phi_{2}(0)h^{(2)}_{0,11}(0) \\
&+2f_{11}\left(\phi_{1}(0)h^{(2)}_{0,11}(0)+\phi_{2}(0)h^{(1)}_{0,11}(0)\right), \\
\mathcal{S}_{2}\left(\overline{\phi}(\theta), h_{0,20}(\theta)\right)&=2f_{20}\overline{\phi}_{1}(0)h^{(1)}_{0,20}(0)+2f_{02}\overline{\phi}_{2}(0)h^{(2)}_{0,20}(0) \\
&+2f_{11}\left(\overline{\phi}_{1}(0)h^{(2)}_{0,20}(0)+\overline{\phi}_{2}(0)h^{(1)}_{0,20}(0)\right), \\
\mathcal{S}_{2}\left(\phi(\theta),h_{2n_{c},11}(\theta)\right)&=2f_{20}\phi_{1}(0)h^{(1)}_{2n_{c},11}(0)+2f_{02}\phi_{2}(0)h^{(2)}_{2n_{c},11}(0) \\
&+2f_{11}\left(\phi_{1}(0)h^{(2)}_{2n_{c},11}(0)+\phi_{2}(0)h^{(1)}_{2n_{c},11}(0)\right), \\
\mathcal{S}_{2}\left(\overline{\phi}(\theta),h_{2n_{c},20}(\theta)\right)&=2f_{20}\overline{\phi}_{1}(0)h^{(1)}_{2n_{c},20}(0)+2f_{02}\overline{\phi}_{2}(0)h^{(2)}_{2n_{c},20}(0) \\
&+2f_{11}\left(\overline{\phi}_{1}(0)h^{(2)}_{2n_{c},20}(0)+\overline{\phi}_{2}(0)h^{(1)}_{2n_{c},20}(0)\right).
\end{aligned}\end{equation*}
Here,
\begin{equation*}\begin{aligned}
f^{(1)}_{20}&=-2\tau_{c}\beta+2\tau_{c}m(1+u_{*})^{-3}v_{*},~f^{(2)}_{20}=-2\tau_{c}su_{*}^{-3}v_{*}^{2}, \\
f^{(1)}_{11}&=-\tau_{c}m(1+u_{*})^{-2},~f^{(2)}_{11}=2\tau_{c}su_{*}^{-2}v_{*}, \\
f^{(1)}_{02}&=0,~f^{(2)}_{02}=-2\tau_{c}su_{*}^{-1}.
\end{aligned}\end{equation*}
Furthermore, we have
\begin{equation*}
\left\{\begin{aligned}
h_{0,20}(\theta)&=\frac{1}{\sqrt{\ell\pi}}\left(\widetilde{\mathcal{M}}_{0}\left(2i\omega_{c}\right)\right)^{-1}A_{20}e^{2i\omega_{c}\theta}, \\
h_{0,11}(\theta)&=\frac{1}{\sqrt{\ell\pi}}\left(\widetilde{\mathcal{M}}_{0}(0)\right)^{-1}A_{11}
\end{aligned}\right.
\end{equation*}
and
\begin{equation*}
\left\{\begin{aligned}
h_{2n_{c},20}(\theta)&=\frac{1}{\sqrt{2\ell\pi}}\left(\widetilde{\mathcal{M}}_{2n_{c}}\left(2i\omega_{c}\right)\right)^{-1}\widetilde{A}_{20}e^{2i\omega_{c}\theta}, \\
h_{2n_{c},11}(\theta)&=\frac{1}{\sqrt{2\ell\pi}}\left(\widetilde{\mathcal{M}}_{2n_{c}}(0)\right)^{-1}\widetilde{A}_{11}
\end{aligned}\right.
\end{equation*}
with
\begin{equation*}
\widetilde{\mathcal{M}}_{n}(\lambda)=\lambda I_{2}+\tau_{c}(n/\ell)^{2}D_{1}+\tau_{c}(n/\ell)^{2}e^{-\lambda}D_{2}-\tau_{c}A.
\end{equation*}
Here,
\begin{equation*}\begin{aligned}
A_{20}&=f_{20}\phi_{1}^{2}(0)+f_{02}\phi_{2}^{2}(0)+2f_{11}\phi_{1}(0)\phi_{2}(0), \\
A_{11}&=2f_{20}\phi_{1}(0)\overline{\phi_{1}}(0)+2f_{02}\phi_{2}(0)\overline{\phi_{2}}(0)+2f_{11}(\phi_{1}(0)\overline{\phi_{2}}(0)+\overline{\phi_{1}}(0)\phi_{2}(0))
\end{aligned}\end{equation*}
and
\begin{equation*}
\left\{\begin{aligned}
\widetilde{A}_{20}&=A_{20}-2\left(n_{c}/\ell\right)^{2}A_{20}^{d}, \\
\widetilde{A}_{11}&=A_{11}-2\left(n_{c}/\ell\right)^{2}A_{11}^{d}
\end{aligned}\right.
\end{equation*}
with
\begin{equation*}
\left\{\begin{aligned}
A_{20}^{d}&=\left(\begin{array}{c}
2\xi \tau_{c}\phi_{1}(0)\phi_{2}(0) \\
-2d_{21}\tau_{c}\phi_{1}(-1)\phi_{2}(0)
\end{array}\right)=\overline{A_{02}^{d}}, \\
A_{11}^{d}&=\left(\begin{array}{c}
4\xi \tau_{c}\operatorname{Re}\left\{\phi_{1}(0)\overline{\phi_{2}}(0)\right\} \\
-4d_{21}\tau_{c}\operatorname{Re}\left\{\phi_{1}(-1)\overline{\phi_{2}}(0)\right\}
\end{array}\right).\end{aligned}\right.
\end{equation*}
\end{enumerate}

\begin{enumerate}[{\bf{Step 4:}}]
\item
\begin{equation*}\begin{aligned}
B_{23}=&-\frac{1}{\sqrt{\ell\pi}}\left(n_{c}/\ell\right)^{2} \psi^{T}\left(\mathcal{S}_{2}^{(d,1)}\left(\phi(\theta),h_{0,11}(\theta)\right)+\mathcal{S}_{2}^{(d,1)}\left(\overline{\phi}(\theta), h_{0,20}(\theta)\right)\right) \\
&+\frac{1}{\sqrt{2\ell\pi}}\psi^{T}\sum_{j=1,2,3} b_{2n_{c}}^{(j)}\left(\mathcal{S}_{2}^{(d,j)}\left(\phi(\theta),h_{2n_{c},11}(\theta)\right)+\mathcal{S}_{2}^{(d,j)}\left(\overline{\phi}(\theta),h_{2n_{c},20}(\theta)\right)\right)
\end{aligned}\end{equation*}
with
\begin{equation*}
b_{2n_{c}}^{(1)}=-\frac{n_{c}^{2}}{\ell^{2}},~b_{2n_{c}}^{(2)}=\frac{2n_{c}^{2}}{\ell^{2}},~b_{2n_{c}}^{(3)}=-\frac{(2n_{c})^{2}}{\ell^{2}}
\end{equation*}
and
\begin{equation*}
\left\{\begin{aligned}
&\mathcal{S}_{2}^{(d,1)}\left(\phi(\theta),h_{0,11}(\theta)\right)=2\left(\begin{array}{c}
\xi\tau_{c}\phi_{2}(0)h_{0,11}^{(1)}(0) \\
-d_{21}\tau_{c}\phi_{1}(-1)h_{0,11}^{(2)}(0)
\end{array}\right), \\
&\mathcal{S}_{2}^{(d,1)}\left(\overline{\phi}(\theta),h_{0,20}(\theta)\right)=2\left(\begin{array}{c}
\xi\tau_{c}\overline{\phi}_{2}(0)h_{0,20}^{(1)}(0) \\
-d_{21}\tau_{c}\overline{\phi}_{1}(-1)h_{0,20}^{(2)}(0)
\end{array}\right), \\
&\mathcal{S}_{2}^{(d,1)}\left(\phi(\theta),h_{2n_{c},11}(\theta)\right)=2\left(\begin{array}{c}
\xi\tau_{c}\phi_{2}(0)h_{2n_{c},11}^{(1)}(0) \\
-d_{21}\tau_{c}\phi_{1}(-1)h_{2n_{c},11}^{(2)}(0)
\end{array}\right), \\
&\mathcal{S}_{2}^{(d,2)}(\phi(\theta),h_{2n_{c},11}(\theta))=2\left(\begin{array}{c}
\xi\tau_{c}(\phi_{2}(0)h_{2n_{c},11}^{(1)}(0)+\phi_{1}(0)h_{2n_{c},11}^{(2)}(0)) \\
-d_{21}\tau_{c}(\phi_{2}(0)h_{2n_{c},11}^{(1)}(-1)+\phi_{1}(-1)h_{2n_{c},11}^{(2)}(0))
\end{array}\right), \\
&\mathcal{S}_{2}^{(d,3)}(\phi(\theta),h_{2n_{c},11}(\theta))=2\left(\begin{array}{c}
\xi \tau_{c}\phi_{1}(0)h_{2n_{c},11}^{(2)}(0) \\
-d_{21}\tau_{c}\phi_{2}(0)h_{2n_{c},11}^{(1)}(-1)
\end{array}\right), \\
&\mathcal{S}_{2}^{(d,1)}\left(\overline{\phi}(\theta),h_{2n_{c},20}(\theta)\right)=2\left(\begin{array}{c}
\xi\tau_{c}\overline{\phi}_{2}(0)h_{2n_{c},20}^{(1)}(0) \\
-d_{21}\tau_{c}\overline{\phi}_{1}(-1)h_{2n_{c},20}^{(2)}(0)
\end{array}\right), \\
&\mathcal{S}_{2}^{(d,2)}(\overline{\phi}(\theta),h_{2n_{c},20}(\theta))=2\left(\begin{array}{c}
\xi\tau_{c}(\overline{\phi}_{2}(0)h_{2n_{c},20}^{(1)}(0)+\overline{\phi}_{1}(0)h_{2n_{c},20}^{(2)}(0)) \\
-d_{21}\tau_{c}(\overline{\phi}_{2}(0)h_{2n_{c},20}^{(1)}(-1)+\overline{\phi}_{1}(-1)h_{2n_{c},20}^{(2)}(0))
\end{array}\right), \\
&\mathcal{S}_{2}^{(d,3)}(\overline{\phi}(\theta),h_{2n_{c},20}(\theta))=2\left(\begin{array}{c}
\xi \tau_{c}\overline{\phi}_{1}(0)h_{2n_{c},20}^{(2)}(0) \\
-d_{21}\tau_{c}\overline{\phi}_{2}(0)h_{2n_{c},20}^{(1)}(-1)
\end{array}\right).
\end{aligned}\right.
\end{equation*}
\end{enumerate}



\begin{thebibliography}{99}

\bibitem{lv1} J. Crank, The Mathematics of Diffusion, Oxford University Press, Oxford, 1979.

\bibitem{lv2} J.D. Murray, Mathematical Biology II: Spatial Models and Biomedical Applications, 3rd ed., Springer-Verlag, New York, 2003.

\bibitem{lv3} A. Okubo, S.A. Levin, Diffusion and Ecological Problems: Modern Perspectives, 2nd ed., Springer-Verlag, New York, 2001.

\bibitem{lv4} Y.H. Lv, Z.H. Liu, Turing-Hopf bifurcation analysis and normal form of a diffusive Brusselator model with gene expression time delay, \textit{Chaos, Solitons and Fractals}. \textbf{152}, 111478, 2021.

\bibitem{lv5} Y.H. Du, S.B. Hsu, A diffusive predator-prey model in heterogeneous environment, \textit{Journal of Differential Equations}. \textbf{203}(2), 331-364, 2004.

\bibitem{lv6} M.L. Rosenzweig, R.H. MacArthur, Graphical representation and stability conditions of predator-prey interactions, \textit{The American Naturalist}. \textbf{97}(895), 209-223, 1963.

\bibitem{lv7} Y.H. Du, J.P. Shi, A diffusive predator-prey model with a protection zone, \textit{Journal of Differential Equations}. \textbf{229}(1), 63-91, 2006.

\bibitem{lv8} S. Djilali, S. Bentout, Spatiotemporal patterns in a diffusive predator-prey model with prey social behavior, \textit{Acta Applicandae Mathematicae}. \textbf{169}(1), 125-143, 2020.

\bibitem{lv9} F. Souna, A. Lakmeche, S. Djilali, Spatiotemporal patterns in a diffusive predator-prey model with protection zone and predator harvesting, \textit{Chaos, Solitons and Fractals}. \textbf{140}, 110180, 2020.

\bibitem{lv10} J.P. Shi, C.C. Wang, H. Wang, et al., Diffusive spatial movement with memory, \textit{Journal of Dynamics and Differential Equations}. \textbf{32}(2), 979-1002, 2020.

\bibitem{lv11} J.P. Shi, C.C. Wang, H. Wang, Diffusive spatial movement with memory and maturation delays, \textit{Nonlinearity}. \textbf{32}(9), 3188, 2019.

\bibitem{lv12} Y.L. Song, Y.H. Peng, T.H. Zhang, The spatially inhomogeneous Hopf bifurcation induced by memory delay in a memory-based diffusion system, \textit{Journal of Differential Equations}. \textbf{300}(5), 597-624, 2021.

\bibitem{lv13} J.F. Wang, S.N. Wu, J.Y. Shi, Pattern formation in diffusive predator-prey systems with predator-taxis and prey-taxis, \textit{Discrete and Continuous Dynamical Systems-B}. \textbf{26}(3), 1273-1289, 2021.

\bibitem{lv14} P. Kareiva, G.T. Odell, Swarms of predators exhibit "preytaxis" if individual predators use area-restricted search, \textit{The American Naturalist}. \textbf{130}(2), 233-270, 1987.

\bibitem{lv15} A.M. Turner, G.G. Mittelbach, Predator avoidance and community structure: interactions among piscivores, planktivores, and plankton, \textit{Ecology}. \textbf{71}(6), 2241-2254, 1990.

\bibitem{lv16} T.M. Zaret, J.S. Suffern, Vertical migration in zooplankton as a predator avoidance mechanism, \textit{Limnology and Oceanography}. \textbf{21}(6), 804-813, 1976.

\bibitem{lv17} A. Chakraborty, M. Singh, D. Lucy, et al., Predator-prey model with prey-taxis and diffusion, \textit{Mathematical and Computer Modelling}. \textbf{46}(3-4), 482-498, 2007.

\bibitem{lv18} S.N. Wu, J.P. Shi, B.Y. Wu, Global existence of solutions and uniform persistence of a diffusive predator-prey model with prey-taxis, \textit{Journal of Differential Equations}. \textbf{260}(7), 5847-5874, 2016.

\bibitem{lv19} B.E. Ainseba, M. Bendahmane, A. Noussair, A reaction-diffusion system modeling predator-prey with prey-taxis, \textit{Nonlinear Analysis: Real World Applications}. \textbf{9}(5), 2086-2105, 2008.

\bibitem{lv20} J.P. Wang, M.X. Wang, The diffusive Beddington-DeAngelis predator-prey model with nonlinear prey-taxis and free boundary, \textit{Mathematical Methods in the Applied Sciences}. \textbf{41}(16), 6741-6762, 2018.

\bibitem{lv21} H.H. Qiu, S.J. Guo, S.Z. Li, Stability and bifurcation in a predator-prey system with prey-taxis, \textit{International Journal of Bifurcation and Chaos}. \textbf{30}(02), 2050022, 2020.

\bibitem{lv22} J.I. Tello, D. Wrzosek, Predator-prey model with diffusion and indirect prey-taxis, \textit{Mathematical Models and Methods in Applied Sciences}. \textbf{26}(11), 2129-2162, 2016.

\bibitem{lv23} Y.V. Tyutyunov, L.I. Titova, I.N. Senina, Prey-taxis destabilizes homogeneous stationary state in spatial Gause-Kolmogorov-type model for predator-prey system, \textit{Ecological Complexity}. \textbf{31}, 170-180, 2017.

\bibitem{lv24} J.P. Wang, M.X. Wang, The dynamics of a predator-prey model with diffusion and indirect prey-taxis, \textit{Journal of Dynamics and Differential Equations}. \textbf{32}(3), 1291-1310, 2020.

\bibitem{lv25} S.N. Wu, J.F. Wang, J.P. Shi, Dynamics and pattern formation of a diffusive predator-prey model with predator-taxis, \textit{Mathematical Models and Methods in Applied Sciences}. \textbf{28}(11), 2275-2312, 2018.

\bibitem{lv26} I. Ahn, C. Yoon, Global solvability of prey-predator models with indirect predator-taxis, \textit{Zeitschrift f$\ddot{u}$r angewandte Mathematik und Physik}. \textbf{72}(1), 1-20, 2021.

\bibitem{lv27} T. Faria, Normal forms and Hopf bifurcation for partial differential equations with delays, \textit{Transactions of the American Mathematical Society}. \textbf{352}(5), 2217-2238, 2000.

\bibitem{lv28} T. Faria, L.T. Magalh$\tilde{a}$es, Normal forms for retarded functional differential equations with parameters and applications to Hopf bifurcation, \textit{Journal of Differential Equations}. \textbf{122}(2), 181-200, 1995.

\bibitem{lv29} S.N. Chow, J.K. Hale, Methods of Bifurcation Theory, Springer-Verlag, New York, 1982.

\bibitem{lv30} S.B. Hsu, T.W. Hwang, Uniqueness of limit cycles for a predator-prey system of Holling and Leslie type, \textit{Canadian Applied Mathematics Quarterly}. \textbf{6}(2), 91-117, 1998.

\bibitem{lv31} S.S. Chen, J.S. Shi, Global stability in a diffusive Holling-Tanner predator-prey model, \textit{Applied Mathematics Letters}. \textbf{25}(3), 614-618, 2012.

\bibitem{lv32} R. Peng, M.X. Wang, Global stability of the equilibrium of a diffusive Holling-Tanner prey-predator model, \textit{Applied Mathematics Letters}. \textbf{20}(6), 664-670, 2007.

\bibitem{lv33} X. Li, W.H. Jiang, J.P., Shi, Hopf bifurcation and Turing instability in the reaction-diffusion Holling-Tanner predator-prey model. \textit{IMA Journal of Applied Mathematics}. \textbf{78}(2), 287-306, 2013.

\end{thebibliography}
\end{document}